\documentclass[12 pt, amsfonts, amssymb,amsthm,color]{article} 
\usepackage{enumitem}
\usepackage{setspace}
\usepackage{pdflscape}

\usepackage{hyperref}
\usepackage{amsmath,amssymb,amsfonts,latexsym,float,graphics,epsfig,xcolor}
\usepackage[all]{xy}
\newenvironment{axioms}
 {\enumerate[label=\textbf{A\arabic*.}, ref=A\arabic*]}
 {\endenumerate} 
\usepackage[top=2cm, bottom=2cm, left=2cm, right=2cm]{geometry} 
\newtheorem{theorem}{Theorem}[section]  
\newtheorem{proof}[theorem]{Proof}

\newtheorem{corollary}[theorem]{Corollary}

\begin{document}

\title{On Globalization Problem of \\ Multi-Hamiltonian Formalisms}

\maketitle

\begin{center} 
 Begüm Ateşli$^{\ast,\dagger}$\footnote{e-mails: 
\href{mailto:b.atesli@gtu.edu.tr}{b.atesli@gtu.edu.tr}, \href{mailto:begumatesli@itu.edu.tr}{begumatesli@itu.edu.tr} corresponding author,} and Aybike Çatal-Özer $^{\ast}$\footnote{e-mail: 
\href{ozerayb@itu.edu.tr}{ozerayb@itu.edu.tr},}
  \\

\bigskip

$^{\ast}$Department of Mathematics Engineering \\ İstanbul Technical University,  34467 Maslak, İstanbul

\bigskip

$^\dagger$Department of Mathematics, \\ Gebze Technical University, 41400 Gebze,
Kocaeli, Turkey.
\bigskip

\end{center}

\begin{abstract}
The globalization problem arises when local tensor fields possess a given property (such as being symplectic or Poisson) but cannot be consistently extended to a global object due to incompatibilities on chart overlaps. A notable instance occurs in locally conformal analysis, where local representatives coincide only up to conformal factors. The locally conformal approach not only enables the definition of novel and rich geometric structures but also provides Hamiltonian formulations for irreversible systems, yielding physically meaningful dynamical consequences. While extensively studied for symplectic, cosymplectic, and Poisson geometries, its systematic extension to multi-Hamiltonian settings remains largely unexplored. 

In this work, we investigate locally conformally Nambu--Poisson and locally conformally generalized Poisson manifolds, showing that these structures naturally induce Nambu--Jacobi and generalized Jacobi manifolds, respectively. From a dynamical point of view, we construct Hamiltonian-type evolution equations, and for locally conformal $3$-Nambu structures, we introduce locally conformal bi-Hamiltonian systems. The resulting dynamics are particularly suitable for modeling irreversible multi-Hamiltonian processes, as they generally do not preserve the system's energy. 

Collectively, this work provides a unified framework for understanding both the geometric structures and Hamiltonian dynamics of classical, Nambu--Poisson, and generalized Poisson manifolds within a locally conformal context.

\end{abstract} 


\tableofcontents

\onehalfspacing

\section{Introduction} 

A global tensorial field (such as a symplectic form, a cosymplectic structure, or more generally a Poisson bivector field) on a manifold necessarily induces local representatives in every coordinate chart that preserve the same structural property \cite{abraham1978foundations,Mars88,LiMa,vaisman2012lectures,weinstein1983local,Weinstein98}.  The converse, however, does not hold in general: the existence of local tensor fields enjoying a given property on each chart does not guarantee that they transform appropriately on overlaps so as to glue into a globally defined tensor field with the same property. This obstruction is referred to as the globalization problem. 

\noindent \textbf{Locally Conformal Approach.} Within the globalization phenomenon, a distinguished instance arises when the local representatives fail to satisfy the exact tensorial transformation rules on overlaps but instead agree only up to multiplicative conformal factors. This perspective is formalized in the theory of  locally conformal geometries, which originated with the introduction of locally conformally symplectic manifolds \cite{Banyaga2002,Bazzoni2018, Hwa,Vaisman85}. In this setting, each chart is equipped with a symplectic form, yet on intersections the local symplectic $2$-forms no longer coincide. Rather, they differ by conformal rescalings. Provided that the corresponding conformal factors satisfy the cocycle condition, the local symplectic forms glue together into a line-bundle-valued $2$-form, which in turn fails to be closed. 

The locally conformal approach should not be regarded as a narrow or exceptional solution to the globalization problem. For instance, locally conformally symplectic manifolds are natural examples of Jacobi manifolds \cite{EsLeSaZaReview,Lich,Marle91,vaisman2002jacobi} (or, equivalently, of local Lie algebras in the sense of Kirillov \cite{Kirillov}). More concretely, every even-dimensional leaf of a Jacobi manifold carries a locally conformal symplectic structure \cite{LeSa17}. Thus, even from a purely geometric perspective, the locally conformal viewpoint on the globalization problem emerges not as a special case, but rather as a natural generalization.

The locally conformal geometry also has important consequences for Hamiltonian dynamics. While the symplectic formalism yields reversible and energy-preserving dynamics, the locally conformal symplectic framework naturally leads to irreversible or/and dissipative dynamics.

In the literature, owing to these geometric and dynamical features, the framework of locally conformal analysis has been extended well beyond symplectic geometry. Direct extensions have been carried out for Lagrangian dynamics~\cite{LCS-Lag-H} while the Hamilton-Jacobi theory in the locally conformal setting has been developed in~\cite{EsenLeonSarZaj1}. Furthermore, locally conformal analogues of cosymplectic structures~\cite{AtEsLeSa23,LeonLCCos} and Poisson structures~\cite{Vais-Dirac,Luca16,Luca16} have been studied in detail. Field-theoretic generalizations have also been pursued in the contexts of $k$-symplectic and multisymplectic geometry~\cite{Es-Cauchy,Es-k-symp}. Altogether, these developments highlight the role of locally conformal analysis as a unifying framework bridging geometric structures and dynamical systems across diverse settings.

\noindent \textbf{Multi-Hamiltonian Formalisms.}
Poisson manifolds constitute one of the most suitable geometric settings for the study of Hamiltonian dynamics, and as such, the literature on this subject is vast, see, for example, \cite{BhasVisw88,Laurent13,Marsden1999,Vaisman94,Weinstein98}. The algebraic structure underlying a Poisson manifold is characterized either by the Jacobi identity for the Poisson bracket---ensuring integrability of the associated Hamiltonian vector fields---or, equivalently, by the vanishing of the Schouten--Nijenhuis bracket of the Poisson bivector field. 
When one moves beyond a single Hamiltonian function, the construction of multi-Hamiltonian formalisms requires generalizations of these two fundamental conditions.

As discussed in the literature, there are two main ways to generalize Poisson geometry to multi-Hamiltonian settings, and they lead to different outcomes, \cite{deAzPeBu96,Gautheron1996,IbanezLeonPadron98,MarmoNJ,Vallejo}. 

\noindent \textbf{(1) Nambu--Poisson.} Extending the condition of integrability of Hamiltonian vector fields leads to the so-called fundamental identity (introduced by Takhtajan), which defines Nambu--Poisson geometry. 

\noindent \textbf{(2) Generalized Poisson.} On the other hand, extending the condition that the Schouten--Nijenhuis bracket of the Poisson bivector vanishes leads instead to generalized Poisson structures, defined by requiring the vanishing of higher-order Schouten--Nijenhuis brackets of multivector fields. 

In Section~\ref{sec:LCP-Jac}, we recall the mathematical formulation of Poisson manifolds and set the stage for their higher-order extensions. We then describe the two distinct generalizations of Poisson geometry in more detail. In our work, these cases are treated separately: Sections~\ref{sec:Nambu-PoisJac} and \ref{sec:LCNPois} are devoted to Nambu--Poisson geometry, while Section~\ref{sec:LCGP} addresses generalized Poisson geometry.

Nambu--Poisson manifolds possess a particularly rich geometry: the fundamental identity provides the basis for bi-Hamiltonian \cite{EsGhGu16,fernandes1994completely,GuNu93} and even multi-Hamiltonian structures \cite{esenvd2016,GoNu01,Goriely,Tempesta}, which in turn underlie complete integrability. For instance, tri-Hamiltonian structures in four dimensions allow super-integrability, a phenomenon of central importance in integrability theory. 

On the other hand, the natural extension of Poisson geometry to dissipative dynamics is realized in the theory of Jacobi manifolds. In this sense, the aforementioned multi-Hamiltonian generalizations also carry over to Jacobi geometry, leading to the notions of Nambu--Jacobi and generalized Jacobi manifolds.

\medskip
\noindent \textbf{Aim of the Present Work.} 
In this work, we investigate the locally conformal extensions of Nambu--Poisson and generalized Poisson manifolds, with particular emphasis on the induced multi-Hamiltonian dynamics. Although some definitions of such locally conformal generalizations can be found in the literature, detailed and systematic studies remain rare. Our primary objective is to establish a coherent framework in which local Nambu--Poisson and generalized Poisson structures---which may fail to be global---are glued together by means of conformal factors.

The novelty of this study lies in providing a unified locally conformal approach to both Nambu--Poisson and generalized Poisson geometries, together with an explicit construction of the corresponding Hamiltonian dynamics. Our first goal is to define locally conformally Nambu--Poisson and locally conformally generalized Poisson manifolds, and to demonstrate that they naturally arise as examples of Nambu--Jacobi and generalized Jacobi manifolds, respectively. Our second goal is to derive Hamiltonian-type dynamical equations within these settings. As a byproduct, we present a theory of locally conformal bi-Hamiltonian dynamics for the case of $3$-Nambu--Poisson manifolds. 

\medskip
\noindent \textbf{Content.} 
The paper is organized as follows. 
Section \ref{sec:LCP-Jac} examines  Hamiltonian dynamics, starting from classical Poisson and Jacobi manifolds and extending to locally conformally Poisson manifolds, highlighting how locally conformal transformations affect the Hamiltonian flow. Section \ref{sec:Nambu-PoisJac} focuses on Nambu--Poisson and Nambu--Jacobi manifolds, including $3$-Nambu structures, and investigates their associated Hamiltonian and bi-Hamiltonian dynamics. Section~\ref{sec:LCNPois} contains our main results on locally conformal analysis of Nambu--Poisson setting. Accordingly,  we introduce locally conformally Nambu--Poisson manifolds and explore the corresponding Hamiltonian dynamics, addressing $3$- and $k$-Nambu cases and the interplay with bi-Hamiltonian structures. Finally, Section \ref{sec:LCGP} extends the analysis to generalized Poisson Hamiltonian systems, discussing generalized Poisson and generalized Jacobi manifolds and their locally conformal analogues. 

\section{Poisson, Jacobi and Locally Conformal Poisson Setting}\label{sec:LCP-Jac}

\subsection{Poisson Manifolds and Hamiltonian Dynamics}

Consider a manifold $\mathcal{P}$ equipped with a skew-symmetric bilinear bracket defined on the space $C^\infty(\mathcal{P}) $ of smooth functions \cite{BhaskaraViswanath,Dufour,Vaisman94,weinstein1983local,Weinstein98} denoted as   \begin{equation}\label{PoissonBracket}
\{ \bullet,\bullet\}~:C^\infty(\mathcal{P})\times C^\infty(\mathcal{P}) \longrightarrow C^\infty(\mathcal{P}) .
\end{equation}
This is called Poisson bracket if it satisfies, for all $H,F_1,F_2$ and $F_3$ in $C^\infty(\mathcal{P}) $, 
\begin{axioms}
  \item[\textbf{P1.}]   Jacobi identity
	\begin{equation} \label{P2}
\{ \{ F_1, F_2\}, F_{3}\} + \{ \{F_2, F_3 \}, F_{1} \} + \{ \{ F_3, F_1\}, F_{2}\}=0,
\end{equation} 
	\item[\textbf{P2.}] Leibniz identity
\begin{equation}\label{P1}
\{ H F_1, F_2\}= H \{F_1, F_2 \} +
\{ H, F_2 \} F_1.
\end{equation}
\end{axioms}
If one assumes only the Leibniz identity (stated in P1) then the bracket is called an almost Poisson bracket. 
 

\textbf{Hamiltonian Vector Fields.}
Given a smooth function $H$ defined on a (almost) Poisson manifold $\mathcal{P}$ , the Hamiltonian vector field $X_H$ is defined to be, for all $F$ in $C^\infty(\mathcal{P})$,
\begin{equation}\label{P-Hamvf}
X_H(F)=\{F,H\}.
\end{equation}
For a given observable $F$, the Hamiltonian dynamics generated by a Hamiltonian function $H$ is 
\begin{equation} \label{P-HamEq}
\frac{dF}{dt}=\{F,H\}.
\end{equation} 
Skew-symmetry of the bracket manifests that the Hamiltonian function $H$ is conserved along the motion, that is $\{H,H\}=0$. Assuming the Hamiltonian function $H$ to be the total energy, this is conservation of energy. 

As a manifestation of the Jacobi identity, the (characteristic) distribution determined by the image spaces of all Hamiltonian vector fields is integrable. Further, we have that 
\begin{equation} \label{JL-iden}
[X_{H_1},X_{H_2}]=-X_{\{H_1,H_2\}},
\end{equation}
where the bracket on the left-hand side is the Jacobi-Lie bracket of vector fields, whereas the bracket on the right-hand side is the Poisson bracket. An important consequence of the Jacobi identity is the Poisson theorem which is the conservation of the flows
\begin{equation}\label{Poisson-theorem}
\frac{d}{dt}\{F_1,F_2\}= \{\frac{d F_1}{dt},F_2\}+ \{ F_1, \frac{d F_2}{dt}\}.
\end{equation}

\textbf{Bivector Field Realization.}
Due to the Leibniz identity, we identify an almost Poisson bracket $\{\bullet,\bullet\}$ with a bivector field $\Lambda$ on $\mathcal{P}$ given by 
\begin{equation} \label{bivec-PoissonBra}
\Lambda(dF_1,dF_2):=\{F_1,F_2\}
\end{equation}
for all $F_1$ and $F_2$ in $C^\infty(\mathcal{P})$. A direct calculation proves that  
\begin{equation} \label{SN-bra-JI}
\frac{1}{2}[\Lambda,\Lambda](dF_1,dF_2,dF_3)=~\circlearrowright \{F_1,\{F_2,F_3\}\},
\end{equation}
where the bracket on the left-hand side is the Schouten--Nijenhuis bracket (see Appendix \ref{sec:SN}) and $\circlearrowright $ denotes the cyclic sum. It is immediate to see from \eqref{SN-bra-JI} that $\{\bullet,\bullet\}$ is a Poisson bracket if and only if the bivector field $\Lambda$ commutes with itself under the Schouten--Nijenhuis bracket, that is, 
\begin{equation} \label{Poisson-cond}
[\Lambda,\Lambda]=0.
\end{equation} 
So, we may define a Poisson manifold by a pair $(\mathcal{P},\Lambda)$ with  $\Lambda$ satisfying the commutativity condition (that is, the Jacobi identity) \eqref{Poisson-cond}. A bivector field satisfying \eqref{Poisson-cond} is called a Poisson bivector. In terms of the bivector notation, for a Hamiltonian function $H$, the Hamiltonian vector field \eqref{P-Hamvf} is computed to be 
\begin{equation}
    X_H=\Lambda^\sharp (dH).
\end{equation}
Here, $\Lambda^\sharp$ is the musical mapping (see Appendix \ref{sec:music}) associated with the bivector field $\Lambda$.

\textbf{Two generalizations of Poisson algebra to higher order algebras.} In the literature, mainly, there are two generalizations of Poisson manifolds to the multilinear setting as a multilinear algebra on the function space:

\noindent \textbf{(1)} One is to generalize the binary Poisson bracket to multilinear Poisson bracket in such a way that the Poisson theorem \eqref{Poisson-theorem} holds to be true. This is equivalent to have an integrable characteristic distribution. As we shall present in the upcoming subsections, this promotes so called fundamental (Fillippov or Takhtajan) identity. In this case we arrive at a Nambu--Poisson algebra.  

\noindent \textbf{(2)} Following the Jacobi identity formulation in \eqref{Poisson-cond}, another way to generalize Poisson bivector field to a higher order multivector field can be achieved by a multivector field commuting under the Schouten--Nijenhuis bracket. Notice that the Schouten--Nijenhuis bracket vanishes for multivector fields of odd order. This generalization can only be performed for multivectors of even order. This is generalized Poisson bracket. 

Both of these geometries cover Poisson manifold if the order is two. We shall examine these generalizations one by one.

\subsection{Jacobi Manifolds and Hamiltonian Dynamics} \label{sec:JacMan}

Consider a manifold $\mathcal{P}$ equipped with a skew-symmetric bilinear bracket
\begin{equation}\label{JacobiBracket}
\{ \bullet,\bullet\}^{\rm(J)}~:C^\infty(\mathcal{P})\times C^\infty(\mathcal{P}) \longrightarrow C^\infty(\mathcal{P}) .
\end{equation}
We call this a  Jacobi bracket if it satisfies the following axioms for all $F_1,F_2,F_3 \in C^\infty(\mathcal{P})$:  

\begin{axioms}
  \item[\textbf{J1.}]  Jacobi identity 
	\begin{equation} \label{J2}
	\{\{F_1, F_2\}^{\rm(J)}, F_3\}^{\rm(J)} + \{\{F_2, F_3\}^{\rm(J)}, F_1\}^{\rm(J)} + \{\{F_3, F_1\}^{\rm(J)}, F_2\}^{\rm(J)} = 0,
	\end{equation} 

  \item[\textbf{J2.}]  Weak Leibniz rule 
  \begin{equation}\label{supp}
        \operatorname{supp}(\{F_1,F_2\}^{\rm(J)}) \subseteq \operatorname{supp}(F_1) \cap \operatorname{supp}(F_2). 
  \end{equation}
 \end{axioms}
A manifold with a Jacobi bracket is called a Jacobi manifold and denoted as $(\mathcal{P},\{\bullet,\bullet\}^{\rm(J)})$. We refer the incomplete list \cite{Lichnerowicz-Poi,Lichnerowicz-Jacobi,Marle-Jacobi,VitaWade20} for the theory of  Jacobi manifolds. Jacobi manifold is a local Lie algebra in the sense of Kirillov \cite{Kirillov}.  The inverse of this assertion is also true, that is, a local Lie algebra determines a Jacobi structure. In the following subsection, we shall illustrate this by showing that every locally conformally Poisson manifold is a Jacobi manifold.

\textbf{Multivector Field Realization.}
Referring to a Jacobi manifold $(\mathcal{P},\{\bullet,\bullet\}^{\rm(J)})$, we associate a vector field $Z$ and a bivector field $\Lambda$ as follows:
\begin{equation}\label{baris}
    \begin{split}
        Z(dF) &:= \{1,F\}^{\rm(J)}, \\
        \Lambda(dF_1,dF_2) &:= \{F_1,F_2\}^{\rm(J)} - F_1 Z(dF_2) + F_2 Z(dF_1),
    \end{split}
\end{equation}
for all $F,F_1,F_2 \in C^\infty(\mathcal{P})$. A direct calculation shows that $\Lambda$ and $Z$ satisfy the identities
\begin{equation}\label{ident-Jac}
        [\Lambda,\Lambda] = -2 Z \wedge \Lambda, 
        \qquad 
        [Z,\Lambda] = 0,
\end{equation}
where $[\bullet,\bullet]$ denotes the Schouten--Nijenhuis bracket.
Rearranging the terms in \eqref{baris}, we can express the Jacobi bracket as
\begin{equation}\label{bra-Jac}
    \{F_1,F_2\}^{\rm(J)} = \Lambda(dF_1,dF_2) + F_1 Z(F_2) - F_2 Z(F_1).
\end{equation}
It is possible to see that the bracket \eqref{bra-Jac} satisfies the Jacobi identity if and only if the conditions in \eqref{ident-Jac} hold. Further, regarding the bracket \eqref{bra-Jac}, for a product of functions, one obtains
\begin{equation}\label{Jacobi-bracket-} 
    \{H \cdot F_1, F_2\}^{\rm(J)} 
    = H \cdot \{F_1,F_2\}^{\rm(J)} 
      + \{H,F_2\}^{\rm(J)} \cdot F_1 
      - H \cdot F_1 \cdot Z(F_2).
\end{equation}
This shows that the Jacobi bracket is a first-order differential operator.  
This property is equivalent to the weak Leibniz rule \eqref{supp}. Thus, one may give an alternative definition of a Jacobi manifold: a pair $(\Lambda,Z)$, called a Jacobi pair (or Jacobi structure), defines a Jacobi bracket precisely when the identities \eqref{ident-Jac} are satisfied. Equivalently, a Jacobi manifold can be denoted by the triple $(\mathcal{P},\Lambda,Z)$.  

Finally, referring to \eqref{Jacobi-bracket-}, we see that the full Leibniz rule holds if and only if $Z$ vanishes identically. In this case, the first identity in \eqref{ident-Jac} reduces to the Poisson condition \eqref{Poisson-cond}, while the second one is automatically satisfied.  
Thus, we conclude that $(\mathcal{P},\Lambda,Z)$ is a Poisson manifold if and only if $Z=0$. 

\textbf{Hamiltonian Vector Fields.} 
Consider a Jacobi manifold determined by the triplet $(\mathcal{P}, \Lambda, Z)$. 
For a smooth  real-valued Hamiltonian function $H$ on a  Jacobi manifold $(\mathcal{P}, \Lambda, Z)$, the  Hamiltonian vector field $X_{H}$ is defined by
\begin{equation}\label{HamVF-Jac}
X_{H}=\Lambda^\sharp(d H)-H Z = \iota_{dH}\Lambda - HZ. 
\end{equation}
In terms of the Jacobi bracket the Hamiltonian vector field can be written as
\begin{equation}
	\{F,H\}^{\rm(J)} = X_H (F) + FZ(H).
\end{equation}
Note that, for the constant function $H=1$ the Hamiltonian vector field is $-Z$. As proved in \cite{Lichnerowicz-Jacobi,Marle85}, the mapping taking a Hamiltonian function $H$ to the Hamiltonian vector field $X_H$ is a Lie algebra homomorphism by satisfying
\begin{equation}
\left[X_{F}, X_{H}\right]= - X_{\{F, H\}^{\rm(J)}}.
\end{equation}

\subsection{Locally Conformally Poisson Manifolds and Hamiltonian Dynamics}\label{sec:LCP}

In this subsection, we focus on the globalization problem for local Poisson structures that do not naturally glue together to form a global Poisson structure. We approach this phenomenon from a locally conformal perspective.

\noindent \textbf{Locally Conformally Poisson Manifolds.}
Consider a manifold $\mathcal{P}$ admitting a local covering 
\begin{equation}
    \mathcal{P} = \bigsqcup_\alpha U_\alpha.
\end{equation}
On each local chart, we assume the existence of a local Poisson structure. That is, we have local Poisson manifolds 
\begin{equation}
(U_\alpha, \Lambda_\alpha), \quad (U_\beta, \Lambda_\beta), \quad (U_\gamma, \Lambda_\gamma), \quad \dots . 
\end{equation}
These local bivector fields may combine to determine a global bivector field by obeying the usual tensorial transformation rules on overlapping charts. This reads a Poisson manifold. In the present work, however, we do not assume this. Instead, on the intersection of two charts, say $U_\alpha \cap U_\beta$, we assume that a conformal factor relates the local Poisson bivector fields:
\begin{equation}\label{kuskun}
    e^{-\sigma_\alpha} \Lambda_\alpha = e^{-\sigma_\beta} \Lambda_\beta.
\end{equation}
Here, the functions $\{\sigma_\alpha\}$ are local functions reflecting the conformal character of this structure. Furthermore, we assume that their exterior derivatives agree on overlaps, that is,
\begin{equation}
d\sigma_\alpha = d\sigma_\beta.
\end{equation}
Hence, by the Poincar\'e lemma, the functions $\sigma_\alpha$ are local
potential functions of a globally defined closed one-form $\theta$, given by
\begin{equation}\label{theta1}
    \theta\vert_\alpha = d\sigma_\alpha.
\end{equation}
The one-form $\theta$ is called the Lee form associated with the
locally conformal structure.

Note that the equality in \eqref{kuskun} does not allow us to glue the local Poisson structures $\{\Lambda_\alpha\}$ into a global tensorial bivector field on $\mathcal{P}$. Instead, the conformal relation \eqref{kuskun} determines the transition scalars
\begin{equation} \label{transition}
\lambda_{\beta \alpha} = \frac{e^{\sigma_\alpha}}{e^{\sigma_\beta}} = e^{-(\sigma_\beta - \sigma_\alpha)},
\end{equation}
which satisfy the cocycle condition
\begin{equation}\label{cocycle}
\lambda_{\beta \alpha} \, \lambda_{\alpha \gamma} = \lambda_{\beta \gamma}.
\end{equation}
Consequently, by gluing the local Poisson structures $\{\Lambda_\alpha\}$, one obtains a line-bundle--valued bivector field on $\mathcal{P}$.

On the other hand, in light of \eqref{kuskun}, the following local bivector fields
\begin{equation}\label{LCP-bi-vector}
\Lambda|_\alpha = e^{-\sigma_\alpha} \Lambda_\alpha
\end{equation}
can be glued together to define a global bivector field on $\mathcal{P}$. We denote the resulting global bivector field by $\Lambda$  whose local expression on $ U_\alpha $ is given by $ \Lambda|_\alpha $.   

The present geometry, namely the collection $\{(U_\alpha,\Lambda_\alpha)\}$ of local Poisson charts that satisfy the locally conformal relation on overlaps as in \eqref{kuskun}, is called a  locally conformally Poisson manifold; see, for instance, \cite{Vais-Dirac, Luca16}. We denote a locally conformally Poisson manifold as 
\begin{equation}
(\mathcal{P},U_\alpha,\Lambda_\alpha,\sigma_\alpha).
\end{equation}
A locally conformally Poisson manifold defines a local Lie algebra in the sense of Kirillov \cite{Kirillov}. It is important to note that we are able to construct a global bivector field $\Lambda$ from the local ones according to \eqref{LCP-bi-vector}; however, this is achieved at the expense of losing the Poisson property of the original local bivectors.

To explore the properties of the bivector field $\Lambda$, and in light of the intersection condition for the local bivector fields in \eqref{kuskun}, we introduce a set of local functions $\{F_\alpha\}$. Referring to these local functions and the conformal factors $\{\sigma_\alpha\}$, we define a new set of local functions by
\begin{equation}\label{gluefunction}
    F|_\alpha= e^{\sigma_\alpha} F_\alpha, \qquad   H|_\alpha= e^{\sigma_\alpha} H_\alpha, 
\end{equation}
in such a way that the $F|_\alpha$ and $H|_\alpha$ can be glued to form global functions $F$ and $H$, respectively.
We compute the local bracket as follows
\begin{equation}
	\begin{split}
		\{F_\alpha,H_\alpha\}_\alpha &= \Lambda_\alpha(dF_\alpha,dH_\alpha) \\
		&= e^{\sigma_\alpha}\Lambda\vert_\alpha(d(e^{-\sigma_\alpha}F\vert_\alpha),d(e^{-\sigma_\alpha}H\vert_\alpha)) \\
		&= e^{-\sigma_\alpha}\Lambda\vert_\alpha(dF\vert_\alpha - F\vert_\alpha d\sigma_\alpha,dH\vert_\alpha - H\vert_\alpha d\sigma_\alpha).
	\end{split}
\end{equation}
Consequently, we obtain
\begin{equation}\label{gluedPoisson}
	\begin{split}
		e^{\sigma_\alpha}\{F_\alpha,H_\alpha\}_\alpha &= \Lambda\vert_\alpha(dF\vert_\alpha,dH\vert_\alpha) - F\vert_\alpha\Lambda\vert_\alpha(d\sigma_\alpha,dH\vert_\alpha) - H\vert_\alpha\Lambda\vert_\alpha(dF\vert_\alpha,d\sigma_\alpha).
	\end{split}
\end{equation}
We now define the vector field
\begin{equation}\label{LCP-vector}
Z\vert_\alpha := \iota_{d\sigma_\alpha}\Lambda\vert_\alpha,
\end{equation}
and state the following theorem (see also \cite{Luca16}).

\begin{theorem}\label{prop-LCP-Jacobi} A locally conformal Poisson manifold is a Jacobi manifold. More precisely, the pair $(\Lambda,Z)$ where $\Lambda$ is in \eqref{LCP-bi-vector} and $Z$ is the vector field in \eqref{LCP-vector} is a Jacobi pair.  
\end{theorem}

\begin{proof} To show that $(\mathcal{P},\Lambda,Z)$ is a Jacobi manifold we need to justify the identities in \eqref{ident-Jac}. 
We start with the local Poisson bivector $\Lambda_\alpha$. Since it is Poisson, its Schouten--Nijenhuis bracket is zero, i.e, $[\Lambda_\alpha,\Lambda_\alpha]=0$. We start with this identity and do the following calculation
\begin{equation}
		\begin{split}		[\Lambda_\alpha,\Lambda_\alpha] &= [e^{\sigma_\alpha}\Lambda\vert_\alpha,e^{\sigma_\alpha}\Lambda\vert_\alpha] 
			\\
			&= e^{\sigma_\alpha}[e^{\sigma_\alpha}\Lambda\vert_\alpha,\Lambda\vert_\alpha] + \iota_{d(e^{\sigma_\alpha})}(e^{\sigma_\alpha}\Lambda\vert_\alpha)\wedge \Lambda\vert_\alpha
			\\
			&= e^{\sigma_\alpha}[\Lambda\vert_\alpha,e^{\sigma_\alpha}  \Lambda\vert_\alpha] + e^{2\sigma_\alpha}\iota_{d{\sigma_\alpha}}\Lambda\vert_\alpha\wedge \Lambda\vert_\alpha \\
                &= e^{\sigma_\alpha}\big(e^{\sigma_\alpha}[\Lambda\vert_\alpha,\Lambda\vert_\alpha] + (\iota_{d(e^{\sigma_\alpha})}\Lambda\vert_\alpha)\wedge \Lambda\vert_\alpha\big) + e^{2\sigma_\alpha}\iota_{d{\sigma_\alpha}}\Lambda\vert_\alpha\wedge \Lambda\vert_\alpha \\
			& = e^{2\sigma_\alpha}[\Lambda\vert_\alpha, \Lambda\vert_\alpha] + e^{2\sigma_\alpha}\iota_{d{\sigma_\alpha}}\Lambda\vert_\alpha\wedge \Lambda\vert_\alpha + e^{2\sigma_\alpha}\iota_{d{\sigma_\alpha}}\Lambda\vert_\alpha\wedge \Lambda\vert_\alpha \\
			&=e^{2\sigma_\alpha}[\Lambda\vert_\alpha, \Lambda\vert_\alpha] + 2e^{2\sigma_\alpha}\iota_{d{\sigma_\alpha}}\Lambda\vert_\alpha\wedge \Lambda\vert_\alpha.
		\end{split}
	\end{equation}
Since this expression is zero, we can conclude that
\begin{equation}
[\Lambda\vert_\alpha,\Lambda\vert_\alpha] =-2 \iota_{d{\sigma_\alpha}}\Lambda\vert_\alpha\wedge \Lambda\vert_\alpha.
\end{equation}
Notice that all the tensorial fields in this expression admit global realizations. That gives us the following global equality
\begin{equation}\label{LCS-Jac-1}
[\Lambda,\Lambda]=-2Z\wedge \Lambda
\end{equation}
where we have employed \eqref{LCP-vector}. This satisfies the first requirement in definition \eqref{ident-Jac}. For the second one, let us recall the Lee form $\theta$ from \eqref{theta1} then write the global expression of the vector field $Z$ defined in \eqref{LCP-vector} as
\begin{equation}\label{Z-global}
    Z = \iota_{\theta}\Lambda.
\end{equation}
Using the distributive properties of interior derivative on Schouten--Nijenhuis bracket and on wedge products, we have the following calculation: 
 \begin{equation}\label{Cals-1}
    \begin{split}
        [Z ,\Lambda]&=[\iota_{\theta}\Lambda,\Lambda]=-\frac{1}{2} \iota_{\theta} [\Lambda,\Lambda]
        \\
        &= -\frac{1}{2}\iota_{\theta}(-2Z\wedge \Lambda) = \iota_{\theta}(\iota_{\theta} \Lambda \wedge \Lambda)  \\
        &= \iota_{\theta}\iota_{\theta} \Lambda  \wedge \Lambda + \iota_{\theta} \Lambda \wedge \iota_{\theta}\Lambda = 0.
    \end{split}
	\end{equation}
The details of the calculation can be given as follows: In the first equality we wrote the global expression \eqref{Z-global} of the vector field Z. In the second equality, we used the distribution property 
    \eqref{iota-SN-2-2-X} of the right contraction on the Schouten--Nijenhuis bracket. For the third equality, we employed \eqref{LCS-Jac-1} while we used the distribution property 
    \eqref{iota-SN-2-2-X-} of the right contraction on the wedge product to arrive at the fifth equality.
In the last equation, both terms vanish identically due to the skew-symmetry. So, we can argue now that $(\mathcal{P},\Lambda,Z)$ is indeed a Jacobi manifold. \end{proof}

By Theorem~\ref{prop-LCP-Jacobi}, every locally conformal Poisson structure admits an associated Jacobi structure $(\mathcal{P},\Lambda,Z)$, with vector field $Z=\iota_\theta\Lambda$ arising canonically from the closed one-form $\theta$.

\textbf{Locally Conformal Hamiltonian Dynamics.}
Let us now concentrate on Hamiltonian dynamics on Poisson manifolds. Consider a local chart  $(U_\alpha,\Lambda_\alpha)$ and a Hamiltonian function $H_\alpha$ on this chart. We write the Hamilton's equations by  
\begin{equation}\label{geohamalpha}
X_{H_\alpha} =\iota_{dH_\alpha}\Lambda_\alpha =\Lambda^\sharp_\alpha (dH_\alpha).
\end{equation}
Here, $X_{H_{\alpha}}$ is the local Hamiltonian function associated to this framework. In order to recast the global picture of Hamilton's equation we first substitute the identities \eqref{LCP-bi-vector} and \eqref{gluefunction} into \eqref{geohamalpha}. Hence, a direct calculation turns the Hamilton's equations into the following form
\begin{equation}\label{ham-for-LCP}
    \begin{split}
        X_{H_\alpha} &=\iota_{dH_\alpha}\Lambda_\alpha = \iota_{d(e^{-\sigma_\alpha}H\vert_\alpha)}e^{\sigma_\alpha}\Lambda\vert_\alpha = \iota_{dH\vert_\alpha-H\vert_\alpha d\sigma_\alpha}\Lambda\vert_\alpha \\ 
        &= \iota_{dH\vert_\alpha}\Lambda\vert_\alpha - H\vert_\alpha\iota_{ d\sigma_\alpha}\Lambda\vert_\alpha \\
        & = \Lambda^\sharp\vert_\alpha(dH\vert_\alpha) - H\vert_\alpha Z\vert_\alpha,
    \end{split}
\end{equation}
where we have employed the identification \eqref{LCP-vector}. Notice that all the terms we obtained in the last line of \eqref{ham-for-LCP} have global realizations. So, we can write
\begin{equation}\label{semiglobal}
X_{H}=\Lambda^\sharp(dH) - H Z,
\end{equation}
where $X_{H}$ is the vector field obtained by gluing all the vector fields $X_{H_{\alpha}}$. That is, we have $X_H\vert_\alpha=X_{H_{\alpha}}$.
The vector field $X_{H}$ defined in (\ref{semiglobal}) is called locally conformal Hamiltonian vector field for Hamiltonian function $H$.

It is worth noting that the Hamiltonian vector field \eqref{semiglobal} arising in our locally conformal analysis takes the same form as the Hamiltonian vector field \eqref{HamVF-Jac} in the Jacobi setting. This correspondence provides a direct dynamical manifestation of Theorem~\ref{prop-LCP-Jacobi}, which asserts that locally conformally Poisson manifolds are Jacobi.

\section{Nambu--Poisson and Nambu--Jacobi Setting}\label{sec:Nambu-PoisJac}

\subsection{Nambu--Poisson Manifolds and Hamiltonian Dynamics}
 
Consider a manifold $\mathcal{P}$ and the space $C^\infty(\mathcal{P})$ of  smooth functions on $\mathcal{P}$ equipped with $k$-multilinear mapping
\begin{equation}\label{NP-n}
\{ \bullet,\bullet,\dotsm,\bullet \}^{\rm(NP)}~:C^\infty(\mathcal{P})\times \dots \times C^\infty(\mathcal{P}) \longrightarrow C^\infty(\mathcal{P}) .
\end{equation}
We assume that the bracket is skew-symmetric, in the sense that it satisfies
\begin{equation}\label{sskew}
\{ F_1, \dotsm , F_k \}^{\mathrm{(NP)}} 
= (-1)^{\epsilon(\sigma)} 
\{ F_{\sigma(1)}, \dotsm , F_{\sigma(k)} \}^{\mathrm{(NP)}},
\end{equation}
for any $\sigma \in S_k$, the symmetric group on $k$ elements, where
$\epsilon(\sigma)$ denotes the parity of $\sigma$.
The bracket \eqref{NP-n} is called a Nambu--Poisson $k$-bracket if the following axioms hold for all 
$H,F_1 , F_2 , \dotsm , F_{2k-1}$ in $C^\infty(\mathcal{P})$:
\begin{axioms}
  \item[\textbf{NP1.}]  
  Fundamental (also called Fillippov or Takhtajan) identity
	\begin{equation} \label{NP2}
	\begin{split}
&\{ \{ F_1, \dotsm , F_{k-1}, F_k \}^{\rm(NP)}, F_{k+1}, \dotsm, F_{2k-1} \}^{\rm(NP)}  \\&\qquad +
\{ F_k, \{ F_1, \dotsm, F_{k-1}, F_{k+1} \}^{\rm(NP)}, F_{k+2}, \dotsm , F_{2k-1} \}^{\rm(NP)} \\& \qquad
 +  \ldots + \{ F_k, \dotsm ,F_{2k-2}, \{ F_1, \dotsm , F_{k-1}, F_{2k-1} \}^{\rm(NP)}\}^{\rm(NP)} \\ &\qquad \qquad
 =  \{ F_1, \dotsm , F_{k-1}, \{ F_k, \dotsm , F_{2k-1} \}^{\rm(NP)}\}^{\rm(NP)},
 \end{split}
\end{equation} 
	\item[\textbf{NP2.}]  Leibniz identity
\begin{equation}\label{NP1}
\{ H \cdot F_1, F_2, \dotsm ,F_{k} \}^{\rm(NP)}=
H \cdot\{F_1, F_2, \dotsm , F_{k} \} ^{\rm(NP)}+
\{ H, F_2, \dotsm, F_{k} \}^{\rm(NP)} \cdot F_1.
\end{equation} 
\end{axioms}
A manifold $\mathcal{P}$ equipped with a Nambu--Poisson $k$-bracket $\{\bullet,\bullet,\dotsm,\bullet\}^{\rm(NP)}$ is called a $k$-Nambu--Poisson manifold, \cite{daletskii1997leibniz,GrabMarm2000,Nakanishi1998,Namb73,Takh94}. Notice that, for $k=2$ one arrives at a Poisson manifold. In this case, \eqref{NP2} turns out to be the Jacobi identity \eqref{P2}.

\textbf{Multivector Field Realization.}
A Nambu--Poisson $k$-bracket is determined by a $k$-vector field 
$\eta^{[k]} \in \Gamma(\wedge^k T\mathcal{P})$, that is, a section of the bundle of skew-symmetric multivector fields on $\mathcal{P}$, satisfying
\begin{equation} \label{eta-intro}
\eta^{[k]}(dF_1,\dotsc,dF_k) = \{F_1,\dotsc,F_k\}^{\rm (NP)}.
\end{equation}
Such a multivector field $\eta^{[k]}$ is referred to as a Nambu--Poisson $k$-vector field.

Any $k$-bracket defined by a $k$-vector field automatically satisfies the Leibniz identity \eqref{NP1}.  
However, in order to fulfill the fundamental identity, additional conditions must be imposed on the $k$-vector field.  
In the literature, several approaches are available to verify whether a given $k$-vector field satisfies the fundamental identity.  
In this work, we adopt the following characterization,  see, for instance, \cite{GrabMarm99}.  
\begin{theorem}\label{emre}
    A $k$-vector field $\eta^{[k]}$ is a Nambu--Poisson $k$-vector field if and only if its contraction  $
\iota_{dF}\eta^{[k]}
$
is a Nambu--Poisson $(k-1)$-vector field for all smooth functions $F$ on $\mathcal{P}$.
\end{theorem}

We now present a corollary of Theorem \ref{emre}, by choosing $k=3$. This establishes the connection between $3$-Nambu--Poisson manifolds and Poisson manifolds and will play a key role in our subsequent analysis.

\begin{corollary}\label{emre1}
A $3$-vector field $\eta^{[3]}$ is Nambu--Poisson if and only if, for an arbitrary function $F$, its contraction
\begin{equation}\label{first-contraction}
    \Lambda = \iota_{dF}\eta^{[3]}
\end{equation}
is a Poisson bivector field.
\end{corollary}

By iterating the contraction process described in Theorem \ref{emre}, one eventually reaches bivector fields.  
Thus, a $k$-bracket $\{\bullet,\dotsc,\bullet\}^{\rm (NP)}$ defines a Nambu--Poisson bracket precisely when all of its contractions with $k-2$ functions yield a Poisson bivector.  
Equivalently, the contracted multivector field—now a bivector field—must satisfy
\begin{equation}\label{condNP}
\big[\iota_{dF_3 \wedge \dotsb \wedge dF_k} \eta^{[k]}, \, \iota_{dF_3 \wedge \dotsb \wedge dF_k} \eta^{[k]}\big] = 0, \qquad \forall F_3, \dotsc, F_k   
\end{equation}
with respect to the Schouten--Nijenhuis bracket (see Appendix \ref{sec:SN}) in accordance with \eqref{Poisson-cond}. 
Accordingly, we denote a Nambu--Poisson manifold by $(\mathcal{P},\eta^{[k]})$, where the $k$-vector field $\eta^{[k]}$ satisfies \eqref{condNP}.



\textbf{Nambu Hamiltonian Dynamics.} Consider a $k$-Nambu--Poisson manifold $\mathcal{P}$. The Nambu Hamiltonian dynamics on $\mathcal{P}$ is governed by
$(k-1)$ Hamiltonian functions $(H_1,\dotsm,H_{k-1})$ in $C^\infty(\mathcal{P})$. The associated Nambu Hamilton's equation is
\begin{equation} \label{NH-Eq}
\frac {dx}{dt} =  \{ x , H_1 ,\dotsm,H_{k-1} \} ^{\rm(NP)}.
\end{equation}
For an observable $F$, the corresponding Nambu Hamiltonian vector field is defined to be
\begin{equation} \label{sorarelif}
 X_{H_1 ,\dots ,H_{k-1}}(F)=\{ F , H_1 ,\dotsm,H_{k-1} \}^{\rm(NP)}.
\end{equation}
In terms of the interior derivative (see Appendix \ref{sec:intder}), we write the Nambu Hamiltonian vector field as
\begin{equation} \label{NP-HVF}
 X_{H_1 ,\dots ,H_{k-1}}= \iota_{dH_1\wedge \dots \wedge dH_{k-1}} \eta^{[k]} = \iota_{dH_1}\dots\iota_{dH_{k-1}} \eta^{[k]} .
\end{equation}

Poisson’s theorem states that the time derivative of a Poisson bracket satisfies the product rule, as can be seen in \eqref{Poisson-theorem}.  
This property extends naturally to the Nambu--Poisson $k$-bracket.  
In particular, the time derivative of a Nambu--Poisson $k$-bracket also satisfies the product rule:
\begin{equation}\label{Poisson-theorem-k}
\frac{d}{dt}\{F_1,\dots,F_k\}^{\rm(NP)} 
=  \sum_{i=1}^k  \{F_1,\dots,F_{i-1},\frac{d F_i}{dt},F_{i+1},\dots,F_k\}^{\rm(NP)}.
\end{equation}


A solution to the Nambu Hamilton's equations of motion produces an evolution 
operator which by virtue of the Takhtajan identity preserves  
the Nambu--Poisson $k$-bracket structure. According to that, a Nambu Hamiltonian vector field preserves the Nambu--Poisson $k$-vector field $\eta^{[k]}$ that is 
\begin{equation} 
\mathcal{L}_{X_{H_1,\dots  ,H_{k-1}}} \eta^{[k]} = 0 .
\end{equation}
Further, the characteristic distribution, generated by all the Nambu Hamiltonian vector fields, is
involutive 
\begin{equation} 
\big[X_{H_1,\dots  ,H_{k-1}} ,X_{F_1,\dots  ,F_{k-1}}\big]= (-1)^{k-1}
\sum_{i=1}^{k-1}  X_{F_1,\dots ,\{H_1,\dots  ,H_{k-1},F_i\}^{\rm(NP)}, \dots , F_{k-1}}.
\end{equation}

\subsection{3-Nambu--Poisson Manifolds and Bi-Hamiltonian Dynamics}\label{sec:BiH}

We begin by recalling the notion of bi-Hamiltonian dynamics and then demonstrate how a $3$-Nambu--Poisson structure naturally gives rise to an instance of such dynamics. 

\noindent \textbf{Bi-Hamiltonian Dynamics.} On a manifold $\mathcal{P}$, assume that there exist two Poisson bivectors, denoted by $\Lambda^{1}$ and $\Lambda^{2}$. We write the corresponding Poisson brackets as $\{\cdot,\cdot\}^{1}$ and $\{\cdot,\cdot\}^{2}$, respectively. These Poisson structures are said to be  compatible if the linear pencil of bivector fields
\begin{equation}
\Lambda^1 + \,\Lambda^2, 
\end{equation}
defines a Poisson structure, see \cite{MaMo84, olver86}.  
This condition is equivalent to the vanishing of their Schouten--Nijenhuis bracket:
\begin{equation}\label{hopa}
[\Lambda^1,\Lambda^2]= 0.
\end{equation}
Equivalently, one may view compatibility as the requirement that the bracket
\begin{equation}\label{pencil}
\{F,H \} = \{F,H \}^1 + \{F,H \}^2
\end{equation}
satisfies the Jacobi identity.  

A dynamical system is called bi-Hamiltonian if it admits two distinct but compatible Hamiltonian structures, that is,
\begin{equation} \label{Biham-Eq}
\frac {dx}{dt} =  \{ x , H_1  \}^{2}, 
\qquad  
\frac {dx}{dt} =  \{ x , H_2  \}^{1}.
\end{equation}
In three dimensions, every autonomous bi-Hamiltonian system is integrable.

\noindent \textbf{Bi-Hamiltonian Dynamics on $3$-Nambu--Poisson Manifolds.} 
In this part, we concentrate on the $k=3$ case of $k$-Nambu--Poisson manifolds. Accordingly, we have a 
skew-symmetric $3$-bracket written as
\begin{equation}\label{NP-3}
\{ \bullet,\bullet,\bullet \}^{\rm(NP)}~:C^\infty(\mathcal{P})\wedge C^\infty(\mathcal{P})\wedge C^\infty(\mathcal{P}) \longrightarrow C^\infty(\mathcal{P}).
\end{equation}
In this case, we record that the fundamental identity \eqref{NP2} turns out to be
\begin{equation}\label{3-Lie-cond2} 
    \begin{split}
        \{F_1,F_2,\{F_3,F_4,F_5\}^{\rm(NP)}\}^{\rm(NP)} & = \{\{F_1,F_2,F_3\}^{\rm(NP)},F_4,F_5\}^{\rm(NP)} + \{F_3,\{F_1,F_2,F_4\}^{\rm(NP)},F_5\}^{\rm(NP)} \\
        & \qquad \qquad \qquad + \{F_3,F_4,\{F_1,F_2,F_5\}^{\rm(NP)}\}^{\rm(NP)}. 
    \end{split}
\end{equation}
As a particular instance of \eqref{eta-intro}, we determine Nambu--Poisson $3$-vector field as 
\begin{equation}
\eta^{[3]}(dF_1,dF_2,dF_3)=\{ F_1,F_2,F_3 \}^{\rm(NP)}. 
\end{equation}

We fix the second and third entries of the Nambu--Poisson $3$-bracket \eqref{NP-3} to obtain the following ($2$-)brackets:
\begin{equation}\label{newPoisson}
\{ F,H\}^1 := \{ F,H_{1},H\}^{\rm (NP)}, 
\qquad 
\{ F,H\}^2 := \{ F,H,H_{2}\}^{\rm (NP)}.
\end{equation}
Referring to the fundamental identity \eqref{3-Lie-cond2}, a direct analysis shows that both brackets $\{\cdot,\cdot\}^1$ and $\{\cdot,\cdot\}^2$ satisfy the Leibniz and Jacobi identities, and hence define Poisson brackets. Moreover, it follows that their linear pencil \eqref{pencil} is again a Poisson bracket, which implies that the two brackets are compatible Poisson structures.  

Let $\Lambda^1$ and $\Lambda^2$ denote the bivector fields associated with the Poisson brackets $\{\cdot,\cdot\}^1$ and $\{\cdot,\cdot\}^2$, respectively. Explicitly, they are defined by
\begin{equation}
    \Lambda^1(dF,dH) = \{F,H\}^1, 
    \qquad 
    \Lambda^2(dF,dH) = \{F,H\}^2.
\end{equation}
According to the definition in \eqref{newPoisson}, these bivector fields can be expressed as right contractions of the Nambu--Poisson $3$-vector field $\eta^{[3]}$:
\begin{equation}\label{NP-to-Poisson}
     \Lambda^1 = -\,\iota_{dH_1}\eta^{[3]}, 
     \qquad 
     \Lambda^2 = \iota_{dH_2}\eta^{[3]}.
\end{equation}
Moreover, from \eqref{first-contraction} it follows immediately that $\Lambda^1$ and $\Lambda^2$ commute with respect to the Schouten--Nijenhuis bracket. Consequently, both $\Lambda^1$ and $\Lambda^2$ are genuine Poisson bivectors.

For a pair $(H_1,H_2)$ of Hamiltonian functions, the Nambu Hamilton’s equation is given by
\begin{equation} \label{NH-Eq-3}
\frac{dx}{dt} = \{x , H_1 , H_2\}^{\rm (NP)}. 
\end{equation}
The corresponding Nambu Hamiltonian vector field acts on functions as
\begin{equation}
    X_{H_1 ,H_2}(F) = \{F , H_1 , H_2\}^{\rm (NP)}.
\end{equation}
In terms of the Nambu--Poisson $3$-vector field $\eta^{[3]}$, and using the right contraction notation, the Nambu Hamiltonian vector field can be written as
\begin{equation}\label{zeynep}
    X_{H_1 ,H_2} = \iota_{dH_1 \wedge dH_2}\eta^{[3]} 
    = \iota_{dH_1}\iota_{dH_2}\eta^{[3]}.
\end{equation}
Hence, fixing the Hamiltonian functions $H_{1}$ and $H_{2}$, the Nambu Hamiltonian system \eqref{NH-Eq-3} can be expressed in the bi-Hamiltonian form (see, for instance, \cite{Guha06,EsGhGu16}):
\begin{equation}\label{NH-2Ham}
    \frac{dx}{dt} = \{x ,H_{1}\}^2 = \{x ,H_2\}^1,
\end{equation}
where the compatible Poisson brackets are those defined in \eqref{newPoisson}.

\subsection{Nambu--Jacobi Manifolds and Hamiltonian Dynamics} \label{sec:NJ}

Consider a manifold $\mathcal{P}$ and a skew-symmetric (see the definition \eqref{sskew}) $k$-bracket 
\begin{equation}\label{NJ-n}
\{ \bullet,\bullet,\dotsm,\bullet \}^{\rm(NJ)}~:C^\infty(\mathcal{P})\times \dots \times C^\infty(\mathcal{P}) \longrightarrow C^\infty(\mathcal{P}) 
\end{equation}
on the space $C^\infty(\mathcal{P})$ of smooth functions on $\mathcal{P}$. 
The bracket is called a Nambu--Jacobi $k$-bracket if the following axioms are satisfied for all 
 $H, F_1 , F_2 , \ldots , F_{2k-1}$ in $C^\infty(\mathcal{P})$:
\begin{axioms}
    \item[\textbf{NJ1.}] 
    \label{item:NJ2} Fundamental (also called Fillippov or Takhtajan) identity
	\begin{equation}
	\begin{split}
&\{ \{ F_1, \dotsm , F_{k-1}, F_k \}^{\rm(NJ)}, F_{k+1}, \dotsm, F_{2k-1} \}^{\rm(NJ)}  \\& \qquad  
+\{ F_k, \{ F_1, \dotsm, F_{k-1}, F_{k+1} \}^{\rm(NJ)}, F_{k+2}, \dotsm , F_{2k-1} \}^{\rm(NJ)} \\ &\qquad 
 +  \ldots + \{ F_k, \dotsm ,F_{2k-2}, \{ F_1, \dotsm , F_{k-1}, F_{2k-1} \}^{\rm(NJ)}\}^{\rm(NJ)} \\ &\qquad \qquad 
 =  \{ F_1, \dotsm , F_{k-1}, \{ F_k, \dotsm , F_{2k-1} \}^{\rm(NJ)}\}^{\rm(NJ)}.
 \end{split}
\end{equation} 

\item[\textbf{NJ2.}] 
\label{item:NJ1} First-order linear differential operator
\begin{equation}
	\begin{split}
\{H \cdot F_1, F_2, \dotsm ,F_{k} \}^{\rm(NJ)}&=
H \cdot \{F_1, F_2, \dotsm , F_{k} \} ^{\rm(NJ)}+
\{ H, F_2, \dotsm, F_{k} \} ^{\rm(NJ)}\cdot F_1 \\& \qquad -H\cdot F_1 \cdot\{1, F_2, \dotsm, F_{k} \}^{\rm(NJ)} .
 \end{split}
\end{equation}
	
\end{axioms}
A manifold $\mathcal{P}$ equipped with a Nambu--Jacobi $k$-bracket is called $k$-Nambu--Jacobi manifold, see, for example,  \cite{GrabMarm99,Hagiwara2004}. Notice that, for $k=2$, one arrives at a Jacobi manifold.

\textbf{Multivector Field Realization.} Consider a $k$-Nambu--Jacobi manifold $\mathcal{P}$. Referring to this structure, one can define a $(k-1)$-vector field $\mathcal{E}$ and a $k$-vector field $\eta$ as follows:
\begin{equation}\label{NJ-bracket-kth}
\begin{split}
\mathcal{E}^{[k-1]}(dF_2,\dotsm,dF_k)&:=\{1,F_2,F_3,\dotsm,F_k\}^{\rm(NJ)},
\\
\eta^{[k]} (dF_1,\dotsm,dF_k)&:=\{ F_1,\dotsm,F_k \} ^{\rm(NJ)}+ \sum_{i=1}^k (-1)^i F_i~ \mathcal{E}^{[k-1]} (dF_1,\dotsm,\widehat{dF_i},\dotsm,dF_k),
\end{split}
\end{equation}
where $\widehat{dF_i}$ denotes that it is omitting. Accordingly, we denote a $k$-Nambu--Jacobi manifold by a three-tuple $(\mathcal{P},\eta^{[k]},\mathcal{E}^{[k-1]})$. 
In light of the multivector fields, the Nambu--Jacobi $k$-bracket is written as
\begin{equation}\label{bjk}
    \{ F_1,\dotsm,F_k \}^{\rm(NJ)}= \eta^{[k]} (dF_1,\dotsm,dF_k) + \sum_{i=1}^k (-1)^{i+1} F_i~ \mathcal{E}^{[k-1]} (dF_1,\dotsm,\widehat{dF_i},\dotsm,dF_k).
\end{equation}
It is evident that if $\mathcal{E}^{[k-1]}$ is identically zero then $\eta^{[k]}$ determines a Nambu--Poisson $k$-bracket  whereas if $\eta^{[k]}$ is identically zero then $\mathcal{E}^{[k-1]}$ becomes an Nambu--Poisson $(k-1)$-bracket.



    


\textbf{Contraction of $k$-brackets.}  
Let $(\mathcal{P},\eta^{[k]},\mathcal{E}^{[k-1]})$ be a $k$-Nambu--Jacobi manifold equipped with the $k$-bracket \eqref{bjk}.  
By fixing the last entry, one obtains a $(k-1)$-bracket defined by   
\begin{equation}\label{bjk1} 
     \{F_1,F_2,\dotsm,F_{k-1}\}^{F_k}   := \{F_1,F_2,\dotsm,F_{k-1},F_k\}^{\rm{(NJ)}} . 
\end{equation}
It can be shown that the $(k-1)$-bracket $\{\bullet,\dots,\bullet\}^{F_k}$ is again a Nambu--Jacobi bracket.  
The converse assertion also holds: if the $(k-1)$-bracket obtained by fixing the last entry is a Nambu--Jacobi bracket of order $(k-1)$, then the original $k$-bracket is a Nambu--Jacobi bracket as well. For a similar discussion, where the first entry of the $k$-bracket is fixed, we refer the reader to \cite{GrabMarm99}.

In a similar fashion, by fixing the last two entries of the $k$-bracket \eqref{bjk} (or, equivalently, the last entry of the $(k-1)$-bracket \eqref{bjk1}), 
one obtains a Nambu--Jacobi $(k-2)$-bracket:
\begin{equation}\label{bjk2}
   \{F_1,F_2,\dotsm,F_{k-2}\}^{F_{k-1},F_k} =  \{F_1,F_2,\dotsm,F_{k-2},F_{k-1}\}^{F_k}=  \{F_1,F_2,\dotsm,F_{k-2},F_{k-1},F_k\}^{\rm{(NJ)}}.  
\end{equation}
By iteration, one can contract to arrive at a Jacobi bracket as
\begin{equation}\label{bjkjac}
   \{F_1,F_2\}^{F_3,\dotsm,F_k} =  \{F_1,F_2,F_3,\dotsm,F_k\}^{\rm{(NJ)}} . 
\end{equation}
starting from a Nambu--Jacobi $k$-bracket. 

According to \eqref{bjk}, the Nambu--Jacobi $k$-bracket can be expressed in terms of a pair consisting of a $k$-vector field and a $(k-1)$-vector field, $(\eta^{[k]},\mathcal{E}^{[k-1]})$.  
On the other hand, in view of \eqref{bra-Jac}, there must exist a pair formed by a bivector and a vector field, $(\Lambda,Z) = (\eta^{[2]},\mathcal{E}^{[1]})$, which determines the Jacobi bracket \eqref{bjkjac}. 
The following theorem establishes the precise relationship between the $k$-Nambu--Jacobi pair $(\eta^{[k]},\mathcal{E}^{[k-1]})$ and the Jacobi pair $(\eta^{[2]},\mathcal{E}^{[1]})$. In this way, the Jacobi structure can explicitly be expressed in terms of the original $k$-vector field $\eta^{[k]}$ and $(k-1)$-vector field $\mathcal{E}^{[k-1]}$.

\begin{theorem}\label{beris}
Assume that a Nambu--Jacobi $k$-bracket is determined by the pair of multivector fields $(\eta^{[k]},\mathcal{E}^{[k-1]})$ as in \eqref{bjk}.  
Then, the Jacobi bracket \eqref{bjkjac} obtained through the iterative contraction of the $k$-bracket takes the form
\begin{equation}\label{kontra}
   \{F_1,F_2\}^{F_3,\dotsm,F_k} 
   = \eta^{[2]}(dF_1,dF_2) 
   + F_1\,\mathcal{E}^{[1]}(F_2) 
   - F_2\,\mathcal{E}^{[1]}(F_1),
\end{equation}
where the associated bivector $\eta^{[2]}$ and vector field $\mathcal{E}^{[1]}$ are given by
\begin{equation}\label{atak}
 \begin{split}
     \eta^{[2]} &= \iota_{dF_3\wedge\cdots\wedge dF_k}\eta^{[k]} 
     + \sum_{j=0}^{k-3}(-1)^{k+1-j}F_{k-j}\,
       \iota_{dF_3\wedge\cdots\wedge \widehat{dF_{k-j}}\wedge\cdots\wedge dF_k}
       \mathcal{E}^{[k-1]}, \\[0.7em]
     \mathcal{E}^{[1]} &= \iota_{dF_3\wedge\cdots\wedge dF_k}\mathcal{E}^{[k-1]} .
 \end{split}
\end{equation}
\end{theorem}

\begin{proof}
We start with the first contraction by substituting the explicit expression \eqref{bjk} of the $k$-bracket into \eqref{bjk1}, and compute
\begin{equation}\label{NJ-contraction}
\begin{split}
     \{F_1,\dots,F_{k-1}\}^{F_k} 
     &= \{F_1,\dots,F_k\}^{\rm{(NJ)}} \\
    &= \eta^{[k]} (dF_1,\dotsm,dF_k) 
    + \sum_{i=1}^k (-1)^{i+1} F_i~ 
    \mathcal{E}^{[k-1]} (dF_1,\dotsm,\widehat{dF_i},\dotsm,dF_k) \\
    &= \iota_{dF_k} \eta^{[k]} (dF_1,\dotsm,dF_{k-1}) \\
    &\quad+ \sum_{i=1}^{k-1} (-1)^{i+1} F_i~ 
    \iota_{dF_k}\mathcal{E}^{[k-1]} (dF_1,\dotsm,\widehat{dF_i},\dotsm,dF_{k-1}) \\
    &\quad + (-1)^{k+1}F_k\mathcal{E}^{[k-1]}(dF_1,\dotsm,dF_{k-1})\\
    &= \Big(\iota_{dF_k} \eta^{[k]} + (-1)^{k+1}F_k\mathcal{E}^{[k-1]} \Big)(dF_1,\dotsm,dF_{k-1}) \\
    &\quad + \sum_{i=1}^{k-1} (-1)^{i+1} F_i~ (\iota_{dF_k}\mathcal{E}^{[k-1]} )(dF_1,\dotsm,\widehat{dF_i},\dotsm,dF_{k-1}).
\end{split}
\end{equation}
Referring to the last line, we recognize that the $(k-1)$-bracket $\{\bullet,\dots,\bullet\}^{F_k}$ can be written in terms of a $(k-1)$-vector field $\eta^{[k-1]}$ and a $(k-2)$-vector field $\mathcal{E}^{[k-2]}$ as
\begin{equation}\label{NJ-contraction-new}
\begin{split}
    \{F_1,\dots,F_{k-1}\}^{F_k} 
    &= \eta^{[k-1]}(dF_1,\dotsm,dF_{k-1}) \\
    &\quad + \sum_{i=1}^{k-1} (-1)^{i+1} F_i~ 
    \mathcal{E}^{[k-2]} (dF_1,\dotsm,\widehat{dF_i},\dotsm,dF_{k-1}),
\end{split}
\end{equation}
where
\begin{equation}\label{gomis}
    \eta^{[k-1]} = \iota_{dF_k} \eta^{[k]} + (-1)^{k+1}F_k\mathcal{E}^{[k-1]}, 
    \qquad 
    \mathcal{E}^{[k-2]} = \iota_{dF_k}\mathcal{E}^{[k-1]}.
\end{equation}
Hence, the pair $(\eta^{[k]},\mathcal{E}^{[k-1]})$ defines a Nambu--Jacobi structure of order $k$ if and only if the pair $(\eta^{[k-1]},\mathcal{E}^{[k-2]})$ defines a Nambu--Jacobi structure of order $(k-1)$. Compactly,
\begin{equation}
    (\eta^{[k]},\mathcal{E}^{[k-1]})
    \;\mapsto\;
    (\eta^{[k-1]},\mathcal{E}^{[k-2]}),
\end{equation}
with the new multivectors given by \eqref{gomis}.

Applying the same procedure to the $(k-1)$-bracket \eqref{NJ-contraction-new} by fixing the last entry yields
\begin{equation}\label{NJ-contraction-order}
    \begin{split}
        \{F_1,\dots,F_{k-2}\}^{F_{k-1},F_k} 
        &= \Big(\iota_{dF_{k-1}}\eta^{[k-1]} + (-1)^{k}F_{k-1}\mathcal{E}^{[k-2]}\Big)(dF_1,\dotsm,dF_{k-2}) \\
        &\quad + \sum_{i=1}^{k-2} (-1)^{i+1} F_i~ 
        \iota_{dF_{k-1}}\mathcal{E}^{[k-2]} (dF_1,\dotsm,\widehat{dF_i},\dotsm,dF_{k-2}).
    \end{split}
\end{equation}
Thus, the $(k-2)$-bracket is a Nambu--Jacobi bracket of order $(k-2)$ if and only if the $(k-1)$-bracket is Nambu--Jacobi. The corresponding multivector fields are
\begin{equation}
\begin{split}
     \eta^{[k-2]} 
     &= \iota_{dF_{k-1}}\eta^{[k-1]} + (-1)^{k}F_{k-1}\mathcal{E}^{[k-2]} \\
     &= \iota_{dF_{k-1}\wedge dF_k} \eta^{[k]} 
     + \Big((-1)^{k+1}F_k\iota_{dF_{k-1}} + (-1)^kF_{k-1}\iota_{dF_k}\Big)\mathcal{E}^{[k-1]}, \\[0.5em]
    \mathcal{E}^{[k-3]} &= \iota_{dF_{k-1}}\mathcal{E}^{[k-2]} 
    = \iota_{dF_{k-1}}\iota_{dF_k}\mathcal{E}^{[k-1]}.
\end{split}
\end{equation}

Continuing recursively, one obtains the general $(k-l)$-bracket
\begin{equation}\label{fener}
\begin{split}
    \{F_1,\dots,F_{k-l}\}^{F_{k-l+1},\dots,F_k} 
    &= \eta^{[k-l]}(dF_1,\dotsm,dF_{k-l}) \\
    &\quad + \sum_{i=1}^{k-l} (-1)^{i+1} F_i~ 
    \mathcal{E}^{[k-l-1]}(dF_1,\dotsm,\widehat{dF_i},\dotsm,dF_{k-l}),
\end{split}
\end{equation}
with multivectors
\begin{equation}\label{fener+}
\begin{split}
    \eta^{[k-l]} &= \iota_{dF_{k-l+1}\wedge\cdots\wedge dF_k}\eta^{[k]} 
    + \sum_{j=0}^{l-1}(-1)^{k+1-j}F_{k-j}\,
    \iota_{dF_{k-l+1}\wedge\cdots\wedge \widehat{dF_{k-j}}\wedge\cdots\wedge dF_k}\mathcal{E}^{[k-1]}, \\[0.5em]
    \mathcal{E}^{[k-l-1]} &= \iota_{dF_{k-l+1}\wedge\cdots\wedge dF_k}\mathcal{E}^{[k-1]}.
\end{split}
\end{equation}

Finally, by taking $l=k-2$ in \eqref{fener}--\eqref{fener+}, one arrives at the Jacobi ($2$-)bracket \eqref{kontra}, determined by the bivector $\eta^{[2]}$ and the vector field $\mathcal{E}^{[1]}$ as in \eqref{atak}. This completes the proof.

\end{proof}

As a particular instance, in the following corollary, we write Theorem \ref{beris} for $k=3$, which establishes the connection between $3$-Nambu--Jacobi manifolds and Jacobi manifolds. We will use this in upcoming analysis.

\begin{corollary}\label{babanis}
Consider the following Nambu--Jacobi $3$-bracket
\begin{equation}
\begin{split}
   \{F_1,F_2,F_3\}^{\rm{(NJ)}} & =
         \eta^{[3]}(dF_1,dF_2,dF_3) +F_1\mathcal{E}^{[2]}(dF_2,dF_3) -F_2\mathcal{E}^{[2]}(dF_1,dF_3) \\
         & \qquad \qquad +F_3\mathcal{E}^{[2]}(dF_1,dF_2)  
\end{split}
\end{equation}
generated by the Nambu--Jacobi pair  $(\eta^{[3]},\mathcal{E}^{[2]})$. 
Then, the Jacobi bracket   obtained through the contraction is
\begin{equation}\label{kontra1}
     \{F_1,F_2\}^{F_3} 
   = \Lambda(dF_1,dF_2) 
   + F_1Z(F_2) 
   - F_2Z(F_1),
\end{equation}
where the associated Jacobi pair $(\Lambda,Z)$ is given by
\begin{equation}\label{atakta} 
     \Lambda  = \iota_{dF_3}\eta^{[3]} 
     + F_3    \mathcal{E}^{[2]},\qquad  Z= \iota_{dF_3}\mathcal{E}^{[2]}. 
\end{equation}    
\end{corollary}

\textbf{Nambu--Jacobi Hamiltonian Vector Field.} Start with a $k$-Nambu--Jacobi manifold $(\mathcal{P}, \eta, \mathcal{E})$. For a collection $(H_{1}, \dotsc, H_{k-1})$ of Hamiltonian functions, the associated Nambu--Jacobi Hamiltonian vector field is defined by \cite{IbanezLopezMarreroPadron2001}
\begin{equation}\label{NJHVF}
X_{H_{1} \cdots H_{k-1}}
= X_{H_{1} \cdots H_{k-1}}^{\eta}
+ \sum_{i=1}^{k-1}(-1)^{i} H_{i}
X_{H_{1} \dotsm \widehat{H_i} \dotsm H_{k-1}}^{\mathcal{E}},
\end{equation}
where the components $X_{H_{1} \cdots H_{k-1}}^{\eta}$ and $X_{H_{1} \dotsm \widehat{H_i} \dotsm H_{k-1}}^{\mathcal{E}}$ are obtained by contraction of the multivector fields with the Hamiltonians:
\begin{equation}
X_{H_{1} \cdots H_{k-1}}^{\eta}
= \iota_{dH_{1} \wedge\cdots\wedge dH_{k-1}} \eta^{[k]},
\qquad
X_{H_{1} \dotsm \widehat{H_i} \dotsm H_{k-1}}^{\mathcal{E}}
= \iota_{dH_{1}\wedge \dotsm \wedge\widehat{dH_i}\wedge\dotsm dH_{k-1}} \mathcal{E}^{[k-1]}.
\end{equation}
Accordingly, the Nambu--Jacobi Hamiltonian dynamics associated with the Hamiltonian functions $(H_1,\dots,H_{k-1})$ is given by
\begin{equation}\label{ruzgar}
    \frac{dx}{dt}= \iota_{dH_{1} \wedge\cdots\wedge dH_{k-1}} \eta^{[k]} +  \sum_{i=1}^{k-1}(-1)^{i} H_{i}\,\iota_{dH_{1}\wedge \dotsm \wedge\widehat{dH_i}\wedge\dotsm dH_{k-1}} \mathcal{E}^{[k-1]}.
\end{equation}

Two immediate consequences follow. If $\mathcal{E}^{[k-1]}$ is zero, then $X_{H_{1} \dotsm \widehat{H_i} \dotsm H_{k-1}}^{\mathcal{E}}$ vanishes identically, and one is left only with $X_{H_{1} \cdots H_{k-1}}^{\eta}$. This is precisely the Hamiltonian vector field \eqref{NP-HVF} on the $k$-Nambu--Poisson manifold $(\mathcal{P},\eta^{[k]})$. 
Conversely, if $X_{H_{1} \cdots H_{k-1}}^{\eta}$ is zero, then the Hamiltonian vector field reduces to $X_{H_{1} \dotsm H_{k-2}}^{\mathcal{E}}$, which coincides with the Hamiltonian vector field on the $(k-1)$-Nambu--Poisson manifold $(\mathcal{P},\mathcal{E}^{[k-1]})$.

The components $X_{H_{1} \cdots H_{k-1}}^{\eta}$ and $X_{H_{1} \cdots H_{k-2}}^{\mathcal{E}}$ satisfy the followings:
\begin{equation}
\begin{split}
    \mathcal{L}_{X_{H_{1} \cdots H_{k-2}}^{\mathcal{E}}} \eta^{[k]} & =0,
    \\ 
    \mathcal{L}_{X_{H_{1} \cdots H_{k-1}}^{\eta}} \mathcal{E}^{[k-1]} &=(-1)^{k-1} \iota_{\mathrm{~d}\mathcal{E}(d H_{1}, \dotsm, dH_{k-1})} \eta^{[k]} .
\end{split}
\end{equation}
Here, the first equality expresses that the multivector field $ \eta^{[k]}$ is invariant under the flow generated by the component of the Hamiltonian vector field associated with $\mathcal{E}^{[k-1]} $. 
The second condition, 
describes how $\mathcal{E}^{[k-1]} $ transforms under the flow of the component of the Hamiltonian vector field associated with $ \eta^{[k]}$. More precisely, the Lie derivative of $\mathcal{E}^{[k-1]} $ along $X^{\eta}$ does not vanish but is governed by the contraction of $ \eta^{[k]}$ with the differential of the Hamiltonian functions evaluated on $\mathcal{E}^{[k-1]} $. Together, these compatibility conditions encode the mutual consistency of the pair $(\eta^{[k]},\mathcal{E}^{[k-1]})$ that defines the Nambu--Jacobi structure. 


\subsection{3-Nambu--Jacobi Manifolds and Jacobi Bi-Hamiltonian Dynamics}\label{ebru}

Our discussion begins with the notion of bi-Hamiltonian dynamics on Jacobi manifolds endowed with compatible Jacobi structures. We then demonstrate how a $3$-Nambu--Jacobi structure naturally decomposes into two compatible Jacobi structures.

\noindent	\textbf{Jacobi Bi-Hamiltonian Dynamics.} Assume that there are two Jacobi structures $(\Lambda^1, Z^1)$ and $(\Lambda^2, Z^2)$ defined on a manifold $\mathcal{P}$. Referring to \eqref{bra-Jac}, we denote their corresponding Jacobi brackets  by $\{\bullet,\bullet\}^{1}$ and $\{\bullet,\bullet\}^{2}$, respectively. These structures are compatible if the sum of the tensor fields also defines a Jacobi structure, \cite{MaMoPa,Nunes98}. That is, if the pair 
\begin{equation}\label{cond-Jac-comp}
\left(\Lambda^1+\Lambda^2, Z^1+Z^2\right)
\end{equation}
satisfies the conditions given in \eqref{ident-Jac}. The two necessary and sufficient conditions for the compatibility of the two Jacobi structures are 
\begin{equation}\label{dogu}
\begin{split}
&\left[\Lambda^1, \Lambda^2\right] =Z^1 \wedge \Lambda^2+Z^2 \wedge \Lambda^1,
\\
&\left[Z^1, \Lambda^2\right] +\left[Z^2, \Lambda^1\right] =0.
\end{split}
\end{equation}
Observe that if $Z^1$ and $Z^2$ both vanish, then one recovers the case of two compatible Poisson structures summarized in Subsection \ref{sec:BiH}. In fact, under this assumption, the first condition \eqref{dogu} reduces to \eqref{hopa}, while the second condition is automatically satisfied.

For the Jacobi pair $\left(\Lambda^1+\Lambda^2, Z^1+Z^2\right)$ obtained from a compatible Jacobi pairs, the  Jacobi bracket is computed to be 
\begin{equation}
\{F,H\}  =\{F,H\} ^{1}+\{F,H\} ^{2}.
 \end{equation}

A dynamical system is called Jacobi bi-Hamiltonian if there exist two non-trivial compatible Jacobi structures $(\Lambda^1, Z^1)$ and $(\Lambda^2, Z^2)$ together with Hamiltonian functions $H_2$ and $H_1$ such that 
\begin{equation}
\frac{dx}{dt}= (\Lambda^1)^\sharp(d H_2)-H_2 Z^1,\qquad  \frac{dx}{dt}= (\Lambda^2)^\sharp(d H_1)-H_1 Z^2.
\end{equation}
Notice that for the Poisson case, we have the bi-Hamiltonian dynamics in \eqref{Biham-Eq}.

\noindent \textbf{Jacobi Bi-Hamiltonian Dynamics on $3$-Nambu--Jacobi Manifolds.} We examine the $k=3$ case for the Nambu--Jacobi manifold theory in Subsection \ref{sec:NJ}. In this case, on a manifold $\mathcal{P}$, we have a 
Nambu--Jacobi $3$-bracket 
\begin{equation}\label{NJ-3}
\{ \bullet,\bullet,\bullet \}^{\rm(NJ)}~:C^\infty(\mathcal{P})\wedge C^\infty(\mathcal{P})\wedge C^\infty(\mathcal{P}) \longrightarrow C^\infty(\mathcal{P}).
\end{equation}
Here, the fundamental identity retains the same structural form as in \eqref{3-Lie-cond2}, whereas the requirement of being a first-order linear differential operator takes the form
\begin{equation}
	\begin{split}
\{HF_1, F_2, F_3 \}^{\rm(NJ)}&=
H \{F_1, F_2, F_3 \} ^{\rm(NJ)}+
\{ H, F_2, F_3 \} ^{\rm(NJ)}F_1 - HF_1\{1,F_2, F_3\}^{\rm(NJ)} ,
 \end{split}
\end{equation}
For $k=3$, we can write the multivector fields realization in \eqref{NJ-bracket-kth} as  
\begin{equation}
   \begin{split}
        \mathcal{E}^{[2]}(dF_2,dF_3) & = \{1,F_2,F_3\}^{\rm(NJ)}, \\
         \eta^{[3]}(dF_1,dF_2,dF_3)& = \{F_1,F_2,F_3\}^{\rm(NJ)} - F_1\mathcal{E}^{[2]}(dF_2,dF_3) + F_2\mathcal{E}^{[2]}(dF_1,dF_3) \\ &\qquad \qquad- F_3\mathcal{E}^{[2]}(dF_1,dF_2).
   \end{split}
\end{equation}
Hence, the explicit form of the  Nambu--Jacobi $3$-bracket by a Nambu--Jacobi pair $(\eta^{[3]}, \mathcal{E}^{[2]})$ is written as
\begin{equation}\label{br-3-NJ}
 \begin{split}
     \{F ,H_1,H_2\} ^{\rm(NJ)}&= \eta^{[3]}(dF ,dH_1,dH_2) + F \mathcal{E}^{[2]}(dH_1,dH_2) - H_1\mathcal{E}^{[2]}(dF ,dH_2) \\ &\qquad \qquad + H_2\mathcal{E}^{[2]}(dF ,dH_1).
      \end{split}
\end{equation}

Fixing the second Hamiltonian function $H_2$ in the Nambu--Jacobi $3$-bracket
\eqref{br-3-NJ} induces the $2$-bracket
\begin{equation}\label{jack-2}
\begin{split}
\{F,H_1\}^{2}
&= \big(\iota_{dH_2}\eta^{[3]} + H_2\mathcal{E}^{[2]}\big)(dF,dH_1)
 + F\,\iota_{dH_2}\mathcal{E}^{[2]}(dH_1)
 - H_1\,\iota_{dH_2}\mathcal{E}^{[2]}(dF).
\end{split}
\end{equation}
By Corollary~\ref{babanis}, this contraction defines a Jacobi bracket of the
form \eqref{bra-Jac}, with associated Jacobi pair
\begin{equation}\label{dogu-1}
\Lambda^{2} = \iota_{dH_2}\eta^{[3]} + H_2\mathcal{E}^{[2]},
\qquad
Z^{2} = \iota_{dH_2}\mathcal{E}^{[2]}.
\end{equation}

Similarly, fixing the first Hamiltonian function $H_1$ in the
Nambu--Jacobi $3$-bracket \eqref{br-3-NJ} yields the $2$-bracket
\begin{equation}\label{jack-1}
\{F,H_2\}^{1}
= -\big(\iota_{dH_1}\eta^{[3]} + H_1\mathcal{E}^{[2]}\big)(dF,dH_2)
 - F\,\iota_{dH_1}\mathcal{E}^{[2]}(dH_2)
 + H_2\,\iota_{dH_1}\mathcal{E}^{[2]}(dF).
\end{equation}
Again by Corollary~\ref{babanis}, the corresponding Jacobi structure is
determined by the pair
\begin{equation}\label{dogu-2}
\Lambda^{1} = -\iota_{dH_1}\eta^{[3]} - H_1\mathcal{E}^{[2]},
\qquad
Z^{1} = -\iota_{dH_1}\mathcal{E}^{[2]},
\end{equation}
and the bracket \eqref{jack-1} is a Jacobi bracket of the form
\eqref{bra-Jac}.

Thus, by contracting the Nambu--Jacobi $3$-bracket \eqref{br-3-NJ} with $H_2$ and $H_1$, we obtain the Jacobi structures $(\Lambda^2,Z^2)$ in \eqref{dogu-1} and $(\Lambda^1,Z^1)$ in \eqref{dogu-2}, respectively. The following theorem establishes that these two Jacobi structures are compatible. We cite \cite{IbanezLeonPadron98} not only for the following theorem but also for additional relationships between higher order brackets and compatible Jacobi pairs.
\begin{theorem}
The Jacobi pairs  $(\Lambda^2,Z^2)$ in \eqref{dogu-1} and $(\Lambda^1,Z^1)$ in \eqref{dogu-2} obtained by the contraction of the Nambu--Jacobi $3$-bracket are compatible. 
\end{theorem}

Referring to \eqref{ruzgar}, the Nambu--Jacobi Hamiltonian dynamics generated by a pair $(H_1,H_2)$ of Hamiltonian functions is computed to be
\begin{equation}\label{ruzgar-3}
\begin{split}
    \frac{dx}{dt} &= \iota_{dH_{1} \wedge dH_{2}} \eta^{[3]} - H_1 \iota_{dH_2} \mathcal{E}^{[2]}+  H_2 \iota_{dH_1} \mathcal{E}^{[2]} \\
   &=  \iota_{dH_{1}} \big( \iota_{dH_{2}} \eta^{[3]} + H_2 \mathcal{E}^{[2]} \big) - H_1 \iota_{dH_2} \mathcal{E}^{[2]}
   \\
   &= - \iota_{dH_{2}}\big( \iota_{dH_{1}} \eta^{[3]} +H_1 \mathcal{E}^{[2]}\big) +  H_2 \iota_{dH_1} \mathcal{E}^{[2]}.
\end{split}
\end{equation}
Here, the second line is manifesting the existence of the Jacobi pair $(\Lambda^2,Z^2)$ in \eqref{dogu-1} while the third line is exhibiting the Jacobi pair $(\Lambda^1,Z^1)$ in \eqref{dogu-2}. Hence,  the Nambu Hamiltonian system \eqref{ruzgar-3} can be expressed in the Jacobi bi-Hamiltonian form 
\begin{equation}
\begin{split}
\frac{dx}{dt} &= (\Lambda^2)^\sharp(d H_1)-H_1 Z^2 = \big(\iota_{dH_2}\eta^{[3]}+ H_2\mathcal{E}^{[2]}\big) (d H_1) - H_1 \iota_{dH_2}\mathcal{E}^{[2]}
\\
\frac{dx}{dt} &= (\Lambda^1)^\sharp(d H_2)-H_2 Z^1 =  \big(-\iota_{dH_1}\eta^{[3]}- H_1\mathcal{E}^{[2]}\big)(dH_2) - H_2 (-\iota_{dH_1}\mathcal{E}^{[2]}).
\end{split}
\end{equation}

\section{Locally Conformal Nambu--Poisson Setting}\label{sec:LCNPois}

\subsection{Local Conformality in Nambu--Poisson Geometry}\label{sec:LCNP}

In Subsection \ref{sec:LCP}, we examined the globalization problem in the Poisson setting. Specifically, we introduced the concept of locally conformally Poisson manifolds, which are obtained by gluing together local Poisson charts that agree on their overlaps up to a conformal factor. In this subsection, we extend this approach to the Nambu--Poisson setting, where the goal is to glue local Nambu--Poisson charts that satisfy the locally conformal type of compatibility on overlaps. 


Assume a manifold $\mathcal{P}$ with an open covering 
\begin{equation}
    \mathcal{P} = \bigsqcup_\alpha U_\alpha
\end{equation}
so that each local chart admits local $k$-Nambu Poisson structures: 
\begin{equation}\label{cici}
(U_\alpha,\eta^{[k]}_\alpha),\quad (U_\beta,\eta^{[k]}_\beta), \quad (U_\gamma,\eta^{[k]}_\gamma), \quad \dots.
\end{equation} 
On the intersections of any two open subsets with nonempty overlap, the local Nambu--Poisson structures may fail to satisfy the compatibility conditions required to define a global tensor directly.  

As a locally conformal approach to the globalization problem, let us assume the existence of a family of local functions defined on each open subset:
\begin{equation}
\sigma_\alpha: U_\alpha \to \mathbb{R}, \quad 
\sigma_\beta: U_\beta \to \mathbb{R}, \quad 
\sigma_\gamma: U_\gamma \to \mathbb{R}, \quad \dots ,
\end{equation}
whose exterior derivatives agree on the overlaps. That is, if one has a non-trivial intersection, say $U_\alpha \cap U_\beta \neq \emptyset$, then
\begin{equation}\label{dsigma}
d\sigma_\alpha = d\sigma_\beta.
\end{equation}
This equality holds analogously on all nonempty intersections. 
In light of Poincar\'{e} lemma, this assumption determines the Lee form $\theta$  whose local picture on $U_\alpha$ is denoted as 
\begin{equation} \label{titaa}
    \theta\vert_\alpha = d\sigma_\alpha.
\end{equation} 
 
Further, we take that, by means of the conformal functions, for any two overlapping subsets $ U_\alpha \cap U_\beta \neq \emptyset $, the corresponding local Nambu--Poisson structures satisfy
\begin{equation}\label{usta}
    e^{-(k-1)\sigma_\alpha}\,\eta^{[k]}_\alpha \;=\; e^{-(k-1)\sigma_\beta}\,\eta^{[k]}_\beta .
\end{equation}
This relation holds analogously on all nonempty intersections. It is immediate to see that the conformal equalities in \eqref{usta} determine the transition scalars
\begin{equation} \label{transition1}
\kappa_{\beta \alpha} = \frac{e^{(k-1)\sigma_\alpha}}{e^{(k-1)\sigma_\beta}} = e^{-(k-1)(\sigma_\beta - \sigma_\alpha)},
\end{equation}
satisfying the cocycle condition
\begin{equation}\label{cocycle1}
\kappa_{\beta \alpha} \, \kappa_{\alpha \gamma} = \kappa_{\beta \gamma}.
\end{equation}
Accordingly,  the local Nambu--Poisson structures $ \{\eta^{[k]}_\alpha\} $ can be glued up to a line bundle valued $k$-multivector field.  

Even though the family $ \{\eta^{[k]}_\alpha\} $ of local Nambu--Poisson tensors listed in \eqref{cici} cannot be glued together directly into a real-valued global object, motivated by \eqref{LCP-bi-vector}, we introduce the rescaled $k$-vector fields
\begin{equation}\label{gs}
\eta^{[k]}|_\alpha := e^{-(k-1)\sigma_\alpha}\,\eta^{[k]}_\alpha.
\end{equation}
These rescaled tensors satisfy the compatibility condition
\begin{equation}\label{bibi}
\eta^{[k]}|_\alpha = \eta^{[k]}|_\beta ,
\end{equation}
and hence define a globally well-defined $k$-vector field $ \eta^{[k]} $, whose local expression, for example, on $ U_\alpha $ is $ \eta^{[k]}|_\alpha $.

In analogy with the Poisson case, we refer to the construction above---namely, the collection $\{(U_\alpha,\eta_\alpha^{[k]})\}$
of local $k$-Nambu--Poisson charts together with the set of conformal functions 
$\{\sigma_\alpha\}$ satisfying the compatibility condition \eqref{usta}---as a  locally conformally $k$-Nambu--Poisson manifold. We denote a locally conformally $k$-Nambu--Poisson manifold by
\begin{equation}
(\mathcal{P},U_\alpha,\eta_\alpha^{[k]},\sigma_\alpha).
\end{equation}

 It is important to note that the global $k$-vector field $\eta^{[k]}$ introduced in \eqref{bibi} does not generally retain the Nambu--Poisson property. However, as we will show in the following subsections, every locally conformally $k$-Nambu--Poisson manifold naturally defines a $k$-Nambu--Jacobi manifold. Our analysis proceeds inductively: we first establish the case $k=3$ as the base step, and then extend the construction to arbitrary $k$ in the subsequent subsections.  The explicit treatment of the $k=3$ case, together with its Hamiltonian structure, will also serve as a foundation for our later discussion of bi-Hamiltonian dynamics in the locally conformal setting.



\subsection{Analyzing Locally Conformally 3-Nambu--Poisson Manifolds}\label{subsec:LC-3-NP}

We now specialize the general construction of locally conformally
$k$-Nambu--Poisson manifolds to the case $k=3$.
Accordingly, let
\begin{equation}\label{LC3NPLocal}
(\mathcal{P},U_\alpha,\eta_\alpha^{[3]},\sigma_\alpha)
  \end{equation}
be a locally conformally $3$-Nambu--Poisson manifold in the sense defined above.
 Referring to \eqref{gs}, we introduce a new family of local $3$-vector fields by
\begin{equation}\label{local-3-vector}
\eta ^{[3]}\vert_\alpha= e^{-2\sigma_\alpha}
\eta_\alpha ^{[3]}
\end{equation}
which glue together to define a global $3$-vector field $\eta^{[3]}$.

To ensure compatibility with the locally conformal structure, we define
local functions $F^i_\alpha$ through the transformation rule
\begin{equation}\label{glue-func}
F^i\big|_\alpha = e^{\sigma_\alpha} F^i_\alpha , \qquad i=1,2,3,
\end{equation}
which allows the functions $F^i\big|_\alpha$ to glue into global
functions $F^i$ on $\mathcal{P}$.
We compute the local Nambu--Poisson $3$-bracket $\{\bullet,\bullet,\bullet\}_\alpha$ determined by the  local Nambu--Poisson $3$-vector field $\eta^{[3]}_\alpha$ in terms $\eta^{[3]}\vert_\alpha$ in \eqref{local-3-vector} and $F^i\vert_\alpha$ in \eqref{glue-func} as 
 \begin{equation}\label{glue-3-NP-brac}
 	\begin{split}
 		\{F_\alpha,F^2_\alpha,F^3_\alpha\}_\alpha &= \eta^{[3]}_\alpha(dF^1_\alpha,dF^2_\alpha,dF^3_\alpha) \\
 		&= e^{2\sigma_\alpha}\eta^{[3]}\vert_\alpha(d(e^{-\sigma_\alpha}F^1\vert_\alpha),d(e^{-\sigma_\alpha}F^2\vert_\alpha),d(e^{-\sigma_\alpha}F^3\vert_\alpha)) \\
 		&= e^{-\sigma_\alpha}\eta^{[3]}\vert_\alpha(dF^1\vert_\alpha - F^1\vert_\alpha d\sigma_\alpha,dF^2\vert_\alpha - F^2\vert_\alpha d\sigma_\alpha,dF^3\vert_\alpha - F^3\vert_\alpha d\sigma_\alpha).
 	\end{split}
 \end{equation}
Multiplying both sides of \eqref{glue-3-NP-brac} by $e^{\sigma_\alpha}$ yields
\begin{equation}\label{glued-3NP}
 	\begin{split}
 		e^{\sigma_\alpha}\{F^1_\alpha,F^2_\alpha,F^3_\alpha\}_\alpha &= \eta^{[3]}\vert_\alpha(dF^1\vert_\alpha,dF^2\vert_\alpha,dF^3\vert_\alpha) - F^1\vert_\alpha\eta^{[3]}\vert_\alpha(d\sigma_\alpha,dF^2\vert_\alpha,dF^3\vert_\alpha) \\
 		&\quad - F^2\vert_\alpha\eta^{[3]}\vert_\alpha(dF^1\vert_\alpha,d\sigma_\alpha,dF^3\vert_\alpha) - F^3\vert_\alpha\eta^{[3]}\vert_\alpha(dF^1\vert_\alpha,dF^2\vert_\alpha,d\sigma_\alpha) \\
        &= \eta^{[3]}\vert_\alpha(dF^1\vert_\alpha,dF^2\vert_\alpha,dF^3\vert_\alpha) - F^1\vert_\alpha\eta^{[3]}\vert_\alpha(dF^2\vert_\alpha,dF^3\vert_\alpha,d\sigma_\alpha) \\
 		&\quad + F^2\vert_\alpha\eta^{[3]}\vert_\alpha(dF^1\vert_\alpha,dF^3\vert_\alpha,d\sigma_\alpha) - F^3\vert_\alpha\eta^{[3]}\vert_\alpha(dF^1\vert_\alpha,dF^2\vert_\alpha,d\sigma_\alpha) \\
        &= \eta^{[3]}\vert_\alpha(dF^1\vert_\alpha,dF^2\vert_\alpha,dF^3\vert_\alpha) - F^1\vert_\alpha\iota_{d\sigma_\alpha}\eta^{[3]}\vert_\alpha(dF^2\vert_\alpha,dF^3\vert_\alpha) \\
 		&\quad + F^2\vert_\alpha\iota_{d\sigma_\alpha}\eta^{[3]}\vert_\alpha(dF^1\vert_\alpha,dF^3\vert_\alpha) - F^3\vert_\alpha\iota_{d\sigma_\alpha}\eta^{[3]}\vert_\alpha(dF^1\vert_\alpha,dF^2\vert_\alpha)
         \\
        &= \eta^{[3]}\vert_\alpha(dF^1\vert_\alpha,dF^2\vert_\alpha,dF^3\vert_\alpha) - F^1\vert_\alpha\iota_{\theta\vert_\alpha}\eta^{[3]}\vert_\alpha(dF^2\vert_\alpha,dF^3\vert_\alpha) \\
 		&\quad + F^2\vert_\alpha\iota_{\theta\vert_\alpha}\eta^{[3]}\vert_\alpha(dF^1\vert_\alpha,dF^3\vert_\alpha) - F^3\vert_\alpha\iota_{\theta\vert_\alpha}\eta^{[3]}\vert_\alpha(dF^1\vert_\alpha,dF^2\vert_\alpha)
        ,
 	\end{split}
 \end{equation}
 where $\theta$ is the Lee form in \eqref{titaa}. 
 Clearly, the left-hand side of this equation is the multiplication of a local function by a conformal factor as in \eqref{glue-func}. So the left-hand side of the calculation \eqref{glued-3NP} can be written as a global function. This defines the global $3$-bracket $\{\bullet,\bullet,\bullet\}$ whose restriction to  the local chart $U_\alpha$ is
 \begin{equation}\label{sennur}
    \{F^1,F^2,F^3\}\vert_\alpha =  e^{\sigma_\alpha}
    \{F^1_\alpha,F^2_\alpha,F^3_\alpha\}_\alpha.
 \end{equation}
 Furthermore, for the right-hand side, we introduce a bivector field
\begin{equation}\label{newE}
    \mathcal{E}^{[2]}\vert_\alpha = -\iota_{d\sigma_\alpha}\eta^{[3]}\vert_\alpha=-\iota_{\theta\vert_\alpha}\eta^{[3]}\vert_\alpha.
\end{equation}
Hence, we can write the global expression of each term in the right-hand side of the calculation \eqref{glued-3NP} as well. As a result, referring to the $3$-vector field $\eta ^{[3]}$ in \eqref{local-3-vector} and the bivector field $\mathcal{E}^{[2]}$ in \eqref{newE}, we determine a global $3$-bracket of the global functions $F^1$, $F^2$ and $F^3$ on $\mathcal{P}$ as
 \begin{equation}\label{LCNJ}
 \begin{split}
\{ F^1,F^2,F^3 \} &= \eta ^{[3]}(dF^1,dF^2,dF^3) + F^1\mathcal{E}^{[2]}(dF^2,dF^3) - F^2 \mathcal{E}^{[2]}(dF^1,dF^3) \\& \qquad  + F^3 \mathcal{E}^{[2]}(dF^1,dF^2).     
 \end{split}
\end{equation}
We call this bracket a locally conformal Nambu--Poisson $3$-bracket.
Note that the $3$-bracket \eqref{LCNJ} is generated by the pair
$(\eta^{[3]},\mathcal{E}^{[2]})$.

A comparison with the Nambu--Jacobi $k$-bracket \eqref{bjk} for $k=3$
suggests a close relationship between these two structures.
The following results make this relationship precise.
In particular, we first show that the contraction of a locally conformally
$3$-Nambu--Poisson manifold naturally induces a locally conformal Poisson
structure.
This structural result then allows us to conclude that the locally
conformal Nambu--Poisson $3$-bracket \eqref{LCNJ} indeed defines a
Nambu--Jacobi $3$-bracket.

\begin{theorem}\label{defne}
Let $(\mathcal{P},U_\alpha,\eta_\alpha^{[3]},\sigma_\alpha)$ be a locally conformally $3$-Nambu--Poisson manifold. 
For any family $\{F_\alpha\}$ of local functions, the contractions of local $3$-vector fields defines a family of bivector fields:
\begin{equation}\label{localcontract-3to2}
  \{ \Lambda_\alpha \} = \{\iota_{dF_\alpha}\eta^{[3]}_\alpha \}.
\end{equation}
The resulting family $\{(U_\alpha,\Lambda_\alpha)\}$ satisfies the compatibility conditions \eqref{kuskun}, and hence defines a locally conformal Poisson structure on $\mathcal{P}$.
\end{theorem}

\begin{proof}
    On a local chart $U_\alpha$, contract the local Nambu--Poisson $3$-vector field $\eta_\alpha^{[3]}$ with a local exact one-form $dF_\alpha$ as in \eqref{localcontract-3to2}. 
According to Corollary \ref{emre1},  $\Lambda_\alpha$ is a local Poisson bivector field since it arises as the contraction of the local Nambu--Poisson $3$-vector field $\eta^{[3]}_\alpha$. We now show that the local Poisson bivector fields $\Lambda_\alpha$, $\Lambda_\beta$, $\dots$, defined as in \eqref{localcontract-3to2}, satisfy the locally conformal relation in \eqref{kuskun}. 
Indeed, on an overlap $U_\alpha\cap U_\beta \neq \emptyset$, we compute
\begin{equation}\label{hayda}
\begin{split}
  e^{-\sigma_\alpha}\Lambda_\alpha& = e^{-\sigma_\alpha}\iota_{dF^3_\alpha}\eta^{[3]}_\alpha = e^{-\sigma_\alpha}\iota_{d(e^{-\sigma_\alpha}F^3\vert_\alpha)}e^{2\sigma_\alpha}\eta^{[3]}\vert_\alpha \\
     & = \iota_{(dF^3\vert_\alpha - F^3\vert_\alpha d\sigma_\alpha)}\eta^{[3]}\vert_\alpha \\
     & = \iota_{dF^3\vert^{[3]}_\alpha}\eta^{[3]}\vert_\alpha - F^3\vert_\alpha\iota_{d\sigma_\alpha}\eta^{[3]}\vert_\alpha \\
     & = \iota_{dF^3\vert_\alpha}\eta^{[3]}\vert_\alpha + F^3\vert_\alpha\mathcal{E}^{[2]}\vert_\alpha
     \\  
     & = \iota_{dF^3\vert_\beta}\eta^{[3]}\vert_\beta + F^3\vert_\beta\mathcal{E}^{[2]}\vert_\beta
     \\& = e^{-\sigma_\beta}\Lambda_\beta,
 \end{split}
\end{equation} 
where we have used the bivector field notation $\mathcal{E}^{[2]}$ in \eqref{newE}, and the local bivector field $\Lambda_\beta$ is assumed to be
\begin{equation}\label{localcontract-3to2+}
\Lambda_\beta = \iota_{dF_\beta}\eta^{[3]}_\beta.
\end{equation}
Since this relation holds true for all of the charts, the local Poisson structures $\{(U_\alpha,\Lambda_\alpha)\}$ determine a locally conformally Poisson manifold. 
\end{proof}

We now use the contraction result above to identify the locally conformal
Nambu--Poisson $3$-bracket \eqref{LCNJ} as a Nambu--Jacobi $3$-bracket.

\begin{theorem}\label{memre}
A locally conformally $3$-Nambu--Poisson manifold is a $3$-Nambu--Jacobi manifold. Further, the locally conformal Nambu--Poisson $3$-bracket in \eqref{LCNJ} is a Nambu--Jacobi $3$-bracket, determined by the pair $(\eta^{[3]},\mathcal{E}^{[2]})$, where $\eta^{[3]}$ is the global $3$-vector field introduced in \eqref{local-3-vector}, and $\mathcal{E}^{[2]}$ is the bivector field given in \eqref{newE}.
\end{theorem}

\begin{proof}

Theorem \ref{prop-LCP-Jacobi} states that every locally conformally Poisson manifold is a Jacobi manifold. To apply this theorem in the present analysis, we need to identify the pair $(\Lambda,Z)$ that determines the Jacobi structure for the locally conformally Poisson manifold $\{(U_\alpha,\Lambda_\alpha)\}$ where $\Lambda_\alpha$'s are defined as in \eqref{localcontract-3to2}.  For this purpose, using the computation in \eqref{hayda} and recalling \eqref{LCP-bi-vector}, we define on each chart $U_\alpha$ the bivector field
\begin{equation}\label{yasincan}
 \begin{split}
     \Lambda\vert_\alpha & = e^{-\sigma_\alpha}\Lambda_\alpha   = \iota_{dF^3\vert_\alpha}\eta^{[3]}\vert_\alpha + F^3\vert_\alpha\mathcal{E}^{[2]}\vert_\alpha,
 \end{split}
 \end{equation}
which can be glued together to form a global bivector field $\Lambda$. Further, referring to the contraction \eqref{LCP-vector}, we compute a vector field 
 \begin{equation}\label{hayt}
 \begin{split}
    Z\vert_\alpha &= \iota_{d\sigma_\alpha}\Lambda\vert_\alpha = \iota_{d\sigma_\alpha}( \iota_{dF^3\vert_\alpha}\eta^{[3]}\vert_\alpha + F^3\vert_\alpha\mathcal{E}^{[2]}\vert_\alpha) \\
    & = \iota_{d\sigma_\alpha}\iota_{dF^3\vert_\alpha}\eta^{[3]}\vert_\alpha + F^3\vert_\alpha \iota_{d\sigma_\alpha}\mathcal{E}^{[2]}\vert_\alpha \\
    & = \iota_{d\sigma_\alpha}\iota_{dF^3\vert_\alpha}\eta^{[3]}\vert_\alpha - F^3\vert_\alpha \iota_{d\sigma_\alpha}\iota_{d\sigma_\alpha}\eta^{[3]}\vert_\alpha \\
    & = - \iota_{dF^3\vert_\alpha}\iota_{d\sigma_\alpha}\eta^{[3]}\vert_\alpha \\
    & = \iota_{dF^3\vert_\alpha}\mathcal{E}^{[2]}\vert_\alpha 
\end{split}
\end{equation}
that can be glued up to a global vector field $Z$. Here, we have substituted $\mathcal{E}^{[2]}$ from \eqref{newE}. 
Hence, we obtain the following global Jacobi pair:
\begin{equation}\label{globalcontract-3to2}
    \Lambda = \iota_{dF^3}\eta ^{[3]}+ F^3\mathcal{E}^{[2]}, \qquad Z = \iota_{dF^3}\mathcal{E}^{[2]}.
\end{equation}
The Jacobi bracket determined by a Jacobi pair $(\Lambda,Z)$ is exhibited in \eqref{bra-Jac}. Accordingly, for the Jacobi pair \eqref{globalcontract-3to2}, we compute the Jacobi bracket as
\begin{equation}
    \{ F^1,F^2\}^{F^3} = \Big(\iota_{dF^3}\eta^{[3]} + F^3\mathcal{E}^{[2]}\Big) (dF^1,dF^2) + F^1\iota_{dF^3}\mathcal{E}^{[2]}(F^2) - F^2 \iota_{dF^3}\mathcal{E}^{[2]}(F^1).
\end{equation}

Corollary \ref{babanis} states that a pair $(\eta^{[3]},\mathcal{E}^{[2]})$ determines a $3$-Nambu--Jacobi manifold structure if and only if the pair $(\Lambda,Z)$ obtained by its contraction \eqref{atakta} defines a Jacobi manifold. In the present situation, we arrive at the same conclusion: the contracted pair $(\Lambda,Z)$ in \eqref{globalcontract-3to2} defines a Jacobi manifold. Therefore, we conclude that the pair $(\eta^{[3]},\mathcal{E}^{[2]})$, where $\eta^{[3]}$ is given in \eqref{local-3-vector} and $\mathcal{E}^{[2]}$ in \eqref{newE}, is a $3$-Nambu--Jacobi pair, and hence determines a $3$-Nambu--Jacobi manifold. This completes the proof. 
\end{proof}

The flow of the proof is instructive for the subsequent analysis in this work. To summarize the argument at a glance, we present the following commutative diagram:
\begin{equation}\label{geomdiagram--}
\xymatrix{\{\eta^{[3]}_\alpha\} \ar@/^2pc/[rrr]^{\text{Contraction } \eqref{localcontract-3to2}}
\ar[ddd]_{ {\scriptsize\begin{matrix} \text{glue up}\\ \text{with} \\\text{Theorem }\ref{memre} \end{matrix}}} \ar@{<->}[rrr]_{\text{Theorem } \ref{defne}}&   & &\{\Lambda_\alpha\} \ar[ddd]^{ {\scriptsize\begin{matrix} \text{glue up} \\ \text{with} \\ \text{Theorem }\ref{prop-LCP-Jacobi}  \end{matrix}}}\\
  &  & &\\
  &  & &\\
(\eta^{[3]},\mathcal{E}^{[2]})\ar@/_2pc/[rrr]_{\text{Contraction } \eqref{globalcontract-3to2}}\ar@{<->} [rrr]^{\text{Corollary } \ref{babanis}}&  & & (\Lambda,Z)}
\end{equation}

The upper row of Diagram~\eqref{geomdiagram--} represents the relationship between a locally conformally $3$-Nambu--Poisson manifold and its contraction, as described in Theorem~\ref{defne}. Accordingly, the contraction yields a locally conformal Poisson structure. The downward arrow on the right refers to Theorem~\ref{prop-LCP-Jacobi}, which establishes that this locally conformal Poisson structure admits an associated Jacobi character $(\Lambda,Z)$.

The lower row, on the other hand, depicts the contraction relation between the globally defined structure $(\eta^{[3]},\mathcal{E}^{[2]})$ (obtained by gluing together the local charts of the locally conformally $3$-Nambu--Poisson manifold) and the Jacobi structure $(\Lambda,Z)$. By Corollary \ref{babanis}, $(\eta^{[3]},\mathcal{E}^{[2]})$ defines a $3$-Nambu--Jacobi structure since $(\Lambda,Z)$ is Jacobi. Hence, the downward arrow on the left represents Theorem~\ref{memre}, which is derived by combining Theorem~\ref{defne}, Theorem~\ref{prop-LCP-Jacobi}, and Corollary~\ref{babanis}, showing that the gluing and contraction processes are mutually consistent.

By Theorem~\ref{memre}, a locally conformal $3$-Nambu--Poisson structure
induces a Nambu--Jacobi structure determined by the pair
$(\eta^{[3]},\mathcal{E}^{[2]})$, with
$\mathcal{E}^{[2]}=-\iota_\theta\eta^{[3]}$.
On the other hand, given the associated global data $(\eta^{[3]},\theta)$,
with a closed one-form $\theta$, one may recover the locally conformal description.
Indeed, the local exactness of $\theta$ yields local potentials
$\sigma_\alpha$ satisfying $d\sigma_\alpha=d\sigma_\beta$ on overlaps, and
on each chart $U_\alpha$ this leads to
\begin{equation}
\eta^{[3]}_\alpha := e^{2\sigma_\alpha}\,\eta^{[3]}\big|_\alpha,
\end{equation}
which coincides with the local expression in \eqref{local-3-vector}. This observation motivates an equivalent global description of a locally
conformally $3$-Nambu--Poisson manifold in terms of either of the triples
\begin{equation}\label{LC-3-NP}
(\mathcal{P},\eta^{[3]},\mathcal{E}^{[2]})
\;\longleftrightarrow\;
(\mathcal{P},\eta^{[3]},\theta).
\end{equation}

The following corollary reformulates Theorem~\ref{defne}, originally stated
in terms of local data, using the global realization of locally conformally
$3$-Nambu--Poisson manifolds via the Lichnerowicz--de~Rham exact one-form
$d_\theta F$ (see Appendix~\ref{sec:LdR}).

\begin{corollary}
The triple
\begin{equation}
(\mathcal{P},\eta^{[3]},\mathcal{E}^{[2]})
=
(\mathcal{P},\eta^{[3]},-\iota_\theta\eta^{[3]})
\end{equation}
defines a locally conformal $3$-Nambu--Poisson structure if and only if,
for any function $F\in C^\infty(\mathcal{P})$, the contraction
\begin{equation}\label{sermet}
(\Lambda,Z)
=
\big(
\iota_{d_\theta F}\eta^{[3]},
-\iota_{d_\theta F\wedge\theta}\eta^{[3]}
\big),
\end{equation}
of the associated Nambu--Jacobi pair by the Lichnerowicz--de~Rham exact
one-form $d_\theta F$ defines a locally conformal Poisson structure with
respect to the same Lee form $\theta$.
\end{corollary}

\begin{proof}
By Theorem \ref{defne}, the contraction of a locally conformally $3$-Nambu--Poisson manifold is a locally conformally Poisson manifold. Moreover, from \eqref{localcontract-3to2} and \eqref{hayt}, we have the contractions. So, start with the local equality in \eqref{localcontract-3to2} and substitute the corresponding global tensorial objects $\eta^{[3]}$ in \eqref{local-3-vector}, $F$ in \eqref{glue-func}, and $\theta$ in \eqref{titaa}: 
    \begin{equation}
    \begin{split}
    e^{\sigma_\alpha} \Lambda\vert_\alpha &=\iota_{d(e^{-\sigma_\alpha}F\vert_\alpha)} e^{2\sigma_\alpha}  \eta^{[3]} \vert_\alpha \\
        \Lambda\vert_\alpha &= \iota_{dF\vert_\alpha-F\vert_\alpha \theta\vert_\alpha} \eta^{[3]} \vert_\alpha.
    \end{split}
    \end{equation}
So, we have the first term on the right-hand side of the  exhibition \eqref{sermet} as 
\begin{equation}
     \Lambda=\iota_{dF-F\theta}\eta^{[3]}=\iota_{d_\theta F}\eta^{[3]},
\end{equation}
    where we have substituted the Lichnerowicz-de Rham differential (see Appendix \ref{sec:LdR}). 
For the second term on the right-hand side of the  exhibition \eqref{sermet}, we recall  \eqref{hayt} and referring to \eqref{newE}, compute   
\begin{equation}
    Z = \iota_{dF}\mathcal{E}^{[2]} =  - \iota_{dF} \iota _\theta  \eta^{[3]} = -\iota_{d  F \wedge \theta  }\eta^{[3]} = -  \iota_{d_\theta F \wedge \theta} \eta^{[3]}.
\end{equation} 
This completes the proof. 
\end{proof}



\subsection{Locally Conformal Nambu Hamiltonian Dynamics, and Its Bi-Hamiltonian Realization}\label{sec:bi-jac}

In this subsection, we merge the local Nambu--Poisson dynamics with the presence of local conformality, as described in the previous subsection. Accordingly, we assume  a locally conformally $3$-Nambu--Poisson manifold $ (\mathcal{P},U_\alpha,\eta_\alpha^{[3]},\sigma_\alpha)$.

\noindent \textbf{Locally Conformal $3$-Nambu--Poisson Hamiltonian Dynamics.} Referring to the Hamiltonian vector field definition  for $3$-Nambu--Poisson manifolds in \eqref{zeynep}, for two local Hamiltonian functions $H^1_\alpha$  and $H^2_\alpha$ obeying the globalization policy defined in \eqref{glue-func}, we compute the local Nambu Hamiltonian vector field as
\begin{equation} \label{luck}
\begin{split}
     X_{H^1_\alpha ,H^2_\alpha}&= \iota_{dH^1_\alpha\wedge dH^2_\alpha} \eta^{[3]}_\alpha = \iota_{dH^1_\alpha}\iota_{dH^2_\alpha} \eta^{[3]}_\alpha = \iota_{d(e^{-\sigma_\alpha}H^1\vert_\alpha)}\iota_{d(e^{-\sigma_\alpha}H^2\vert_\alpha)} e^{2\sigma_\alpha}\eta^{[3]}\vert_\alpha \\
     &  = \iota_{(dH^1\vert_\alpha-H^1\vert_\alpha d\sigma_\alpha)}\iota_{(dH^2\vert_\alpha-H^2\vert_\alpha d\sigma_\alpha)}\eta^{[3]}\vert_\alpha \\
     & = \iota_{dH^1\vert_\alpha}\iota_{dH^2\vert_\alpha}\eta^{[3]}\vert_\alpha - H^1\vert_\alpha\iota_{ d\sigma_\alpha}\iota_{dH^2\vert_\alpha}\eta^{[3]}\vert_\alpha - H^2\vert_\alpha\iota_{dH^1\vert_\alpha}\iota_{ d\sigma_\alpha}\eta^{[3]}\vert_\alpha \\
     & = \iota_{dH^1\vert_\alpha}\iota_{dH^2\vert_\alpha}\eta^{[3]}\vert_\alpha + H^1\vert_\alpha\iota_{dH^2\vert_\alpha}\iota_{ d\sigma_\alpha}\eta^{[3]}\vert_\alpha - H^2\vert_\alpha\iota_{dH^1\vert_\alpha}\iota_{ d\sigma_\alpha}\eta^{[3]}\vert_\alpha,
\end{split}
\end{equation}
where we have employed the definition of $\eta^{[3]}$ in \eqref{local-3-vector}. 
In terms of the bivector field $\mathcal{E}^{[2]}$ in \eqref{newE}, the calculation in \eqref{luck} can be written as
\begin{equation}\label{senne}
    X_{H^1_\alpha ,H^2_\alpha}=\iota_{dH^1\vert_\alpha}\iota_{dH^2\vert_\alpha}\eta^{[3]}\vert_\alpha - H^1\vert_\alpha\iota_{dH^2\vert_\alpha}\mathcal{E}^{[2]}\vert_\alpha + H^2\vert_\alpha\iota_{dH^1\vert_\alpha}\mathcal{E}^{[2]}\vert_\alpha .
\end{equation}

Notice that all the terms on the right-hand side of \eqref{senne} have global realizations. 
This observation naturally leads us to assume that the vector fields $\{X_{H^1_\alpha ,H^2_\alpha}\}$ generating the local particle motion are global. 
In other words, two local Nambu Hamiltonian dynamics, $X_{H^1_\alpha ,H^2_\alpha}$ and $X_{H^1_\beta ,H^2_\beta}$, coincide on the overlapping region $U_\alpha \cap U_\beta$. 
Equivalently, for the global Hamiltonian functions $H^1$ and $H^2$ with globalization property \eqref{glue-func}, we conclude that there exists a global vector field $X_{H^1 ,H^2}$ such that
\begin{equation}
   X_{H^1 ,H^2}\vert_{\alpha} = X_{H^1_\alpha ,H^2_\alpha}.
\end{equation}
Therefore, from \eqref{senne}, the locally conformal Nambu Hamiltonian dynamics $X_{H^1 ,H^2}$ is obtained to be 
\begin{equation}
    X_{H^1 ,H^2}=\iota_{dH^1}\iota_{dH^2}\eta^{[3]} - H^1\iota_{dH^2}\mathcal{E}^{[2]} + H^2\iota_{dH^1}\mathcal{E}^{[2]}
\end{equation}
Hence, the action of the vector field on a function, or in other terms, the dynamics of the observable $F$ is recorded as
\begin{equation}
    X_{H^1 ,H^2}(F)=\iota_{dH^1}\iota_{dH^2}\eta^{[3]}(F) - H^1\iota_{dH^2}\mathcal{E}^{[2]}(F) + H^2\iota_{dH^1}\mathcal{E}^{[2]}(F).
\end{equation}
Let us once more consider the locally conformal Nambu Hamiltonian vector field but this time we rewrite it in terms of locally conformal Nambu--Poisson $3$-bracket as
\begin{equation}
    X_{H^1 ,H^2}(F)=\{F,H^1,H^2\} - F\mathcal{E}(dH^1 ,dH^2)
\end{equation}

\noindent \textbf{Locally Conformal Bi-Hamiltonian Dynamics.} 
As discussed in Subsection~\ref{sec:BiH}, a Nambu Hamiltonian
dynamics can be decomposed into a bi-Hamiltonian structure.
We now employ this approach in the locally conformal setting.
Starting from a local $3$-Nambu--Poisson structure
$(U_\alpha,\eta^{[3]}_\alpha)$, we contract the local Nambu--Poisson
$3$-vector field $\eta^{[3]}_\alpha$ with the local Hamiltonian functions
$H^2_\alpha$ and $H^1_\alpha$ in order to obtain two local Poisson bivector
fields.
We then show that these compatible local bivectors globalize naturally and
give rise to a bi-Hamiltonian dynamics.

We begin with the local bivector field obtained by contracting
$\eta^{[3]}_\alpha$ with the local Hamiltonian function $H^2_\alpha$:
\begin{equation}
\begin{split}
\Lambda^2_\alpha
&= \iota_{dH^2_\alpha}\eta^{[3]}_\alpha \\
&= \iota_{d(e^{-\sigma_\alpha}H^2|_\alpha)}\,
e^{2\sigma_\alpha}\eta^{[3]}|_\alpha \\
&= e^{\sigma_\alpha}\,
\iota_{(dH^2|_\alpha - H^2|_\alpha\, d\sigma_\alpha)}\eta^{[3]}|_\alpha \\
&= e^{\sigma_\alpha}\,
\iota_{dH^2|_\alpha}\eta^{[3]}|_\alpha
- e^{\sigma_\alpha}H^2|_\alpha\,
\iota_{d\sigma_\alpha}\eta^{[3]}|_\alpha .
\end{split}
\end{equation}
Multiplying both sides by the conformal factor $e^{-\sigma_\alpha}$ yields
\begin{equation}
\begin{split}
e^{-\sigma_\alpha}\Lambda^2_\alpha
&= \iota_{dH^2|_\alpha}\eta^{[3]}|_\alpha
- H^2|_\alpha\,\iota_{d\sigma_\alpha}\eta^{[3]}|_\alpha \\
&= \iota_{dH^2|_\alpha}\eta^{[3]}|_\alpha
+ H^2|_\alpha\,\mathcal{E}^{[2]}|_\alpha .
\end{split}
\end{equation}
Both sides of the equation admit global realizations, provided that there
exists a global bivector field $\Lambda^{2}$ whose local representatives
satisfy
\[
\Lambda^2_\alpha = e^{\sigma_\alpha}\,\Lambda^2|_\alpha .
\]
Accordingly, we obtain the global expression
\begin{equation}
\Lambda^2 = \iota_{dH^2}\eta^{[3]} + H^2\mathcal{E}^{[2]},
\end{equation}
which coincides with the form given in \eqref{dogu-1}.

A similar argument applies when contracting $\eta^{[3]}_\alpha$ with the local
Hamiltonian function $H^1_\alpha$:
\begin{equation}
\begin{split}
e^{-\sigma_\alpha}\Lambda^1_\alpha
&= -e^{-\sigma_\alpha}\iota_{dH^1_\alpha}\eta^{[3]}_\alpha \\
&= -\iota_{dH^1|_\alpha}\eta^{[3]}|_\alpha
- H^1|_\alpha\,\mathcal{E}^{[2]}|_\alpha .
\end{split}
\end{equation}
Assuming the existence of a global bivector field $\Lambda^{1}$ with local
representatives
\[
\Lambda^1_\alpha = e^{\sigma_\alpha}\,\Lambda^1|_\alpha ,
\]
we obtain the global bivector field
\begin{equation}
\Lambda^1 = -\iota_{dH^1}\eta^{[3]} - H^1\mathcal{E}^{[2]},
\end{equation}
which agrees with \eqref{dogu-2}.

We are now ready to assemble the local bi-Hamiltonian dynamics.
We first consider the local Poisson structure $\Lambda^2_\alpha$ and the
Hamiltonian dynamics generated by the local Hamiltonian function
$H^1_\alpha$:
\begin{equation}
\begin{split}
    \frac {dx}{dt} &= \iota_{dH^1_\alpha}\Lambda^2_\alpha = \iota_{d(e^{-\sigma_\alpha}H^1\vert_\alpha)}e^{\sigma_\alpha}\Lambda^2\vert_\alpha \\
    &= \iota_{dH^1\vert_\alpha}\Lambda^2\vert_\alpha - H^1\vert_\alpha\iota_{ d\sigma_\alpha}\Lambda^2\vert_\alpha
\end{split}
\end{equation}
As in \eqref{LCP-vector}, introducing the global vector field $Z^2$ with local
representation $Z^2|_\alpha=\iota_{d\sigma_\alpha}\Lambda^2|_\alpha$, the
dynamics attains the global form
\begin{equation}
\frac{dx}{dt} = \iota_{dH^1}\Lambda^2 - H^1 Z^2 .
\end{equation}
Moreover, the vector field $Z^2$ can be computed explicitly as
\begin{equation} \begin{split} Z^2 = \iota_{ \theta}\Lambda^2 &= \iota_{ \theta}(\iota_{dH^2}\eta^{[3]} + H^2\mathcal{E}^{[2]}) \\ & = \iota_{ \theta}\iota_{dH^2}\eta^{[3]} + H^2\iota_{ \theta}\mathcal{E}^{[2]} \\ & = -\iota_{dH^2}\iota_{ \theta}\eta^{[3]} - H^2\iota_{ \theta}\iota_{ \theta}\eta^{[3]} \\ & = \iota_{dH^2}\mathcal{E}^{[2]}. \end{split} \end{equation}
This vector field coincides with the one appearing in \eqref{dogu-1}.

Proceeding similarly for the local Poisson structure $\Lambda^1_\alpha$, the
local dynamics generated by $H^2_\alpha$ is given by
\begin{equation}
\frac{dx}{dt}
= \iota_{dH^2|_\alpha}\Lambda^1|_\alpha
- H^2|_\alpha\,\iota_{d\sigma_\alpha}\Lambda^1|_\alpha .
\end{equation}
The local term
$\iota_{d\sigma_\alpha}\Lambda^1|_\alpha$ defines a global vector field $Z^1$.
The corresponding global dynamics then takes the form
\begin{equation}
\frac{dx}{dt} = \iota_{dH^2}\Lambda^1 - H^2 Z^1 .
\end{equation}
A direct computation shows that the global vector field $Z^1$ agrees with the vector field appearing in \eqref{dogu-2}, that is,
\begin{equation}
Z^1 = -\iota_{dH^1}\mathcal{E}^{[2]} .
\end{equation}

This analysis leads to the following conclusion:
locally conformal bi-Hamiltonian dynamics arises as a particular instance of
the Jacobi bi-Hamiltonian dynamics described in
Subsection~\ref{ebru}.
This result is expected, since locally conformally $3$-Nambu--Poisson
manifolds are precisely $3$-Nambu--Jacobi manifolds.

\subsection{Analyzing Locally Conformally k-Nambu--Poisson Manifolds}

Let $(\mathcal{P},U_\alpha,\eta_\alpha^{[k]},\sigma_\alpha)$ be a locally
conformally $k$-Nambu--Poisson manifold, so that the local $k$-vector fields
$\eta^{[k]}_\alpha$ glue to a global $k$-vector field $\eta^{[k]}$ according
to \eqref{gs}.
Suppose further that we are given local functions $F^i_\alpha$, $i=1,\dots,k$,
chosen compatibly with the locally conformal structure, as in the Poisson and
$3$-Nambu--Poisson cases, in the sense that they glue to global functions
$F^i$ via
\begin{equation}\label{kalinkamaya}
    F^i\big|_\alpha = e^{\sigma_\alpha} F^i_\alpha ,\qquad i=1,\dots,k.
\end{equation}
With the above choice of local functions $F^i_\alpha$, the local
Nambu--Poisson $k$-bracket $\{\bullet,\dots,\bullet\}_\alpha$ determined by the
local $k$-vector field $\eta^{[k]}_\alpha$ takes the form
\begin{equation}
\begin{split}
\{ F^1_\alpha,\dots,F^k_\alpha \}_\alpha
&= \eta^{[k]}_\alpha (dF^1_\alpha,\dots,dF^k_\alpha) \\
&= e^{(k-1)\sigma_\alpha}\,
\eta^{[k]}\big|_\alpha
\big(d(e^{-\sigma_\alpha}F^1|_\alpha),\dots,
d(e^{-\sigma_\alpha}F^k|_\alpha)\big) \\
&= e^{-\sigma_\alpha}\,
\eta^{[k]}\big|_\alpha
\big(dF^1|_\alpha - F^1|_\alpha\, d\sigma_\alpha,\dots,
dF^k|_\alpha - F^k|_\alpha\, d\sigma_\alpha\big),
\end{split}
\end{equation}
where we have employed the identities \eqref{gs} and \eqref{kalinkamaya}.
 Then, multiplying both sides by $e^{\sigma_\alpha}$, we obtain
\begin{equation}\label{desperado}
\begin{split}
       e^{\sigma_\alpha} \{ F^1_\alpha,\dots,F^k_\alpha \}_\alpha   &= \eta^{[k]}\vert_\alpha(dF^1\vert_\alpha,\dots,dF^k\vert_\alpha ) - F^1\vert_\alpha\eta^{[k]}\vert_\alpha(  d\sigma_\alpha,dF^2\vert_\alpha\dots,dF^k\vert_\alpha )  \\
    & \qquad \qquad  - F^2\vert_\alpha\eta^{[k]}\vert_\alpha(  dF^1\vert_\alpha,d\sigma_\alpha,\dots,dF^k\vert_\alpha ) \\
    & \qquad \qquad- \dots - F^k\vert_\alpha\eta^{[k]}\vert_\alpha(dF^1\vert_\alpha,\dots,dF^{k-1}\vert_\alpha,d\sigma_\alpha)\\
    &  = \eta^{[k]}\vert_\alpha(dF^1\vert_\alpha,\dots,dF^k\vert_\alpha ) \\
    & \qquad \qquad-\sum^k_{i=1}(-1)^{k-i}F^i\vert_\alpha\eta^{[k]}\vert_\alpha(dF^1\vert_\alpha,\dots,\widehat{dF^i}\vert_\alpha,\dots,dF^k\vert_\alpha,d\sigma_\alpha) \\
    &  = \eta^{[k]}\vert_\alpha(dF^1\vert_\alpha,\dots,dF^k\vert_\alpha ) \\
    & \qquad \qquad-\sum^k_{i=1}(-1)^{k-i}F^i\vert_\alpha\iota_{d\sigma_\alpha}\eta^{[k]}\vert_\alpha(dF^1\vert_\alpha,\dots,\widehat{dF^i}\vert_\alpha,\dots,dF^k\vert_\alpha) \\
    &  = \eta^{[k]}\vert_\alpha(dF^1\vert_\alpha,\dots,dF^k\vert_\alpha ) \\
    & \qquad \qquad-\sum^k_{i=1}(-1)^{k-i}F^i\vert_\alpha\iota_{\theta_\alpha}\eta^{[k]}\vert_\alpha(dF^1\vert_\alpha,\dots,\widehat{dF^i}\vert_\alpha,\dots,dF^k\vert_\alpha).
    \end{split}
\end{equation}
Here, $\theta$ is the Lee form in \eqref{titaa}. The left-hand side of \eqref{desperado} is the multiplication of a local function by a locally conformal factor, therefore obeying our policy displayed in \eqref{kalinkamaya}. This leads us to introduce a global  $k$-bracket $\{\bullet,\dotsm,\bullet\}$ on the whole manifold. Note that the restriction of this global bracket to a local chart $U_\alpha$ is 
 \begin{equation}\label{sennure}
    \{F^1,\dotsm,F^k\}\vert_\alpha =  e^{\sigma_\alpha}
    \{F^1_\alpha,\dotsm,F^k_\alpha\}_\alpha.
 \end{equation}
 Furthermore, an immediate observation helps us to determine a global $(k-1)$-vector field $\mathcal{E}^{[k-1]}$ admitting the local characterization 
\begin{equation}\label{newE-k}
    \mathcal{E}^{[k-1]}\vert_\alpha = (-1)^k\iota_{d\sigma_\alpha}\eta^{[k]}\vert_\alpha=(-1)^k\iota_{\theta\vert_\alpha}\eta^{[k]}\vert_\alpha.
\end{equation}
As a result, referring to the global $k$-vector field 
$\eta ^{[k]}$ and the global $(k-1)$-vector field $\mathcal{E}^{[k-1]}$ we arrive at the following global $k$-bracket of the global functions $F^i$ on $\mathcal{P}$ as
 \begin{equation}\label{malinkamaya}
 \begin{split}
     \{ F^1,\dots,F^k \} &=  \eta (dF^1,\dots,dF^k) + \sum^k_{i=1}(-1)^{i-1}F^i\mathcal{E}(dF^1,\dots\widehat{dF^i},\dots,dF^k)
 \end{split}
\end{equation}
 for $i=1,\ldots,k$. We call this bracket a locally conformal Nambu--Poisson $k$-bracket generated by the pair $(\eta ^{[k]},\mathcal{E}^{[k]})$. 

 \noindent \textbf{Locally Conformally $k$-Nambu--Poisson Manifold is $k$-Nambu--Jacobi.} In Theorem~\ref{memre}, we showed that, for $k=3$, the construction described
above yields a locally conformally $3$-Nambu--Poisson manifold which is, in
fact, a $3$-Nambu--Jacobi manifold. In this section, we prove that the same
correspondence holds for arbitrary order $k$. More precisely, we show that
the bracket defined in~\eqref{malinkamaya} satisfies the defining properties
of a Nambu--Jacobi $k$-bracket. Following the same approach as in the $k=3$ case (see Subsection~\ref{subsec:LC-3-NP}), we begin by identifying the
relation between a locally conformally $k$-Nambu--Poisson manifold and its
reduction to a locally conformally $(k-1)$-Nambu--Poisson manifold.

\begin{theorem}\label{defne2}
Let $(\mathcal{P},U_\alpha,\eta_\alpha^{[k]},\sigma_\alpha)$ be a locally conformally $k$-Nambu--Poisson manifold.
For any family $\{F_\alpha\}$ of local smooth functions on $\mathcal{P}$, the contractions of the local $k$-vector fields define a family of $(k-1)$-vector fields
\begin{equation}\label{localcontract-k-to-k-1}
  \{ \eta^{[k-1]}_\alpha \} = \{\iota_{dF_\alpha}\eta^{[k]}_\alpha \}.
\end{equation}
Then the family $\{(U_\alpha,\eta^{[k-1]}_\alpha)\}$ determines a locally conformal $(k-1)$-Nambu--Poisson structure on $\mathcal{P}$,
with the same conformal transition functions $\{\sigma_\alpha\}$ inherited from the original $k$-structure. 
\end{theorem}

\begin{proof} Start with a local Nambu--Poisson $k$-vector field $\eta_\alpha^{[k]}$. By Theorem \ref{emre}, we know that the local $(k-1)$-vector field 
\begin{equation}\label{pelo}
    \eta^{[k-1]}_\alpha = \iota_{dF_\alpha}\eta^{[k]}_\alpha
\end{equation}
is a local Poisson $(k-1)$-vector field. Thus, \eqref{localcontract-k-to-k-1} gives a family of local Nambu--Poisson $(k-1)$-vector fields. 
Now we will show that, on an arbitrary intersection, pair of local $(k-1)$-Nambu--Poisson structures satisfy the conformal relation
    \begin{equation}\label{emirak}
       \eta^{[k-1]}\vert_\alpha = e^{-(k-2)\sigma_\alpha}\,\eta^{[k-1]}_\alpha \;=\; e^{-(k-2)\sigma_\beta}\,\eta^{[k-1]}_\beta = \eta^{[k-1]}\vert_\beta. 
    \end{equation}
To have this, considering the local chart $(U_\alpha,\eta_\alpha^{[k-1]})$, we compute
\begin{equation}\label{haydari}
\begin{split}
  e^{-(k-2)\sigma_\alpha}\eta^{[k-1]}_\alpha& = e^{-(k-2)\sigma_\alpha}\iota_{dF_\alpha}\eta^{[k]}_\alpha = e^{-(k-2)\sigma_\alpha}\iota_{d(e^{-\sigma_\alpha}F\vert_\alpha)}e^{(k-1)\sigma_\alpha}\eta^{[k]}\vert_\alpha \\
     & = \iota_{(dF\vert_\alpha - F\vert_\alpha d\sigma_\alpha)}\eta^{[k]}\vert_\alpha \\
     & = \iota_{dF\vert^{[k]}_\alpha}\eta^{[k]}\vert_\alpha - F\vert_\alpha\iota_{d\sigma_\alpha}\eta^{[k]}\vert_\alpha \\
     & = \iota_{dF\vert_\alpha}\eta^{[k]}\vert_\alpha + (-1)^{k-1}F\vert_\alpha\mathcal{E}^{[k-1]}\vert_\alpha.
 \end{split}
\end{equation}
Here $\mathcal{E}^{[k-1]}\vert_\alpha$ is the $(k-1)$-vector field defined in \eqref{newE-k}. A similar computation for the local chart $(U_\beta,\eta_\beta^{[k-1]})$ gives
\begin{equation}\label{haydarhaydar}
  e^{-(k-2)\sigma_\beta}\eta_\beta^{[k-1]} = \iota_{dF\vert_\beta}\eta^{[k]}\vert_\beta + F\vert_\beta\mathcal{E}^{[k-1]}\vert_\beta.
\end{equation}
The compatibility condition \eqref{bibi} for $(\mathcal{P},U_\alpha,\eta_\alpha^{[k]},\sigma_\alpha)$ and the global definition for \eqref{newE-k} reads that  \eqref{haydari} and \eqref{haydarhaydar} are equal. So, we see that the condition \eqref{emirak} is satisfied. Since this relation holds for all of the charts, the family $\{(U_\alpha,\eta^{[k-1]}_\alpha)\}$ of local $(k-1)$-Nambu--Poisson structures determines a locally conformally $(k-1)$-Nambu--Poisson manifold. 
\end{proof}

See that for $k=3$, Theorem \ref{defne2} turns out to be Theorem \ref{defne}.  

In Theorem \ref{defne2}, we established that the contraction of a locally conformally $k$-Nambu--Poisson manifold yields a locally conformally $(k-1)$-Nambu--Poisson manifold.  
By applying this result iteratively, one observes that performing $(k-2)$ successive contractions on a locally conformally $k$-Nambu--Poisson manifold produces a locally conformally $2$-Nambu--Poisson manifold, namely, a locally conformally Poisson manifold.  
Explicitly, the hierarchy of successive contractions reads:
\begin{equation}\label{alper+}
\begin{split}
   & \eta^{[k-1]}_\alpha = \iota_{dF^k_\alpha}\eta^{[k]}_\alpha, \\ 
   & \eta^{[k-2]}_\alpha = \iota_{dF^{k-1}_\alpha}\eta^{[k-1]}_\alpha = \iota_{dF^{k-1}_\alpha\wedge dF^k_\alpha}\eta^{[k]}_\alpha, \\
   & \;\vdots \\[-1mm]
   & \eta^{[2]}_\alpha =  \iota_{dF^3_\alpha}\eta^{[3]}_\alpha = \iota_{dF^3_\alpha\wedge\ldots\wedge dF^{k-1}_\alpha\wedge dF^k_\alpha}\eta^{[k]}_\alpha .
\end{split}
\end{equation}
Consequently, we obtain a family of local Poisson structures
\begin{equation}
\{(U_\alpha,\eta^{[2]}_\alpha)\} = \{(U_\alpha,\Lambda_\alpha)\},
\end{equation}
which naturally determines a locally conformal Poisson structure on
$\mathcal{P}$. We are now ready to state the following theorem,
asserting that locally conformally $k$-Nambu--Poisson manifolds are
$k$-Nambu--Jacobi manifolds.

\begin{theorem}\label{memre2}
    Locally conformally $k$-Nambu--Poisson manifolds are $k$-Nambu--Jacobi manifolds for all $k$.

\end{theorem}

\begin{proof} The iteration described in \eqref{alper+} shows that applying $(k-2)$ successive
contractions to a locally conformally $k$-Nambu--Poisson manifold
$(\mathcal{P},U_\alpha,\eta_\alpha^{[k]},\sigma_\alpha)$ yields a locally
conformal Poisson structure on $\mathcal{P}$.

Moreover, by Theorem~\ref{prop-LCP-Jacobi}, every locally conformally Poisson manifold is a Jacobi manifold. Referring to the notation of the present subsection, we identify the global pair $(\Lambda,\mathcal{E})$ as the
$(k-2)$-th contraction of the global pair
$(\eta^{[k]},\mathcal{E}^{[k-1]})$.

Indeed, from \eqref{emirak} and \eqref{haydari} we obtain
\begin{equation}\label{etatpur}
\eta^{[k-1]}\big|_\alpha
= \iota_{dF^k|_\alpha}\eta^{[k]}\big|_\alpha
+ (-1)^{k-1} F^k|_\alpha \mathcal{E}^{[k-1]}\big|_\alpha .
\end{equation}
Further, contracting $\eta^{[k-1]}|_\alpha$ with $d\sigma_\alpha$ gives
\begin{equation}
    \begin{split}
        \mathcal{E}^{[k-2]}\vert_\alpha &= (-1)^{k-1}\iota_{d\sigma_\alpha}\eta^{[k-1]}\vert_\alpha \\
        & = (-1)^{k-1}\iota_{d\sigma_\alpha}\iota_{dF^k\vert_\alpha}\eta^{[k]}\vert_\alpha + F^k\vert_\alpha\iota_{d\sigma_\alpha} \mathcal{E}^{[k-1]}\vert_\alpha \\
        & = (-1)^{k}\iota_{dF^k\vert_\alpha}\iota_{d\sigma_\alpha}\eta^{[k]}\vert_\alpha \\
        & = (-1)^{k}\iota_{dF^k\vert_\alpha}\mathcal{E}^{[k-1]}\vert_\alpha.
    \end{split}
\end{equation}

Thus, the global pair $(\eta^{[k-1]},\mathcal{E}^{[k-2]})$ defining the locally
conformal $(k-1)$-Nambu--Poisson structure is given by
\begin{equation}\label{pp1}
\eta^{[k-1]} = \iota_{dF^k}\eta^{[k]} + (-1)^{k-1}F^k\mathcal{E}^{[k-1]},
\qquad
\mathcal{E}^{[k-2]} = (-1)^{k}\iota_{dF^k}\mathcal{E}^{[k-1]}.
\end{equation}
By Theorem~\ref{beris}, the pair $(\eta^{[k]},\mathcal{E}^{[k-1]})$ defines a
$k$-Nambu--Jacobi structure if and only if
$(\eta^{[k-1]},\mathcal{E}^{[k-2]})$ defines a $(k-1)$-Nambu--Jacobi structure.
Iterating this argument completes the proof.
\end{proof}

In particular, for $k=3$, Theorem \ref{memre2} becomes Theorem \ref{memre}. We present the following picture relating locally conformal Nambu--Poisson structures and associated Nambu--Jacobi structures: 
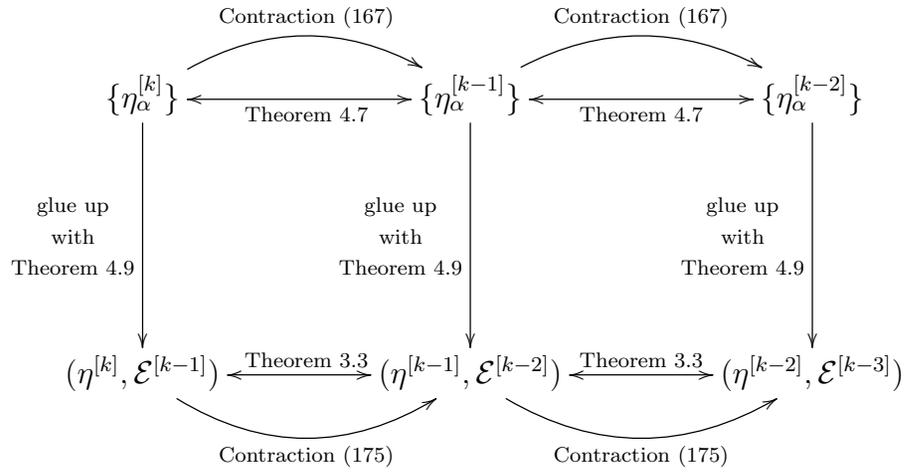
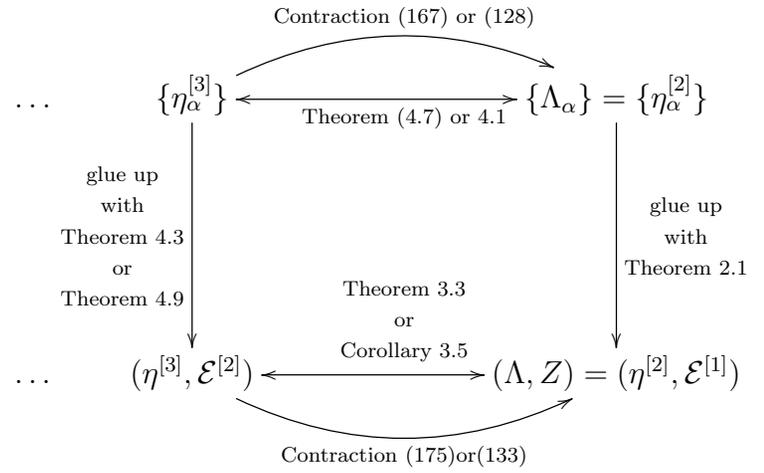
\begin{landscape}
\topskip0pt
\vspace*{\fill}
\begin{equation*}\label{geomdiagram--n}
    \xymatrix{
    \{\eta^{[k]}_\alpha\}
    \ar@/^2pc/[rr]^{\text{Contraction } \eqref{pelo}}
    \ar[ddd]_{ {\scriptsize\begin{matrix} \text{glue up}    \\ \text{with} \\ \text{Theorem }\ref{memre2}       \end{matrix}}}  
    \ar@{<->}[rr]_{\text{Theorem } \ref{defne2}}
    &&  
    \{\eta^{[k-1]}_\alpha\} \ar@/^2pc/[rr]^{\text{Contraction } \eqref{pelo}}
    \ar[ddd]_{ {\scriptsize\begin{matrix} \text{glue up} \\ \text{with} \\ \text{Theorem }\ref{memre2}  \end{matrix}}} 
    \ar@{<->}[rr]_{\text{Theorem } \ref{defne2}}
    &&
    \{\eta^{[k-2]}_\alpha\}
    \ar[ddd]_{ {\scriptsize\begin{matrix} \text{glue up}    \\ \text{with} \\ \text{Theorem }\ref{memre2}       \end{matrix}}}  
     & 
    \dots 
       & 
    \{\eta^{[3]}_\alpha \} \ar@/^2pc/[rrr]^{\text{Contraction } \eqref{pelo} \text{ or } \eqref{localcontract-3to2}}
    \ar[ddd]_{ {\scriptsize
    \begin{matrix} \text{glue up} \\ \text{with} \\ \text{Theorem }\ref{memre}   \\ 
    \text{or} \\ \text{Theorem }\ref{memre2}  
    \end{matrix}}}
    \ar@{<->}[rrr]_{\text{Theorem } \eqref{defne2} \text{ or } \ref{defne}}
    && &
    \{\Lambda_\alpha\}=\{\eta^{[2]}_\alpha \}
    \ar[ddd]^{ {\scriptsize\begin{matrix} \text{glue up} \\ \text{with} \\ \text{Theorem }\ref{prop-LCP-Jacobi}  \end{matrix}}} 
    \\
    \\
    \\
    (\eta^{[k]},\mathcal{E}^{[k-1]})
    \ar@/_2pc/[rr]_{\text{Contraction } \eqref{pp1}} \ar@{<->} [rr]^{\text{Theorem } \ref{beris}} 
    &&  
    (\eta^{[k-1]},\mathcal{E}^{[k-2]}) \ar@/_2pc/[rr]_{\text{Contraction } \eqref{pp1}}
    \ar@{<->} [rr]^{\text{Theorem } \ref{beris}}
    &&   
     (\eta^{[k-2]},\mathcal{E}^{[k-3]})
      & \dots & 
    (\eta^{[3]},\mathcal{E}^{[2]}) 
    \ar@{<->} [rrr]^
    { {\scriptsize
    \begin{matrix} \text{Theorem } \ref{beris}  \\ 
    \text{or} \\ \text{Corollary } \ref{babanis}
    \end{matrix} }}
        \ar@/_2pc/[rrr]_{  \text{Contraction } 
    \eqref{pp1}    
    \text{or}  
    \eqref{globalcontract-3to2}
  }
    && & (\Lambda,Z)=(\eta^{[2]},\mathcal{E}^{[1]})}
\end{equation*}
\vspace*{\fill}
\end{landscape}
Notice that the upper row of the commutative diagram establishes the relations between locally conformal $k$-Nambu--Poisson structures and their locally conformal $(k-1)$-Nambu--Poisson counterparts, descending from $k$ down to $3$. These relations follow directly from Theorem~\ref{defne2} together with the construction in~\eqref{pelo}.

On the base level, the lower row encodes the corresponding relations between the $k$-Nambu--Jacobi and $(k-1)$-Nambu--Jacobi structures, descending again from $k$ to $3$, in accordance with Theorem~\ref{beris} and the contraction~\eqref{pp1}.

The vertical arrows represent the application of Theorem~\ref{memre2}, which establishes the correspondence between a locally conformal $k$-Nambu--Poisson structure and its associated $k$-Nambu--Jacobi structure.

Finally, we remark that the rightmost box in the commutative diagram illustrates how Diagram~\ref{geomdiagram--} is embedded within the present geometric framework.

\subsection{Locally Conformal k-Nambu--Poisson Hamiltonian Dynamics}\label{sec:Ham-LCNP}

To analyze the Hamiltonian dynamics on the locally conformal setting, we start with a local $k$-Nambu--Poisson manifold $(U_\alpha,\eta_\alpha^{[k]})$ where  $\eta_\alpha^{[k]}$ satisfies \eqref{usta}. Let $H^i_\alpha$ be local Hamiltonian functions for $i=1,\dotsm,k-1$ that can be glued to global functions $H^i$ by the rule given in \eqref{kalinkamaya}. We recall the Hamiltonian vector field definition given for $k$-Nambu--Poisson manifolds in \eqref{NP-HVF} to obtain the local Hamiltonian vector field on the local $k$-Nambu--Poisson manifold $(U_\alpha,\eta_\alpha^{[k]})$:
\begin{equation}\label{dalinda}
     X_{H^1_\alpha, \dotsm, H^{k-1}_\alpha}= \iota_{dH^1_\alpha\wedge \dotsm \wedge dH^{k-1}_\alpha} \eta^{[k]}_\alpha.
\end{equation}
Accordingly, the local Nambu Hamiltonian vector field is computed to be
\begin{equation}
    \begin{split}
        X_{H^1_\alpha, \ldots, H^{k-1}_\alpha}&= \iota_{dH^1_\alpha\wedge \ldots \wedge dH^{k-1}_\alpha} \eta^{[k]}_\alpha = \iota_{dH^1_\alpha}\ldots \iota_{dH^{k-1}_\alpha}\eta^{[k]}_\alpha \\
     & = \iota_{d(e^{-\sigma_\alpha}H^1\vert_\alpha)}\ldots \iota_{d(e^{-\sigma_\alpha}H^{k-1}\vert_\alpha)} e^{({k-1})\sigma_\alpha}\eta^{[k]}\vert_\alpha \\
     &  = \iota_{(dH^1\vert_\alpha-H^1\vert_\alpha d\sigma_\alpha)}\ldots\iota_{(dH^{k-1}\vert_\alpha-H^{k-1}\vert_\alpha d\sigma_\alpha)}\eta^{[k]}\vert_\alpha \\
     & = \iota_{dH^1\vert_\alpha}\ldots\iota_{dH^{k-1}\vert_\alpha}\eta^{[k]}\vert_\alpha - \sum^{k-1}_{i=1}H^i\vert_\alpha \iota_{dH^1\vert_\alpha}\ldots\iota_{dH^{i-1}\vert_\alpha}\iota_{d\sigma_\alpha}\iota_{dH^{i+1}\vert_\alpha}\ldots\iota_{dH^{k-1}\vert_\alpha}\eta^{[k]}\vert_\alpha \\
      & = \iota_{dH^1\vert_\alpha \ldots dH^{k-1}\vert_\alpha}\eta^{[k]}\vert_\alpha + \sum^{k-1}_{i=1}(-1)^{i+k} H^i\vert_\alpha \iota_{dH^1\vert_\alpha \ldots \widehat{dH^{i}\vert_\alpha} \ldots dH^{k-1}\vert_\alpha}\iota_{d\sigma_\alpha}\eta^{[k]}\vert_\alpha.
    \end{split}
\end{equation}
In the last line of the calculation, we can use the definition of $\mathcal{E}^{[k-1]}$ given in \eqref{newE-k} and write the Hamiltonian vector field in terms of $\mathcal{E}^{[k-1]}$:
\begin{equation}\label{erikdali}
   X_{H^1_\alpha, \ldots, H^{k-1}_\alpha} =  \iota_{dH^1\vert_\alpha \ldots dH^{k-1}\vert_\alpha}\eta^{[k]}\vert_\alpha + \sum^{k-1}_{i=1}(-1)^{i} H^i\vert_\alpha \iota_{dH^1\vert_\alpha \ldots \widehat{dH^{i}\vert_\alpha} \ldots dH^{k-1}\vert_\alpha}\mathcal{E}^{[k-1]}\vert_\alpha.
\end{equation}

All the terms on the right-hand side of \eqref{erikdali} has their global identifications. This forces the vector fields $\{ X_{H^1_\alpha, \ldots, H^{k-1}_\alpha}\}$ generating the local particle motion to be global. 
That is, on any  overlapping region $U_\alpha \cap U_\beta$, we have $ X_{H^1_\alpha, \ldots, H^{k-1}_\alpha} =  X_{H^1_\beta, \ldots, H^{k-1}_\beta}$. Hence, we can define a global vector field $ X_{H^1, \ldots, H^{k-1}}$ satisfying
\begin{equation}
   X_{H^1, \ldots, H^{k-1}}\vert_\alpha = X_{H^1_\alpha, \ldots, H^{k-1}_\alpha}.
\end{equation}
Using this observation, the Hamiltonian vector field we obtained in \eqref{erikdali} can be written in global form:
\begin{equation}
    X_{H^1, \ldots, H^{k-1}} = \iota_{dH^1\vert \ldots dH^{k-1}}\eta^{[k]} + \sum^{k-1}_{i=1}(-1)^{i} H^i \iota_{dH^1 \ldots \widehat{dH^{i}} \ldots dH^{k-1}}\mathcal{E}^{[k-1]}.
\end{equation}
Hence, the dynamics of the observable $F$ is determined by
\begin{equation}\label{erikdaligevrektir}
    X_{H^1, \ldots, H^{k-1}} (F) = \iota_{dH^1\ldots dH^{k-1}}\eta^{[k]}(F) + \sum^{k-1}_{i=1}(-1)^{i} H^i \iota_{dH^1 \ldots \widehat{dH^{i}} \ldots dH^{k-1}}\mathcal{E}^{[k-1]}(F).
\end{equation}
Equivalently, referring to the locally conformal Nambu--Poisson $k$-bracket we obtained in \eqref{malinkamaya} to write \eqref{erikdaligevrektir} as
\begin{equation}
    X_{H^1, \ldots, H^{k-1}} (F)=\{F,H^1, \ldots, H^{k-1}\} - F\mathcal{E}(dH^1 , \ldots, H^{k-1}).
\end{equation}

\section{Locally Conformal Generalized Poisson Setting}\label{sec:LCGP}

In this section, we will focus on another generalization of Poisson and Jacobi manifolds. This time, we will generalize the identities \eqref{Poisson-cond} and \eqref{ident-Jac} (we will give the details later). However, since the multivectors of odd order trivially satisfy the commutation rule under the Schouten--Nijenhuis bracket, we will only work on the $k$-vectors where $k=2p$. Then we will do the local conformality analysis for the manifolds whose local pictures are $2p$-generalized Poisson.

\subsection{Generalized Poisson Manifolds}

Consider a manifold $\mathcal{P}$ equipped with a skew-symmetric (see the definition \eqref{sskew}) $2p$-multilinear mapping
\begin{equation}\label{GP-n}
\{ \bullet,\bullet,\dotsm,\bullet \}^{\rm(GP)}~:C^\infty(\mathcal{P})\times \dots \times C^\infty(\mathcal{P}) \longrightarrow C^\infty(\mathcal{P}).
\end{equation}
The bracket is called a generalized Poisson $2p$-bracket if the following axioms are satisfied:
\begin{axioms} 
  \item[\textbf{GP1.}]  Generalized Jacobi identity
	\begin{equation}  \label{item:GP1}
\circlearrowright  \{F_1,F_2,\dotsm,F_{2p-1},\{F_{2p},\dotsm,F_{4p-1}\} ^{\rm(GP)} \} ^{\rm(GP)} = 0,
\end{equation}
	\item[\textbf{GP2.}]  Leibniz identity
\begin{equation}\label{item:GP2}
\{ H \cdot F_1, F_2, \dotsm ,F_{2p} \}^{\rm(GP)}=
H \cdot \{F_1, F_2, \dotsm , F_{2p} \} ^{\rm(GP)}+
\{ H, F_2, \dotsm, F_{2p} \}^{\rm(GP)} \cdot F_1
\end{equation}
\end{axioms}
 for all 
 $F_1 , F_2 , \dotsm , F_{4p-1}$ in $\mathcal{F}(P)$. Here, $\circlearrowright$ denotes the skew-symmetric cyclic sum.
A manifold admitting such a bracket is called a $2p$-generalized Poisson manifold  \cite{ArcPerelo96,AzcarragaPerel96}. We denote a $2p$-generalized Poisson manifold by two tuple 
\begin{equation}\label{suud}
(\mathcal{P},\{ \bullet,\bullet,\dotsm,\bullet \}^{\rm(GP)}).
\end{equation}

\textbf{Multivector Field Realization.}
In light of the axioms \eqref{item:GP1} and \eqref{item:GP2}, one can define a multivector field 
\begin{equation}
\eta^{[2p]}(dF_1,dF_2,\dots, dF_{2p})= \{F_1, F_2, \dotsm , F_{2p} \} ^{\rm(GP)}
\end{equation}
corresponding to a generalized Poisson bracket. Note that, in this case, the Leibniz identity is automatically satisfied while the generalized Jacobi identity \eqref{item:GP1} asks for the commutation of the $2p$-vector field by itself under the Schouten--Nijenhuis bracket that is 
\begin{equation}
[\eta^{[2p]},\eta^{[2p]}]=0.
\end{equation}
Hence, we can denote a $2p$-generalized Poisson manifold by a two-tuple 
\begin{equation}
(\mathcal{P},\eta^{[2p]}).
\end{equation}
For the case $p=1$, we arrive at a bivector field $\eta^{[2]}$ which is a Poisson bivector field \eqref{bivec-PoissonBra}; hence, one arrives at the classical Poisson manifold.

\textbf{Generalized Poisson Hamiltonian Dynamics.} Consider a $2p$-generalized Poisson manifold as given in \eqref{suud}. Then the generalized Poisson Hamiltonian dynamics determined by the set $H_1$, $H_2$,\dots $H_{2p-1}$ of Hamiltonian functions,  is defined to be  
\begin{equation}
\dot{x} = \{x, H_1, H_2, \dotsm , H_{2p-1} \}^{\rm(GP)}.
\end{equation}
For an observable $F$, the corresponding generalized Poisson Hamiltonian vector field for Hamiltonian functions $H_1$, $H_2$,\dots $H_{2p-1}$, is defined to be
\begin{equation} 
 X_{H_1 ,\dots ,H_{2p-1}}(F)=\{ F , H_1 ,\dotsm,H_{2p-1} \}^{\rm(GP)}.
\end{equation}
In terms of the interior derivative (see Appendix \ref{sec:intder}), we write the generalized Poisson Hamiltonian vector field as
\begin{equation} \label{GP-HVF}
 X_{H_1 ,\dots ,H_{2p-1}}= \iota_{dH_1\wedge \dots \wedge dH_{2p-1}} \eta^{[2p]} = \iota_{dH_1}\dots\iota_{dH_{2p-1}} \eta^{[2p]} .
\end{equation}

\subsection{Generalized Jacobi Manifolds}

Consider a manifold $\mathcal{P}$ and a skew-symmetric (see the definition \eqref{sskew}) $2p$-bracket 
\begin{equation}\label{GJ-n}
\{ \bullet,\bullet,\dotsm,\bullet \}^{\rm(GJ)}~:C^\infty(\mathcal{P})\times \dots \times C^\infty(\mathcal{P}) \longrightarrow C^\infty(\mathcal{P}) 
\end{equation}
on the space $C^\infty(\mathcal{P})$ of smooth functions on $\mathcal{P}$. 
The bracket is called a generalized Jacobi $2p$-bracket if the following axioms are satisfied for all $F_1 , F_2 , \dotsm , F_{4p-1}$ in $C^\infty(\mathcal{P})$,
\begin{axioms} 
  \item[\textbf{GJ1.}]   Generalized Jacobi identity
	\begin{equation}
\label{item:GJ1}\circlearrowright \{F_1,F_2,\dotsm,F_{2p-1},\{F_{2p},\dotsm,F_{4p-1}\} ^{\rm(GJ)} \} ^{\rm(GJ)} = 0.
\end{equation}
	\item[\textbf{GJ2.}]  First-order linear differential operator 
\begin{equation}\label{item:GJ2}
	\begin{split}
\{ H \cdot F_1, F_2, \dotsm ,F_{2p} \}^{\rm(GJ)}&=
H \cdot \{F_1, F_2, \dotsm , F_{2p} \} ^{\rm(GJ)}+
\{ H F_2, \dotsm, F_{2p} \} ^{\rm(GJ)}F_1 \\& \qquad -H \cdot F_1\{1, F_2, \dotsm, F_{2p} \}^{\rm(GJ)} .
 \end{split}
\end{equation}
\end{axioms}
Here, $\circlearrowright$ denotes the skew-symmetric cyclic sum. A manifold admitting such a bracket is called a $2p$-generalized Jacobi manifold \cite{IbanezLeonPadron98}. We denote  a $2p$-generalized Jacobi manifold  by two-tuple
\begin{equation}\label{saab}
    (\mathcal{P},\{ \bullet,\bullet,\dotsm,\bullet \}^{\rm(GJ)}). 
\end{equation}

Notice that, for $p=1$, the generalized Jacobi $2p$-bracket reduces to a Jacobi bracket \eqref{JacobiBracket}. Hence, for $p=1$ we have the classical Jacobi manifold setting given in Subsection \ref{sec:JacMan}. 

\textbf{Multivector Field Realization.} Consider a $2p$-generalized Jacobi manifold $\mathcal{P}$ as given in \eqref{saab}. Referring to this structure, we define a pair of multivector fields, namely a $(2p-1)$-vector field $\mathcal{E}^{[2p-1]}$ and a $2p$-vector field $\eta^{[2p]}$, as follows 
\begin{equation}\label{tuzla}
\begin{split}
\mathcal{E}^{[2p-1]}(dF_2,\dotsm,dF_{2p})&:=\{1,F_2,F_3,\dotsm,F_{2p}\}^{\rm(GJ)}
\\
\eta^{[2p]} (dF1,\dotsm,dF_{2p})&:=\{ F_1,\dotsm,F_{2p} \}^{\rm(GJ)} + \sum_{i=1}^{2p} (-1)^i F_i~ \mathcal{E}^{[2p-1]} (dF_1,\dotsm,\widehat{dF_i},\dotsm,dF_{2p})
\end{split}
\end{equation}
where $\widehat{dF_i}$ denotes that it is omitting. 

Referring to a pair of multivector fields $(\eta^{[2p]},\mathcal{E}^{[2p-1]})$, we can use the second line in \eqref{tuzla} to a define a $2p$-bracket:
\begin{equation}\label{genJac-brac}
    \{ F_1,\dotsm,F_{2p} \}:=\eta^{[2p]} (dF1,\dotsm,dF_{2p}) + \sum_{i=1}^{2p} (-1)^{i+1} F_i~ \mathcal{E}^{[2p-1]} (dF_1,\dotsm,\widehat{dF_i},\dotsm,dF_{2p}).
\end{equation} 
This bracket satisfies the first order differential operator condition in \eqref{item:GJ2} automatically. On the other hand, the generalized Jacobi identity \eqref{item:GJ1}  is satisfied if and only if the following identities hold \cite{Perez97}:
\begin{equation}\label{ident-genJac}
[\eta^{[2p]},\eta^{[2p]}]=-2(2p-1) \mathcal{E}^{[2p-1]}\wedge \eta^{[2p]},\qquad [\eta^{[2p]},\mathcal{E}^{[2p-1]}]=0.
\end{equation}
This equivalence motivates us to give an alternative notation for $2p$-generalized Jacobi manifolds in terms of multivector fields: 
\begin{equation}
(\mathcal{P},\eta^{[2p]},\mathcal{E}^{[2p-1]}). 
\end{equation}
It is evident that, when $p=1$,  the conditions in \eqref{ident-genJac} reduce to the conditions \eqref{ident-Jac}, those for Jacobi manifolds.

For a $2p$-generalized Jacobi manifold, if  $\mathcal{E}^{[2p-1]}$ is identically zero then we arrive at a $2p$-generalized Poisson manifold. On the other extreme case, if $\eta^{[2p]}$ is identically zero for a $2p$-generalized Jacobi manifold then $\mathcal{E}^{[2p-1]}$ determines a $(2p-1)$-generalized Poisson structure.

\textbf{Generalized Jacobi Hamiltonian Vector Field.}
Assume a $2p$-generalized Jacobi manifold $(\mathcal{P}, \eta^{[2p]}, \mathcal{E}^{[2p-1]})$. We define the Hamiltonian dynamics generated by the collection $(H_1$, $H_2$ $\dots$, $H_{2p-1})$ of Hamiltonian functions as \cite{IbanezLopezMarreroPadron2001}:
\begin{equation}\label{serk}
\begin{split}
X_{H_{1} \cdots H_{2p-1}}:=\,& X_{H_{1} \cdots H_{2p-1}}^{\eta}+\sum_{i=1}^{2p-1}(-1)^{i} H_{i}\, X_{H_{1} \dotsm \widehat{H_{i}} \dotsm H_{2p-1}}^{\mathcal{E}} 
\\=\,& \iota_{dH_{1}\wedge \dots \wedge dH_{2p-1}}\eta^{[2p]} +\sum_{i=1}^{2p-1}(-1)^{i} H_{i}\,  \iota_{dH_{1} \dotsm \widehat{dH_{i}} \dotsm dH_{2p-1}}\mathcal{E}^{[2p-1]}.
\end{split}
\end{equation}
When $\mathcal{E}^{[2p-1]}$ is zero, we have the Hamiltonian vector field  \eqref{GP-HVF} for the $2p$-generalized Poisson manifold setting. On the other hand, for $p=1$, the vector field definition in \eqref{serk} turns out to be the Hamiltonian vector field definition \eqref{HamVF-Jac} for the Jacobi manifold case.

\subsection{Locally Conformally Generalized Poisson Manifold}\label{subsec:LCGP}

In this subsection, we will extend our analysis of local conformality to generalized Poisson manifolds. To achieve this, we will use a similar method with the one we use in Subsection \ref{sec:LCNP}. That is, we will glue the local generalized Poisson charts according to their locally conformal relations.

Consider a manifold $\mathcal{P}$ with an open covering 
\begin{equation}
    \mathcal{P} = \bigsqcup_\alpha U_\alpha
\end{equation}
so that each local chart admits local $2p$-generalized Poisson structures: 
\begin{equation}\label{cicii}
(U_\alpha,\eta^{[2p]}_\alpha),\quad (U_\beta,\eta^{[2p]}_\beta), \quad (U_\gamma,\eta^{[2p]}_\gamma), \quad \dots.
\end{equation} 
Defining the open covering like this does not guarantee the equality of the local $2p$-generalized Poisson structures on the intersections. That is, the family $ \{\eta^{[2p]}_\alpha\} $ of local generalized Poisson structures cannot be glued directly. To deal with this problem, we assume the existence of a family of local functions defined on each open subset:
\begin{equation}\label{func-rel}
\sigma_\alpha: U_\alpha \to \mathbb{R}, \quad 
\sigma_\beta: U_\beta \to \mathbb{R}, \quad 
\sigma_\gamma: U_\gamma \to \mathbb{R}, \quad \dots ,
\end{equation}
whose exterior derivatives agree on the overlaps. That is, for any non-trivial intersection $U_\alpha \cap U_\beta \neq \emptyset$, we have
\begin{equation}\label{dsigma1}
d\sigma_\alpha = d\sigma_\beta.
\end{equation}
The Poincar\'{e} lemma allows us to define a Lee form $\theta$ so that its local picture on $U_\alpha$ is denoted as 
\begin{equation} \label{titaaa}
    \theta\vert_\alpha = d\sigma_\alpha.
\end{equation}

We will use the set $\{\sigma_\alpha\}$ of conformal functions defined in \eqref{func-rel} as follows. For any two subsets satisfying $ U_\alpha \cap U_\beta \neq \emptyset $, we suppose that the corresponding local generalized Poisson structures satisfy the relation
\begin{equation}\label{ustaa}
    e^{-(2p-1)\sigma_\alpha}\,\eta^{[2p]}_\alpha \;=\; e^{-(2p-1)\sigma_\beta}\,\eta^{[2p]}_\beta .
\end{equation}
These conformal equalities in \eqref{ustaa} determine the transition scalars
\begin{equation} \label{transition2}
\rho_{\beta \alpha} = \frac{e^{(2p-1)\sigma_\alpha}}{e^{(2p-1)\sigma_\beta}} = e^{-(2p-1)(\sigma_\beta - \sigma_\alpha)},
\end{equation}
satisfying the cocycle condition
\begin{equation}\label{cocycle2}
\rho_{\beta \alpha} \, \rho_{\alpha \gamma} = \rho_{\beta \gamma}.
\end{equation}
Accordingly,  the local generalized Poisson structures $ \{\eta^{[2p]}_\alpha\} $ can be glued up to a line bundle valued $2p$-vector field.

In analogy with the Poisson and Nambu--Poisson cases, we refer to the construction above---namely, the collection $\{(U_\alpha,\eta_\alpha^{[2p]})\}$
of local $2p$-generalized Poisson charts together with the set of conformal functions 
$\{\sigma_\alpha\}$ satisfying the compatibility condition \eqref{ustaa}---as a  locally conformally $2p$-generalized Poisson manifold. We denote a locally conformally $2p$-generalized Poisson manifold by 
\begin{equation}
    (\mathcal{P},U_\alpha,\eta_\alpha^{[2p]},\sigma_\alpha).
\end{equation}

Even though it may not possible to glue the local generalized Poisson structures in \eqref{cicii} to a real-valued global object, we can define the $2p$-vector fields
\begin{equation}\label{gs1}
\eta^{[2p]}|_\alpha := e^{-(2p-1)\sigma_\alpha}\,\eta^{[2p]}_\alpha
\end{equation}
which satisfy the compatibility condition 
\begin{equation}\label{bibii}
\eta^{[2p]}|_\alpha = \eta^{[2p]}|_\beta 
\end{equation}
as a result of the relation \eqref{ustaa}. Thus one can define a globally well-defined $2p$-vector field $ \eta^{[2p]} $, whose local expression on any local chart $ U_\alpha $ is $ \eta^{[2p]}|_\alpha $. We note that this global $2p$-vector field is not necessarily a generalized Poisson $2p$-vector field. Instead, it determines a $2p$-generalized Jacobi manifold structure as we shall examine in the sequel. 

Let us analyze the structure on an arbitrary local chart $U_\alpha$. Assume that we are given local functions $F^i_\alpha$, chosen compatibly with the locally conformal structure, which glue to global functions $F^i\in C^\infty(\mathcal{P})$
according to
 \begin{equation}\label{glue-func-n}
    F^i\vert_\alpha = e^{\sigma_\alpha} F^i_\alpha ,\qquad i=1,\dotsm,2p.
 \end{equation}
We compute the local generalized Poisson $2p$-bracket $\{\bullet,\dotsm,\bullet,\bullet\}_\alpha$ determined by the  local generalized Poisson $2p$-vector field $\eta^{[2p]}_\alpha$ in terms $\eta^{[2p]}\vert_\alpha$ in \eqref{gs1} and $F^i\vert_\alpha$ in \eqref{glue-func-n} as 
 \begin{equation}\label{glue-n-genPois-brac}
	\begin{split}
		& \{F^1_\alpha,\dotsm,F^{2p}_\alpha\}_\alpha = \eta_\alpha(dF^1_\alpha,\dotsm,dF^{2p}\vert_\alpha) \\
		&\qquad= e^{(2p-1)\sigma_\alpha}\eta\vert_\alpha(d(e^{-\sigma_\alpha}F^1\vert_\alpha),\dotsm,d(e^{-\sigma_\alpha}F^{2p}\vert_\alpha)) \\
		&\qquad = e^{-\sigma_\alpha}\eta\vert_\alpha(dF^1\vert_\alpha - F^1\vert_\alpha d\sigma_\alpha,\dotsm,dF^{2p}\vert_\alpha - F^{2p}\vert_\alpha d\sigma_\alpha).
	\end{split}
\end{equation}
We multiply both sides of \eqref{glue-n-genPois-brac} by $e^{\sigma_\alpha}$ to obtain an equation where each term has its own global realization:
\begin{equation}\label{gluedGP}
	\begin{split}
		e^{\sigma_\alpha}\{F^1_\alpha,\dotsm,F^{2p}_\alpha\}_\alpha &= \eta\vert_\alpha(dF^1\vert_\alpha,\dotsm,dF^{2p}\vert_\alpha) - F^1\vert_\alpha\eta\vert_\alpha(d\sigma_\alpha,dF^2\vert_\alpha,\dotsm,dF^{2p}\vert_\alpha) \\
        &\qquad - F^{2}\vert_\alpha\eta\vert_\alpha(dF^1\vert_\alpha,d\sigma_\alpha,dF^3\vert_\alpha,\dotsm,dF^{2p-1}\vert_\alpha)\\
        &\qquad - F^{3}\vert_\alpha\eta\vert_\alpha(dF^1\vert_\alpha,dF^2\vert_\alpha,d\sigma_\alpha,dF^4\vert_\alpha,\dotsm,dF^{2p-1}\vert_\alpha)\\
		&\qquad - \dotsm \\
        &\qquad - F^{2p}\vert_\alpha\eta\vert_\alpha(dF^1\vert_\alpha,\dotsm,dF^{2p-1}\vert_\alpha,d\sigma_\alpha) \\
        & = \eta\vert_\alpha(dF^1\vert_\alpha,\dotsm,dF^{2p}\vert_\alpha) + F^1\vert_\alpha\eta\vert_\alpha(dF^2\vert_\alpha,\dotsm,dF^{2p}\vert_\alpha,d\sigma_\alpha) \\
        & \qquad - F^2\vert_\alpha\eta\vert_\alpha(dF^1\vert_\alpha,dF^3\vert_\alpha,\dotsm,dF^{2p}\vert_\alpha,d\sigma_\alpha) \\
        & \qquad + F^3\vert_\alpha\eta\vert_\alpha(dF^1\vert_\alpha,dF^2\vert_\alpha,dF^4\vert_\alpha,\dotsm,dF^{2p}\vert_\alpha,d\sigma_\alpha) \\
        & \qquad - \dotsm \\
        & \qquad + F^{2p}\vert_\alpha\eta\vert_\alpha(dF^1\vert_\alpha,\dotsm,dF^{2p-1}\vert_\alpha,d\sigma_\alpha).
	\end{split}
\end{equation}
Using the Lee form $\theta$ defined in \eqref{titaaa}, this equation turns out to be
\begin{equation}\label{gluedGP1}
    \begin{split}
         e^{\sigma_\alpha}\{F^1_\alpha,\dotsm,F^{2p}_\alpha\}_\alpha &= \eta\vert_\alpha(dF^1\vert_\alpha,\dotsm,dF^{2p}\vert_\alpha) + F^1\vert_\alpha\iota_{\theta\vert_\alpha}\eta\vert_\alpha(dF^2\vert_\alpha,\dotsm,dF^{2p}\vert_\alpha) \\
        & \qquad - F^2\vert_\alpha\iota_{\theta\vert_\alpha}\eta\vert_\alpha(dF^1\vert_\alpha,dF^3\vert_\alpha,\dotsm,dF^{2p}\vert_\alpha) \\
        & \qquad + F^3\vert_\alpha\iota_{\theta\vert_\alpha}\eta\vert_\alpha(dF^1\vert_\alpha,dF^2\vert_\alpha,dF^4\vert_\alpha,\dotsm,dF^{2p}\vert_\alpha) \\
        & \qquad - \dotsm \\
        & \qquad + F^{2p}\vert_\alpha\iota_{\theta\vert_\alpha}\eta\vert_\alpha(dF^1\vert_\alpha,\dotsm,dF^{2p-1}\vert_\alpha).
    \end{split}
\end{equation}
See that both sides of equation \eqref{gluedGP1} can be glued to their corresponding global expressions. More explicitly, for the left-hand side, \eqref{glue-func-n} allows us to write a global $2p$-bracket whose restriction to the local chart is
\begin{equation}\label{sennur1}
    \{F^1,\dotsm,F^{2p}\}\vert_\alpha =  e^{\sigma_\alpha}
    \{F^1_\alpha,\dotsm,F^{2p}_\alpha\}_\alpha.
 \end{equation}
 Furthermore, for the right-hand side, we define a $(2p-1)$-vector field
\begin{equation}\label{newE-n}
    \mathcal{E}^{[2p-1]}\vert_\alpha = \iota_{d\sigma_\alpha}\eta^{[2p]}\vert_\alpha = \iota_{\theta\vert_\alpha}\eta^{[2p]}\vert_\alpha.
\end{equation}
That gives the locally conformality condition as  
\begin{equation}\label{newE-n++}
\mathcal{E}^{[2p-1]}= \iota_{\theta}\eta^{[2p]}.
\end{equation}
Hence, we can write the global expression of each term in the right-hand side of the calculation \eqref{gluedGP1} as well. As a result, referring to the $2p$-vector field $\eta ^{[2p]}$ in \eqref{gs1} and the bivector field $\mathcal{E}^{[2]}$ in \eqref{newE-n++}, we determine a global $2p$-bracket of the global functions $F^1,\dotsm,F^{2p}$ on $\mathcal{P}$ as
\begin{equation}\label{LCGP}
 \begin{split}
\{ F^1,\dotsm,F^{2p} \} &= \eta ^{[2p]}(dF^1,\dotsm,dF^{2p}) + \sum_{i=1}^{2p}(-1)^{i-1} F^i\mathcal{E}^{[2p-1]}(dF^1,\dotsm,\widehat{dF^i},\dotsm,dF^{2p}) .    
 \end{split}
\end{equation}
We refer to this $2p$-bracket as a locally conformal generalized Poisson $2p$-bracket. Observe that the $2p$-bracket \eqref{LCGP} is generated by the pair $(\eta ^{[2p]},\mathcal{E}^{[2p-1]})$. Accordingly, it coincides with the generalized Jacobi $2p$-bracket \eqref{genJac-brac}. In the following theorem, we see that every locally conformally $2p$-generalized Poisson manifold is, in fact, a $2p$-generalized Jacobi manifold.

\begin{theorem} \label{prop-LCGP-Jacobi} A locally conformally $2p$-generalized Poisson manifold is a $2p$-generalized Jacobi manifold. More precisely, the pair $(\eta^{[2p]},\mathcal{E}^{[2p-1]})$, where $\eta^{[2p]}$ is the $2p$-vector field in \eqref{gs1} and $\mathcal{E}^{[2p-1]}$ is the $(2p-1)$-vector field in \eqref{newE-n}, is a $2p$-generalized Jacobi pair.  
\end{theorem}

\begin{proof}
To see that $(\eta^{[2p]},\mathcal{E}^{[2p-1]})$ is a $2p$-generalized Jacobi pair, we need to show that the identities in \eqref{ident-genJac} are satisfied. We first consider the local generalized Poisson $2p$-vector field
$\eta^{[2p]}_\alpha$. Since $\eta^{[2p]}_\alpha$ defines a $2p$-generalized
Poisson structure, its Schouten--Nijenhuis bracket vanishes, that is, $[\eta^{[2p]}_\alpha,\eta^{[2p]}_\alpha]=0$. With this identity in mind, we proceed to the following calculation
\begin{equation}\label{gener-calc-2p}
	\begin{split}
	[\eta^{[2p]}_\alpha,\eta^{[2p]}_\alpha] & = [e^{{(2p-1)}\sigma_\alpha}\eta^{[2p]}\vert_\alpha,e^{{(2p-1)}\sigma_\alpha}\eta^{[2p]}\vert_\alpha] 
	\\
    & = e^{{(2p-1)}\sigma_\alpha}[e^{{(2p-1)}\sigma_\alpha}\eta^{[2p]}\vert_\alpha,\eta^{[2p]}\vert_\alpha] + \iota_{d(e^{{(2p-1)}\sigma_\alpha})}(e^{{(2p-1)}\sigma_\alpha}\eta^{[2p]}\vert_\alpha) \wedge \eta\vert_\alpha \\
    & = e^{{(2p-1)}\sigma_\alpha}[e^{{(2p-1)}\sigma_\alpha}\eta^{[2p]}\vert_\alpha,\eta^{[2p]}\vert_\alpha] + {(2p-1)}e^{2{(2p-1)}\sigma_\alpha}\iota_{d\sigma_\alpha}\eta^{[2p]}\vert_\alpha \wedge \eta^{[2p]}\vert_\alpha \\
    & = e^{{(2p-1)}\sigma_\alpha}( e^{{(2p-1)}\sigma_\alpha}[\eta^{[2p]}\vert_\alpha,\eta^{[2p]}\vert_\alpha] + \iota_{d(e^{{(2p-1)}\sigma_\alpha})}\eta^{[2p]}\vert_\alpha \wedge \eta^{[2p]}\vert_\alpha) \\
    & \qquad + {(2p-1)}e^{2{(2p-1)}\sigma_\alpha}\iota_{d\sigma_\alpha}\eta^{[2p]}\vert_\alpha \wedge \eta^{[2p]}\vert_\alpha \\
    & = e^{2{(2p-1)}\sigma_\alpha}[\eta^{[2p]}\vert_\alpha,\eta^{[2p]}\vert_\alpha] + 2{(2p-1)} e^{2{(2p-1)}\sigma_\alpha}\iota_{d\sigma_\alpha}\eta^{[2p]}\vert_\alpha \wedge \eta^{[2p]}\vert_\alpha \\
    & = e^{2{(2p-1)}\sigma_\alpha}[\eta^{[2p]}\vert_\alpha,\eta^{[2p]}\vert_\alpha] + 2{(2p-1)} e^{2{(2p-1)}\sigma_\alpha}\mathcal{E}^{[2p-1]}\vert_\alpha \wedge \eta^{[2p]}\vert_\alpha,
	\end{split}
\end{equation}
Since the left-hand side is zero, we obtain
\begin{equation}\label{localoca}
[\eta^{[2p]}\big|_\alpha,\eta^{[2p]}\big|_\alpha]
= -2(2p-1)\,
\mathcal{E}^{[2p-1]}\big|_\alpha \wedge \eta^{[2p]}\big|_\alpha,
\end{equation}
The global expressions of each term in \eqref{localoca} give the following global equality
\begin{equation}\label{beril}
	[\eta^{[2p]},\eta^{[2p]}] = -2(2p-1)\mathcal{E}^{[2p-1]} \wedge \eta^{[2p]},
\end{equation}
which is exactly the same as the first identity in \eqref{ident-genJac}.
The second identity in \eqref{ident-genJac} follows from the distributive properties of interior derivative on Schouten--Nijenhuis bracket and on wedge products: 
 \begin{equation}\label{Cals-2}
    \begin{split}
             [\mathcal{E}^{[2p-1]},\eta^{[2p]}] & =  [\iota_\theta\eta^{[2p]},\eta^{[2p]}] = -\frac{1}{2}\iota_\theta[\eta^{[2p]},\eta^{[2p]}] = -\frac{1}{2}\iota_\theta\big(-2(2p-1)\mathcal{E}^{[2p-1]} \wedge \eta^{[2p]}\big)  \\
     & = (2p-1)\iota_\theta(\iota_\theta\eta^{[2p]} \wedge \eta^{[2p]}) = (2p-1)\big(\iota_\theta\iota_\theta\eta^{[2p]} \wedge \eta^{[2p]} + \iota_\theta\eta^{[2p]} \wedge \iota_\theta\eta^{[2p]}\big) \\
     & = 0.
    \end{split}
	\end{equation}
The details of the calculation are as follows: In the first equality we wrote the global expression \eqref{newE-n++} of the $(2p-1)$-vector field $\mathcal{E}^{[2p-1]}$. In the second equality, we applied the distribution property 
    \eqref{iota-SN-m-n} of the right contraction on the Schouten--Nijenhuis bracket. For the third equality, we employed \eqref{beril}, While in the fifth equality, we used the distribution property 
    \eqref{iota-SN-2-2-X-} of the right contraction on the wedge product.
In the last equation, both terms vanish identically due to the skew-symmetry. Hence, we conclude that $(\mathcal{P},\eta^{[2p]},\mathcal{E}^{[2p-1]})$ is indeed a $2p$-generalized Jacobi manifold.

\end{proof}

\textbf{Locally Conformal Generalized Poisson Hamiltonian Dynamics.}
We once more start with a local $2p$-generalized Poisson manifold $(U_\alpha,\eta_\alpha^{[2p]})$, where $\eta_\alpha^{[2p]}$ is a local generalized Poisson $2p$-vector field satisfying \eqref{gs1}. We further consider that $H_\alpha^i$ for $i=1,\ldots,2p-1$ are local Hamiltonian functions satisfying the global identification in \eqref{glue-func-n}. Applying the definition of the generalized Poisson Hamiltonian vector field
\eqref{GP-HVF}, the local generalized Poisson Hamiltonian vector field on
$(U_\alpha,\eta^{[2p]}_\alpha)$ is written as
\begin{equation}\label{dalinda2}
X_{H^1_\alpha,\dots,H^{2p-1}_\alpha}
= \iota_{dH^1_\alpha \wedge \dots \wedge dH^{2p-1}_\alpha}\,
\eta^{[2p]}_\alpha .
\end{equation}
To understand its global behavior, we first compute
\begin{equation}
    \begin{split}
        X_{H^1_\alpha, \ldots, H^{2p-1}_\alpha}&= \iota_{dH^1_\alpha\wedge \ldots \wedge dH^{2p-1}_\alpha} \eta^{[2p]}_\alpha = \iota_{dH^1_\alpha}\ldots \iota_{dH^{2p-1}_\alpha}\eta^{[2p]}_\alpha \\
     & = \iota_{d(e^{-\sigma_\alpha}H^1\vert_\alpha)}\ldots \iota_{d(e^{-\sigma_\alpha}H^{2p-1}\vert_\alpha)} e^{({2p-1})\sigma_\alpha}\eta^{[2p]}\vert_\alpha \\
     &  = \iota_{(dH^1\vert_\alpha-H^1\vert_\alpha d\sigma_\alpha)}\ldots\iota_{(dH^{2p-1}\vert_\alpha-H^{2p-1}\vert_\alpha d\sigma_\alpha)}\eta^{[2p]}\vert_\alpha \\
     & = \iota_{dH^1\vert_\alpha}\ldots\iota_{dH^{2p-1}\vert_\alpha}\eta^{[2p]}\vert_\alpha \\
     & \qquad - \sum^{2p-1}_{i=1}H^i\vert_\alpha \iota_{dH^1\vert_\alpha}\ldots\iota_{dH^{i-1}\vert_\alpha}\iota_{d\sigma_\alpha}\iota_{dH^{i+1}\vert_\alpha}\ldots\iota_{dH^{2p-1}\vert_\alpha}\eta^{[2p]}\vert_\alpha \\
      & = \iota_{dH^1\vert_\alpha \ldots dH^{2p-1}\vert_\alpha}\eta^{[2p]}\vert_\alpha + \sum^{2p-1}_{i=1}(-1)^{i} H^i\vert_\alpha \iota_{dH^1\vert_\alpha \ldots \widehat{dH^{i}\vert_\alpha} \ldots dH^{2p-1}\vert_\alpha}\iota_{d\sigma_\alpha}\eta^{[2p]}\vert_\alpha.
    \end{split}
\end{equation}
Notice that the definition of $\mathcal{E}^{[2p-1]}$ given in \eqref{newE-n} allows us write the Hamiltonian vector field in terms of $\mathcal{E}^{[2p-1]}$:
\begin{equation}\label{erikdali2}
   X_{H^1_\alpha, \ldots, H^{2p-1}_\alpha} =  \iota_{dH^1\vert_\alpha \ldots dH^{2p-1}\vert_\alpha}\eta^{[2p]}\vert_\alpha + \sum^{2p-1}_{i=1}(-1)^{i} H^i\vert_\alpha \iota_{dH^1\vert_\alpha \ldots \widehat{dH^{i}\vert_\alpha} \ldots dH^{2p-1}\vert_\alpha}\mathcal{E}^{[2p-1]}\vert_\alpha.
\end{equation}
We can glue the right-hand side of \eqref{erikdali2} to a global expression. It follows that the vector fields $ X_{H^1_\alpha, \ldots, H^{2p-1}_\alpha}$ need to coincide on the overlappings, that is, we have $X_{H^1_\alpha, \ldots, H^{2p-1}_\alpha} = X_{H^1_\beta, \ldots, H^{2p-1}_\beta}$ on $U_\alpha \cap U_\beta$. Then, a global vector field $X_{H^1, \ldots, H^{2p-1}}$ can be defined by
\begin{equation}
   X_{H^1, \ldots, H^{2p-1}}\vert_\alpha = X_{H^1_\alpha, \ldots, H^{2p-1}_\alpha},
\end{equation}
which gives the global expression
\begin{equation}
    X_{H^1, \ldots, H^{2p-1}} = \iota_{dH^1\vert \ldots dH^{2p-1}}\eta^{[2p]} + \sum^{2p-1}_{i=1}(-1)^{i} H^i \iota_{dH^1 \ldots \widehat{dH^{i}} \ldots dH^{2p-1}}\mathcal{E}^{[2p-1]}.
\end{equation}
Hence, the dynamics of the observable $F$ can be written as
\begin{equation}\label{erikdaligevrektir2}
    X_{H^1, \ldots, H^{2p-1}} (F) = \iota_{dH^1\ldots dH^{2p-1}}\eta^{[2p]}(F) + \sum^{2p-1}_{i=1}(-1)^{i} H^i \iota_{dH^1 \ldots \widehat{dH^{i}} \ldots dH^{2p-1}}\mathcal{E}^{[2p-1]}(F).
\end{equation}
Moreover, using the definition \eqref{LCGP} of the locally conformal generalized Poisson $2p$-bracket, we can rewrite \eqref{erikdaligevrektir2} as
\begin{equation}
    X_{H^1, \ldots, H^{2p-1}} (F)=\{F,H^1, \ldots, H^{2p-1}\} - F\mathcal{E}(dH^1 , \ldots, H^{2p-1}).
\end{equation}

\subsection{From Locally Conformal $3$-Nambu--Poisson to Locally Conformal $4$-Generalized Poisson Structures}\label{sec:4-3}

Assume a locally conformally $3$-Nambu--Poisson manifold $(\mathcal{P},\eta^{[3]},\mathcal{E}^{[2]})$ as given in \eqref{LC-3-NP}. 
Recall that the locally conformal character, represented by the presence of the Lee form $\theta$, is encoded in the relation
\begin{equation}\label{cemile}
  \mathcal{E}^{[2]}=-\iota_\theta \eta^{[3]}.
\end{equation}
Let $H^1$ and $H^2$ be two Hamiltonian functions.
Following the construction outlined in Subsection~\ref{sec:bi-jac}, we obtain
a pair of compatible Jacobi structures given by
\begin{equation}\label{coskun}
\begin{split}
\Lambda^1 &= -\iota_{dH^1}\eta^{[3]} - H^1\mathcal{E}^{[2]},
\qquad
Z^1 = -\iota_{dH^1}\mathcal{E}^{[2]},\\
\Lambda^2 &= \iota_{dH^2}\eta^{[3]} + H^2\mathcal{E}^{[2]},
\qquad
Z^2 = \iota_{dH^2}\mathcal{E}^{[2]}.
\end{split}
\end{equation}
Based on these data, we construct a locally conformal $4$-generalized Poisson
structure on $\mathcal{P}$.
Following \cite{IbanezLeonPadron98}, we define a $4$-vector field $\eta^{[4]}$
and a $3$-vector field $\mathcal{E}^{[3]}$ on $\mathcal{P}$ by
\begin{equation}
\eta^{[4]}=\Lambda^{1}\wedge\Lambda^{2},
\qquad
\mathcal{E}^{[3]}=Z^1\wedge\Lambda^{2}+Z^2\wedge\Lambda^{1}.
\end{equation}
We show that the pair $(\eta^{[4]},\mathcal{E}^{[3]})$ defines a locally
conformal $4$-generalized Poisson structure.
To verify this, let us first substitute the locally conformal relation~\eqref{cemile} into~\eqref{coskun}. 
For the $4$-vector field $\eta^{[4]}$, we compute
\begin{equation}\label{sermet2}
\begin{split}
\eta^{[4]} & = \Lambda^{1} \wedge \Lambda^{2} = \big(-\iota_{dH^1}\eta^{[3]} - H^1\mathcal{E}^{[2]} \big)\wedge \big(\iota_{dH^2}\eta^{[3]}+ H^2\mathcal{E}^{[2]}  \big)\\ 
 & = -\iota_{dH^1}\eta^{[3]} \wedge \iota_{dH^2}\eta^{[3]} - \iota_{dH^1}\eta^{[3]} \wedge H^2\mathcal{E}^{[2]} - H^1\mathcal{E}^{[2]} \wedge \iota_{dH^2}\eta^{[3]} - H^1\mathcal{E}^{[2]} \wedge H^2\mathcal{E}^{[2]} \\
 & = -\iota_{dH^1}\eta^{[3]} \wedge \iota_{dH^2}\eta^{[3]} + H^2\iota_{dH^1}\eta^{[3]} \wedge \iota_{\theta}\eta^{[3]} + H^1\iota_{\theta}\eta^{[3]} \wedge \iota_{dH^2}\eta^{[3]} - H^1H^2\iota_{\theta}\eta^{[3]} \wedge \iota_{\theta}\eta^{[3]},
\end{split}
\end{equation}
where we repeatedly used the distributive property of the interior product over the wedge product.

Next, for the $3$-vector field $\mathcal{E}^{[3]}$, we obtain
\begin{equation}
\begin{split}
\mathcal{E}^{[3]} = & ~Z^1 \wedge \Lambda^{2} + Z^2 \wedge \Lambda^{1} 
\\ = & ~ (-\iota_{dH^1}\mathcal{E}^{[2]})\wedge (\iota_{dH^2}\eta^{[3]}+ H^2\mathcal{E}^{[2]}) + \iota_{dH^2}\mathcal{E}^{[2]}\wedge (-\iota_{dH^1}\eta^{[3]} - H^1\mathcal{E}^{[2]})
\\
=& ~ (\iota_{dH^1}\iota_\theta \eta^{[3]})\wedge (\iota_{dH^2}\eta^{[3]}- H^2 \iota_\theta \eta^{[3]}) - \iota_{dH^2}\iota_\theta \eta^{[3]}\wedge (-\iota_{dH^1}\eta^{[3]} + H^1\iota_\theta \eta^{[3]})
\\
=& ~  \iota_{dH^1}\iota_\theta \eta^{[3]} \wedge \iota_{dH^2}\eta^{[3]} -H^2\iota_{dH^1}\iota_\theta \eta^{[3]}\wedge\iota_\theta \eta^{[3]} + \iota_{dH^2}\iota_\theta \eta^{[3]}\wedge \iota_{dH^1}\eta^{[3]} \\
&\qquad \qquad - H^1 \iota_{dH^2}\iota_\theta \eta^{[3]}\wedge\iota_\theta \eta^{[3]} 
\\
=&-\iota_\theta  \iota_{dH^1}\eta^{[3]}\wedge \iota_{dH^2}\eta^{[3]}  + H^2\iota_\theta\iota_{dH^1} \eta^{[3]}\wedge\iota_\theta \eta^{[3]}  - \iota_{dH^1}\eta^{[3]} \wedge \iota_\theta \iota_{dH^2}\eta^{[3]} \\
&\qquad \qquad + H^1 \iota_\theta\iota_{dH^2} \eta^{[3]}\wedge\iota_\theta \eta^{[3]} 
\\
=& - \iota_\theta (\iota_{dH^1}\eta^{[3]}\wedge \iota_{dH^2}\eta^{[3]}) + \iota_\theta (H^2 \iota_{dH^1} \eta^{[3]}\wedge\iota_\theta \eta^{[3]} )  + \iota_\theta (H^1 \iota_{dH^2} \eta^{[3]}\wedge\iota_\theta \eta^{[3]} ). 
\end{split}
\end{equation}
We observe that $\mathcal{E}^{[3]}$ lies in the image of the contraction operator $\iota_\theta$. Moreover, a direct comparison with the explicit form of $\eta^{[4]}$ given in \eqref{sermet2} reveals a close structural correspondence between the two, indicating that $\mathcal{E}^{[3]}$ can be regarded as the image of $\eta^{[4]}$ under $\iota_\theta$, that is
\begin{equation}
    \begin{split}
\mathcal{E}^{[3]} &= - \iota_\theta (\iota_{dH^1}\eta^{[3]}\wedge \iota_{dH^2}\eta^{[3]}) + \iota_\theta (H^2 \iota_{dH^1} \eta^{[3]}\wedge\iota_\theta \eta^{[3]} )  + \iota_\theta (H^1 \iota_{dH^2} \eta^{[3]}\wedge\iota_\theta \eta^{[3]} )\\
&=\iota_\theta \big( -\iota_{dH^1}\eta^{[3]} \wedge \iota_{dH^2}\eta^{[3]} + H^2\iota_{dH^1}\eta^{[3]} \wedge \iota_{\theta}\eta^{[3]} + H^1\iota_{\theta}\eta^{[3]} \wedge \iota_{dH^2}\eta^{[3]} \\
&\qquad \qquad - H^1H^2\iota_{\theta}\eta^{[3]} \wedge \iota_{\theta}\eta^{[3]} \big)
\\
& = \iota_\theta \eta^{[4]}.
\end{split}
\end{equation} 
This is precisely the locally conformal condition given in~\eqref{newE-n++} for order $4$ (that is, $p=2$).  

Therefore, we conclude that the triple $(\mathcal{P}, \eta^{[4]}, \theta)$ defines a locally conformally $4$-generalized Poisson manifold. 
This construction demonstrates how locally conformal $3$-Nambu--Poisson dynamics naturally lift to a locally conformal $4$-generalized Poisson structure preserving the underlying conformal geometry.

\section*{Conclusion}

In this work, we have investigated the  globalization problem  within the framework of  locally conformal geometry  for Hamiltonian systems defined on  Nambu--Poisson  and  generalized Poisson  manifolds. 
Building upon this setting, we have proposed two new geometric structures:  locally conformally Nambu--Poisson manifolds (in Section \ref{sec:LCNPois}) and  locally conformally generalized Poisson  manifolds (in Section \ref{sec:LCGP}). 
We demonstrated that a locally conformally Nambu--Poisson manifold naturally acquires the structure of a  Nambu--Jacobi manifold (Theorem \ref{memre2}), while a locally conformally generalized Poisson manifold gives rise to a  generalized Jacobi manifold (Theorem \ref{prop-LCGP-Jacobi}).

Furthermore, we established explicit contraction relations between locally conformal Nambu--Poisson structures of different orders, revealing a hierarchical correspondence that is summarized in Diagram~\ref{geomdiagram--n}. 
We also developed Hamiltonian-type dynamics on both locally conformally Nambu--Poisson (in Subsection  \ref{sec:Ham-LCNP}) and locally conformally generalized Poisson manifolds (in Subsection  \ref{subsec:LCGP}). 
Special attention was devoted to the case of locally conformal $3$-Nambu--Poisson structures (in Subsections \ref{subsec:LC-3-NP}, \ref{sec:bi-jac}) and \ref{sec:4-3}) where the formalism admits the definition of  locally conformal bi-Hamiltonian systems and on locally conformal analysis of complete integrability. 

As a future direction, we intend to extend this geometric globalization framework by incorporating \emph{time as an explicit variable}, thereby enabling a locally conformal treatment of non-autonomous and irreversible Hamiltonian dynamics.

\section*{Acknowledgements}
We gratefully acknowledge the support of TÜBİTAK through the 2218 Postdoctoral Research Fellowship Program, Project No. 124C484. We also express our sincere gratitude to Oğul Esen for his detailed reading and valuable comments on this work.

\section{Appendix}

\subsection{Lichnerowicz-deRham (abbreviated as LdR) differential.}\label{sec:LdR}
Consider now an arbitrary manifold $\mathcal{Q}$. Fix a closed one-form $\theta$.  The Lichnerowicz-deRham differential is defined as
\begin{equation} \label{LdR-Diff}
d_\theta: \Gamma^k(\mathcal{Q}) \rightarrow  \Gamma^{k+1}(\mathcal{Q}) : \beta \mapsto d\beta-\theta\wedge\beta,
\end{equation}
where $d$ denotes the exterior (deRham) derivative \cite{GuLi84}. See that the Lichnerowicz-deRham (abbreviated as LdR) differential is a nilpotent operator, i.e.,  $d_{\theta}^2=0$. See, for example, \cite{ChanMurp19} for more properties and some details on cohomological discussions.

\subsection{Musical Mappings} \label{sec:music}
 We shall make frequent use of the notation $\Gamma^m(\mathcal{Q})$ for the space of $k$-forms on a manifold $\mathcal{Q}$, $\mathfrak{X}^m(\mathcal{Q})$ for the space of k-vector fields on $\mathcal{Q}$, and $\mathcal{F}(\mathcal{Q})$ for the space of smooth functions on $\mathcal{Q}$.

We shall use two musical mappings; a musical flat mapping, and a musical sharp mapping. The former is defined, given a two-form $\Omega \in \Gamma^2(\mathcal{Q})$, as 
\begin{equation} \label{flat}
\Omega^\flat: \mathfrak{X}^1(\mathcal{Q})\to \Gamma^1(\mathcal{Q}), \qquad \Omega^\flat(X) (Y)=(\iota_X \Omega) (Y): = \Omega(X,Y).
\end{equation}
It happens that $\Omega^\flat$ is invertible if $\Omega$ is non-degenerate. In this case, the inverse is denoted by $\Omega^\sharp:\Gamma^1(\mathcal{Q}) \to\mathfrak{X}^1(\mathcal{Q})$. If, furthermore, $\mathcal{Q}$ is a symplectic manifold with the symplectic two-form $\Omega \in \Gamma^2(\mathcal{Q})$, then 
\begin{equation}
\alpha(X)=\Omega(\Omega^{\sharp }(\alpha),X).
\end{equation}
for any one-form $\alpha \in \Gamma^1(\mathcal{Q})$, and any vector field $X\in \mathfrak{X}^1(\mathcal{Q})$. 

On the other hand, given a $m$-vector  field $\mathcal{X}$ in  $\mathfrak{X}^k(\mathcal{Q})$, the (right) musical sharp mapping is defined to be
\begin{equation} \label{sharp}
\mathcal{X}^\sharp: \Gamma^{m-1}(\mathcal{Q})\longrightarrow \mathfrak{X}^1(\mathcal{Q}), \qquad \mathcal{X}^\sharp(\beta) (\alpha_1,\alpha_2,\dots,\alpha_{k-1}) = \mathcal{X}(\alpha_1,\alpha_2,\dots,\alpha_{k-1}, \beta).
\end{equation}
We do note that the contraction operator in \eqref{flat} is the left interior derivative, whereas the one in \eqref{sharp} is the right interior derivative, \cite{Marle-SN}. Accordingly, given a non-degenerate $\Omega \in \Gamma^2(\mathcal{Q})$, let a bivector field be given by
\begin{equation} 
\Lambda(\alpha,\beta):=\Omega(\Omega^\sharp(\alpha),\Omega^\sharp(\beta))
=\langle \alpha, \Omega^\sharp(\beta))\rangle. 
\end{equation}
It then follows at once that $\Lambda^\sharp = \Omega^\sharp$.

\subsection{Schouten--Nijenhuis Algebra} \label{sec:SN}

  As indicated in \cite{AlGu96}, the Nambu--Poisson multivector fields are decomposable. Referring to this fact, we state the  Schouten--Nijenhuis bracket. For multivector fields $\mathcal{X}$ of order $m \geq 1$ and   $\mathcal{Y}$ of order $n \geq 1$ given by
\begin{equation} \label{decompi}
\mathcal{X}=X_1\wedge \dots\wedge X_m, \qquad \mathcal{Y}=Y_1\wedge \dots\wedge Y_n,
\end{equation}
the Schouten--Nijenhuis bracket in terms of the Jacobi-Lie bracket is as follows
\begin{equation} \label{SN-def1}
[\mathcal{X},\mathcal{Y}]_{SN}:=\sum_{i,j}(-1)^{i+j}[X_i,Y_j]\wedge X_1\wedge \dots\wedge\widehat X_{i} \wedge \dots \wedge X_m\wedge Y_1\wedge \dots \wedge \widehat Y_{j} \wedge \dots \wedge Y_n, 
\end{equation}
where $\widehat{X}_{i}$ and $\widehat{Y}_{j}$ denote the omission of the elements $X_i$ and $Y_j$, respectively, from the sequences. Notice that the resulting multivector field $[\mathcal{X},\mathcal{Y}]_{SN}$ is of order $m+n-1$. For the case $m=1$ and $n=1$, one has two vector fields and it can be quickly observed that the Schouten--Nijenhuis bracket becomes the Jacobi-Lie bracket: \begin{equation}
[X,Y]_{SN}=[X,Y]=\mathcal{L}_X Y.
\end{equation}

For the case $m \geq 1$ and $n=0$, that is, if there is a multivector field $\mathcal{X}$ of order $m\geq 1$ and a function $F$, the Schouten--Nijenhuis bracket is defined as 
\begin{equation}\label{X-F-SN}
    [\mathcal{X},F]_{SN} := \sum^m_{i=1}(-1)^{i+1}X_i(F) X_1\wedge \dots\wedge\widehat X_{i} \wedge \dots \wedge X_m. 
\end{equation}
See that $[\mathcal{X},F]_{SN} $ is a multivector field of order $m-1$.
If $m=1$ that is we have a vector field, then
\begin{equation}
 [X,F]_{SN}=[X,F]=X(F)=\mathcal{L}_X F,
\end{equation}
where $X(F)$ is the directional derivative, that is, the Lie derivative of $F$ in the direction of $X$.

For $m=n=0$, one has two functions, say $F$ and $H$. The Schouten--Nijenhuis is then defined to be identically zero, that is
\begin{equation}
    [F,H]_{SN}:=0.
\end{equation}



 Consider a multivector field $\mathcal{X}$ of order $m \geq 1$ and a multivector field $\mathcal{Y} $ of order $n\geq 1$. The Schouten--Nijenhuis bracket is a supercommutative bracket satisfying
\begin{equation}\label{SN-commute}
    [\mathcal{Y},\mathcal{X}]_{SN} = -(-1)^{(m-1)(n-1)} [\mathcal{X},\mathcal{Y}]_{SN}.
\end{equation}
Notice that two multivector fields of even order commute, while any other combination of orders gives rise to anticommutation of such multivectors. Now, consider a multivector field $\mathcal{Z}$ of any order. Then, the Schouten--Nijenhuis satisfies the following identity:
\begin{equation}\label{SN-id}
    [\mathcal{X},\mathcal{Y} \wedge \mathcal{Z}]_{SN} = [\mathcal{X},\mathcal{Y}]_{SN} \wedge \mathcal{Z} + (-1)^{(m+1)n}\mathcal{Y} \wedge [\mathcal{X},\mathcal{Z}]_{SN}.
\end{equation}
For a function $F$, an identity similar to \eqref{SN-id} might be written as
\begin{equation}\label{X-FY}
     [\mathcal{X},F\mathcal{Y}]_{SN} = F[\mathcal{X},\mathcal{Y}]_{SN} + (-1)^{m+1} [\mathcal{X},F] _{SN}\wedge \mathcal{Y},
\end{equation}
where $\mathcal{X}$ is a multivector field of order $m$.

\subsection{Right Contraction Mapping (Interior Derivative)} \label{sec:intder} 

The (right) contraction of the multivector field $\mathcal{X}$ in \eqref{decompi}  with a one-form $\theta$ is defined as
\begin{equation}\label{right-cont}
    \iota_{\theta}\mathcal{X} := \sum_{i=1}^{m} (-1)^{i+m} \langle \theta, X_i \rangle X_1 \wedge \dots \wedge \widehat{X_i} \wedge \dots \wedge X_m.
\end{equation}
In terms of the (right) musical sharp map \eqref{sharp}, we can equivalently write
\begin{equation}
  \iota_{\theta} \mathcal{X} = \mathcal{X}^\sharp(\theta).
\end{equation}

By taking $\theta = dF$ as an exact one-form, we can express the Schouten--Nijenhuis bracket of a multivector field $\mathcal{X}$ with a smooth function $F$ as in \eqref{X-F-SN}:
\begin{equation} 
  [\mathcal{X}, F]_{SN} = (-1)^{m+1} \iota_{dF} \mathcal{X}.
\end{equation}
Using the right contraction mapping, the identity exhibited in \eqref{X-FY} becomes
\begin{equation}\label{X-FY-2}
    [\mathcal{X}, F \mathcal{Y}]_{SN} = F [\mathcal{X}, \mathcal{Y}]_{SN} + \iota_{dF} \mathcal{X} \wedge \mathcal{Y}.
\end{equation}

Let us show now the distribution of the right contraction mapping over the wedge product. We consider once more the multivectors $\mathcal{X}$ and $\mathcal{Y}$ given in \eqref{decompi}. Later we have 
\begin{equation}\label{iota-SN-m-n}
     \iota_{dF} (\mathcal{X} \wedge \mathcal{Y}) = (-1)^n\iota_{dF}\mathcal{X} \wedge \mathcal{Y} + \mathcal{X} \wedge \iota_{dF} \mathcal{Y} .
\end{equation}
For the case of $m=1$ and $n=2$ we have the distribution of the right conraction mapping as
\begin{equation}\label{iota-SN-2-2-X}
 \iota_{dF} (X \wedge \mathcal{Y}) =X(F) \mathcal{Y} + X\wedge  \iota_{dF} \mathcal{Y}.
\end{equation}
For the case of $m=n=2$ we have 
\begin{equation}\label{iota-SN-2-2-X-}
 \iota_{dF} (\mathcal{X} \wedge \mathcal{Y}) = \iota_{dF}\mathcal{X} \wedge \mathcal{Y} + \mathcal{X} \wedge \iota_{dF} \mathcal{Y} .
\end{equation}
For $\alpha\in\Gamma^k(\mathcal{Q})$, $\beta\in\Gamma^l(\mathcal{Q})$, we have
\begin{equation}
    \iota_{\alpha \wedge \beta} X = \iota_\alpha\iota_\beta X.
\end{equation}
In particular, for $F_1,...,F_k\in\mathcal{F}(\mathcal{Q})$,
\begin{equation}
     \iota_{dF_1 \wedge \dots  \wedge dF_k} X = \iota_{dF_1}\dots \iota_{dF_k} X .
\end{equation}

For a vector field $X$ and a bivector field $\mathcal{Y}$, the right contraction $\iota_{dF}$ distributes over the Schouten--Nijenhuis bracket as follows:
\begin{equation}
    \iota_{dF}[X,\mathcal{Y}]_{SN} = [\mathcal{Y},\iota_{dF}X]_{SN} + [X,\iota_{dF}\mathcal{Y}]_{SN}.
\end{equation}
For two bivector fields $\mathcal{X} = X_1 \wedge X_2$ and $\mathcal{Y} = Y_1 \wedge Y_2$, the right contraction $\iota_{dF}$ distributes over the Schouten--Nijenhuis bracket as follows:
\begin{equation}\label{iota-SN-2-2+}
    \iota_{dF}([\mathcal{X},\mathcal{Y}]_{SN}) = [\mathcal{Y},\iota_{dF}(\mathcal{X})]_{SN} + [\mathcal{X},\iota_{dF}(\mathcal{Y})]_{SN},
\end{equation}
for a function $F$. Notice that, if $\mathcal{X}=\mathcal{Y}$ of order  $2$ then we have  
\begin{equation}\label{iota-SN-2-2-X+}
 \iota_{dF}([\mathcal{X},\mathcal{X}]_{SN}) = 2[\mathcal{X},\iota_{dF}(\mathcal{X})]_{SN}. 
\end{equation}
In general, for $\mathcal{X}= X_1 \wedge \dots X_m$ and $\mathcal{Y}= Y_1 \wedge \dots Y_n$, the right contraction $\iota_{dF}$ distributes over the Schouten--Nijenhuis bracket as
\begin{equation}
    \iota_{dF}([\mathcal{X},\mathcal{Y}]_{SN}) = (-1)^n [\mathcal{Y},\iota_{dF}\mathcal{X}]_{SN} + [\mathcal{X},\iota_{dF}\mathcal{Y}]_{SN}.
\end{equation}

\bibliographystyle{abbrv}
\bibliography{references}

\end{document}